\documentclass[11pt]{amsart}

\usepackage[utf8]{inputenc}
\usepackage{newpxtext}
\usepackage[T1]{fontenc}
\usepackage{lipsum}

\usepackage{amsmath}
\usepackage{amssymb}
\usepackage{amsthm}
\usepackage{mathtools}
\usepackage{xfrac}
\usepackage{tikz-cd}
\usepackage{cases}
\usepackage{tabularx}
\usepackage{enumitem}
\usepackage{stmaryrd}
\usepackage[scr=boondox]{mathalpha}

\usepackage{graphicx}
\usepackage[export]{adjustbox}
\usepackage[dvipsnames, table]{xcolor} 
\definecolor{Blue}{RGB}{0, 25, 61}
\definecolor{Brown}{RGB}{102, 12, 0}
\definecolor{Green}{RGB}{0, 51, 0}
\definecolor{Gray}{RGB}{120,120,120}
\definecolor{Cream}{RGB}{245,240,200}
\definecolor{blue}{RGB}{0,120,140}
\definecolor{red}{RGB}{140,20,0}
\definecolor{green}{RGB}{0,90,0}

\usepackage[a4paper, margin=1.7cm]{geometry}

\theoremstyle{plain}
\newtheorem{theorem}{Theorem}[section]
\newtheorem{proposition}[theorem]{Proposition}
\newtheorem{lemma}[theorem]{Lemma}
\newtheorem{corollary}[theorem]{Corollary}

\theoremstyle{definition}
\newtheorem{definition}[theorem]{Definition}
\newtheorem{example}[theorem]{Example}

\theoremstyle{remark}
\newtheorem{remark}[theorem]{Remark}

\usepackage{hyperref}
\hypersetup{
    colorlinks=true,
    linkcolor=Blue,
    citecolor=Green,
    filecolor=magenta,
    urlcolor=cyan
}

\let\oldtocsection=\tocsection
\let\oldtocsubsection=\tocsubsection
\let\oldtocsubsubsection=\tocsubsubsection
\renewcommand{\tocsection}[2]{\hspace{0em}\oldtocsection{#1}{#2}}
\renewcommand{\tocsubsection}[2]{\hspace{1.5em}\oldtocsubsection{#1}{#2}}
\renewcommand{\tocsubsubsection}[2]{\hspace{3em}\oldtocsubsubsection{#1}{#2}}

\usepackage[url=false]{biblatex}
\title{Morita equivalence of Nijenhuis structures}
\author{Andrés I. Rodríguez}
\address{Instituto de Matemática Pura e Aplicada, Estrada Dona Castorina 110, Rio de Janeiro, 22460-320, Brazil.}
\email{andres.rodriguez@impa.br}
\date{}

\begin{document}

\begin{abstract}
We introduce Morita equivalence for Nijenhuis groupoids and for their infinitesimal counterparts, establishing a global-to-infinitesimal correspondence under the Lie functor. A special case is that of holomorphic Lie groupoids and algebroids. We use our framework to enhance the known Morita equivalences for quasi-symplectic groupoids and Dirac structures with compatible Nijenhuis structures. Finally, subject to certain conditions, we prove that the modular class of Poisson-Nijenhuis manifolds is invariant under Morita equivalence.
\end{abstract}

\maketitle
\tableofcontents
\section{Introduction}

Nijenhuis structures provide a convenient setting for the study of integrability of several geometric structures. A key example is found in complex geometry, where a holomorphic structure on a manifold is equivalently described by an almost-complex structure satisfying the Nijenhuis condition (i.e., whose Nijenhuis torsion vanishes) \cite{newlander1957complex,frolicher1959differentialgeometrie}. Nijenhuis structures are also important in the theory of integrable systems, serving as a tool for the systematic construction of hierarchies of functions in involution \cite{magri1978simple}. More recently, the study of \textit{Nijenhuis geometry} as a subject in its own right has emerged (see, e.g. \cite{bolsinov2022nijenhuis,konyaev2021nijenhuis,bolsinov2024nijenhuis,bolsinov2022applications}).

Nijenhuis structures on Lie groupoids naturally arise e.g. in the study of holomorphic Lie groupoids \cite{laurent2009integration} and in the context of integration of Poisson-Nijenhuis manifolds to  symplectic-Nijenhuis groupoids \cite{stienon-xu2007poisson,drummond2022lie,das2019poisson}. 
The primary goal of this paper is to establish an appropriate notion of Morita equivalence for Nijenhuis groupoids. This is achieved by introducing a general notion of Morita equivalence for multiplicative (1,1)-tensor fields which is compatible with the Nijenhuis condition, so it restricts to multiplicative Nijenhuis structures. This paves the way for the study of the Morita equivalence of more general \textit{quasi-Nijenhuis} groupoids, a topic to be addressed in a future work.

We also develop a parallel infinitesimal  theory for (1,1)-tensor fields on Lie algebroids, and in particular for Nijenhuis tensor fields. This leads to a notion of Morita equivalence for infinitesimal Nijenhuis structures, which is directly shown to correspond to the Morita equivalence of their integrations on source-simply connected Lie groupoids.

This framework finds a natural application in Poisson geometry. A compatibility between Poisson and Nijenhuis structures was established by Magri \cite{magri1978simple} in the context of bi-Hamiltonian integrable systems. This compatibility  was later generalized to \textit{Dirac--Nijenhuis structures} by Bursztyn, Drummond and Netto \cite{bursztyn-drummond-netto2022dirac}, and integrated to its global counterpart, \textit{quasi-symplectic--Nijenhuis groupoids}. We show that our definition of Morita equivalence of Nijenhuis structures preserves compatibility with these structures, at both the global and infinitesimal levels. This allows us to define a consistent notion of Morita equivalence for Dirac--Nijenhuis structures and quasi-symplectic--Nijenhuis groupoids. In particular, this encompasses the case of holomorphic presymplectic groupoids and holomorphic Dirac structures. 

We investigate the compatibility between Morita equivalence and the hierarchies of geometric structures naturally induced by Poisson--Nijenhuis and Dirac--Nijenhuis manifolds. Under suitable assumptions, we demonstrate that Morita equivalence persists throughout these hierarchies. As a cohomological consequence, we examine the intrinsic modular vector fields of Poisson--Nijenhuis manifolds introduced by Damianou and Fernandes \cite{damianou-fernandes2008integrable}. Recall that the standard modular class of a Poisson manifold, which measures the obstruction to the existence of an invariant measure under Hamiltonian flows, is preserved under Morita equivalence \cite{crainic2003differentiable}. In the enriched setting, the Poisson cohomology class of the intrinsic modular vector field serves as an additional modular class; its vanishing is the precise condition for the vector field to be bi-Hamiltonian. In Theorem \ref{thm:modular-vectors}, we establish that this modular class is also invariant under our introduced notion of Morita equivalence. Specifically, if two Poisson--Nijenhuis manifolds $(M_{1}, \pi_{1}, r_{1})$ and $(M_{2}, \pi_{2}, r_{2})$ are Morita equivalent via a non-degenerate infinitesimal bibundle, then their respective modular classes $[\mathscr{Y}_{r_{1}}]$ and $[\mathscr{Y}_{r_{2}}]$ correspond under the induced isomorphism between the first Poisson cohomology groups of $\big(M_{1}, \pi_{1}^{(1)}\big)$ and $\big(M_{2}, \pi_{2}^{(1)}\big)$.

\subsubsection*{Outline of the paper}
The paper is organized as follows:
\begin{itemize}
  \item Section \ref{sec:nijenhuis-groupoids}: Following a review of Nijenhuis structures, we introduce Nijenhuis groupoids and define Morita equivalence for multiplicative $(1,1)$-tensor fields (Definition \ref{def:morita-nijenhuis}). In Theorem \ref{thm:morita-equiv-relation}, we prove that this notion constitutes an equivalence relation.
  \item Section \ref{sec:nijenhuis-lie-algebroids}: Transitioning to the infinitesimal setting, we introduce Morita equivalence for infinitesimal Nijenhuis structures on Lie algebroids (Definition \ref{def:morita-equiv-algebroid}). Theorems \ref{thm:morita-nijenhuis-groupoid-to-algebroid} and \ref{thm:morita-nijenhuis-algebroid-to-groupoid} establish the respective differentiation and integration procedures, yielding a one-to-one correspondence between global and infinitesimal Morita equivalences (Corollary \ref{cor:corresp-morita-nijenhuis}).
  \item Section \ref{sec:symplectic-nijenhuis}: We study the compatibility of the Morita equivalence for Nijenhuis structures with both quasi-symplectic groupoids and twisted Dirac structures. Consequently, we extend the established notions of Morita equivalence for quasi-symplectic groupoids and Dirac structures to quasi-symplectic-Nijenhuis groupoids and Dirac-Nijenhuis structures, respectively. We establish the global-to-infinitesimal correspondence via differentiation (Theorem \ref{thm:symp-nij-morita-dirac-nij}) and integration (Theorem \ref{thm:morita-diracnij-to-sympnij}). Finally, in Theorem \ref{thm:modular-vectors}, we prove that the modular class of Poisson-Nijenhuis manifolds is invariant under Morita equivalence via non-degenerate infinitesimal bibundles.
\end{itemize}

\subsubsection*{Acknowledgements} I would like to thank Henrique Bursztyn for suggesting this project and for his invaluable guidance. I am also grateful to Alejandro Cabrera and Daniel Álvarez for their comments on this manuscript. This work was supported by a grant from CNPq.

\section{Nijenhuis Groupoids}

\label{sec:nijenhuis-groupoids}

In this section, we establish a notion of Morita equivalence for multiplicative (1,1)-tensor fields on Lie groupoids that is compatible with the Nijenhuis condition. We provide two equivalent characterizations (Proposition \ref{prop:morita-equiv-characterization}) for our definition of Morita equivalence (Definition \ref{def:morita-nijenhuis}). The main result of this section, Theorem \ref{thm:morita-equiv-relation}, establishes that this relation constitutes an equivalence relation for Lie groupoids endowed with multiplicative (1,1)-tensor fields, and in particular for Nijenhuis groupoids.
\subsection{Preliminaries on Nijenhuis structures}

 We denote by $\Omega^{k}(P,TP)=\Gamma(\wedge^{k}T^{*}P\otimes TP)$ the space of $(1,k)$-tensor fields (also referred to as {\em vector-valued $k$-forms}) on a manifold $P$. Recall that given a $(1,1)$-tensor field $N\in \Omega^{1}(P,TP)$, its \textbf{Nijenhuis torsion} \cite{nijenhuis1951xn} is the $(1,2)$-tensor field $\mathscr{N}_{N}\in \Omega^{2}(P,TP)$ given by
\begin{equation}
  \label{def:nijenhuis-torsion}
  \mathscr{N}_{N}\left(u,v \right) \coloneqq \left[ N(u),N(v) \right] -N\left( \left[ N(u) , v\right]+\left[ u,N(v) \right] - N([ u,v ])\right),
\end{equation}
for every $u,v\in \mathfrak{X}(P)$. The tensor field $N\in \Omega^{1}(P,TP)$ is said to be a \textbf{Nijenhuis tensor field} if its Nijenhuis torsion vanishes.

Given a smooth map $\varphi:P\to B$, we say that two $(1,k)$-tensor fields $N \in \Omega^{k}(P,TP)$ and $N_{B}\in \Omega^{k}(B,TB)$ are \textit{$\varphi$-related} (see e.g., \cite{kolar2013natural}) if for every $p\in P$ and $u_{1},\ldots,u_{k} \in T_{p}P$ we have
\[
\mathrm{d}\varphi \left( N_{p}(u_{1},\ldots,u_{k}) \right) = (N_{B})_{\varphi(p)}\left( \mathrm{d}\varphi(u_{1}), \ldots,\mathrm{d}\varphi(u_{k}) \right).
\]
In the case where $\varphi:P\to B$ is a surjective submersion, we say that $N\in \Omega^{k}(P,TP)$ is \textbf{$\varphi$-projectable} if it is $\varphi$-related to some $N_{B}\in \Omega^{k}(B,TB)$.

The Nijenhuis torsion is \textit{natural} in the sense that $\varphi$-related (1,1)-tensor fields have $\varphi$-related Nijenhuis torsions; in particular given a $\varphi$-projectable $(1,1)$-tensor field $N\in \Omega^{1}(P,TP)$, its Nijenhuis torsion $\mathscr{N}_{N} \in \Omega^{2}(P,TP)$ is also  $\varphi$-projectable \cite[Section 8.15]{kolar2013natural}. So if $N$ is a Nijenhuis $\varphi$-projectable $(1,1)$-tensor field $\varphi$-related to $N_{B}\in \Omega^{1}(B,TB)$, then the Nijenhuis torsion $\mathscr{N}_{N_{B}}$ of $N_{B}$ vanishes; hence $N_{B}$ is also a Nijenhuis tensor field.

We want to characterize the conditions under which a $(1,1)$-tensor field is projectable. Let $\varphi:P\to B$ be a surjective submersion. The differential of $\varphi$ induces a vector bundle map $\mathrm{d}\varphi:TP\to \varphi^{*}TB$ covering the identity. Then, we have the short exact sequence of vector bundles over $P$:
\begin{equation*}
  0 \xrightarrow{\quad \quad} \operatorname{ker}\mathrm{d}\varphi \xhookrightarrow{\quad  \quad} TP \xrightarrow{\quad \mathrm{d}\varphi \quad} \varphi^{*}TB \xrightarrow{\quad \quad } 0.
\end{equation*}
Consequently, the vector bundle $\varphi^{*}TB$ is isomorphic to $TP / \operatorname{ker}\mathrm{d}\varphi$. Recall that there is an action of the submersion groupoid $P\times_{B} P\rightrightarrows P$ on $\varphi^{*}TB$ given by $\left(p,q  \right)\cdot  \left(q,v  \right)=\left(p,v  \right)$, and it induces a unique action of $P\times_{B} P$ on $TP / \operatorname{ker} \mathrm{d}\varphi$ such that the isomorphism $\mathrm{d}\varphi$ is equivariant (see \cite[Section 1.6]{mackenzie2005general}).

\begin{lemma}
  \label{lemma:projectable-iff-preserves}
  Let $N\in \Omega^{1}(P,TP)$ be a $(1,1)$-tensor field and let $\varphi:P\to B$ be a surjective submersion. Then $N$ is $\varphi$-projectable if and only if $N\left( \operatorname{ker}\mathrm{d}\varphi \right) \subseteq \operatorname{ker}\mathrm{d}\varphi$ and the induced map $\bar{N}:TP/\operatorname{ker}\mathrm{d}\varphi\to TP/\operatorname{ker}\mathrm{d}\varphi$ is $(P\times_{B}P)$-equivariant.
\end{lemma}

\begin{proof}
  First, assume that $N$ is $\varphi$-projectable. There exists a tensor field $N_{B} \in \Omega^{1}(B,TB)$ such that $\mathrm{d}\varphi(N_{p}(u)) = (N_{B})_{\varphi(p)}(\mathrm{d}\varphi(u))$ for every $u \in T_{p}P$. If $u \in \operatorname{ker}\mathrm{d}\varphi$, it follows that
  \[
    \mathrm{d}\varphi(N(u)) = N_{B}(\mathrm{d}\varphi(u)) = N_{B}(0) = 0.
  \]
  Thus, $N(\operatorname{ker}\mathrm{d}\varphi)\subseteq \operatorname{ker}\mathrm{d}\varphi$, and $N$ descends to a well-defined map $\bar{N}$ on the quotient bundle $TP/\operatorname{ker}\mathrm{d}\varphi$.

  To establish equivariance, we use the fact that the isomorphism $\varphi^{*}TB \cong TP/\operatorname{ker}\mathrm{d}\varphi$ induces an action of $P\times_{B}P$ on the quotient bundle. Specifically, an arrow $(p,q) \in P\times_{B}P$ acts on $\bar{v} \in (TP/\operatorname{ker}\mathrm{d}\varphi)_{q}$ to yield $(p,q)\cdot \bar{v}= \bar{u} \in (TP/\operatorname{ker}\mathrm{d}\varphi)_{p}$ if and only if $\mathrm{d}\varphi(u) = \mathrm{d}\varphi(v)$. Using the definition of the projected tensor $N_{B}$, we compute:
  \[
    \bar{N}_{p}(\bar{u}) \cong (N_{B})_{\varphi(p)}(\mathrm{d}\varphi(u)) = (N_{B})_{\varphi(q)}(\mathrm{d}\varphi(v)) \cong \bar{N}_{q}(\bar{v}).
  \]
  Since $(p,q)$ acts as the identity on the base vectors, this equality implies $\bar{N}_{p}((p,q)\cdot \bar{v}) = (p,q) \cdot \bar{N}_{q}(\bar{v})$, proving that $\bar{N}$ is $(P\times_{B}P)$-equivariant.

  Conversely, suppose that $N(\operatorname{ker}\mathrm{d}\varphi)\subseteq \operatorname{ker}\mathrm{d}\varphi$ and that the induced map $\bar{N}$ is equivariant. We define a tensor field $N_{B} \in \Omega^{1}(B,TB)$ by setting $(N_{B})_{b}(w) \coloneqq \mathrm{d}\varphi(N_{p}(u))$, where $p\in \varphi^{-1}(b)$ and $u\in T_{p}P$ satisfies $\mathrm{d}\varphi(u)=w$. The value of $(N_{B})_{b}(w)$ is independent of the choice of the lift $u$ because $N$ preserves the vertical kernel. Furthermore, it is independent of the choice of $p$ because the induced map $\bar{N}$ is $(P\times_{B}P)$-equivariant. Thus, $N_{B}$ is well-defined and $N$ is $\varphi$-projectable.
\end{proof}
Recall that an action of a Lie groupoid $\mathscr{G}=\left(G_{1}\rightrightarrows G_{0}\right)$ on a manifold $P$ is \textit{principal} if there exists a surjective submersion $\varphi:P\to B$ such that the map $G_1 \times_{G_{0}} P \to P\times_{B}P$ given by $(g,p)\mapsto (g\cdot p,p)$ is a diffeomorphism. In this context, the action groupoid $\mathscr{G}\ltimes P$ is isomorphic to the submersion groupoid $P\times_{B}P$, yielding the following result.
\begin{corollary}
    \label{cor:projectable-iff-preserves-action}
    Suppose that a Lie groupoid $\mathscr{G}$ acts principally on $P$ with quotient map $\varphi:P\to B$. A $(1,1)$-tensor field $N\in \Omega^{1}(P,TP)$ is $\varphi$-projectable if and only if $N$ preserves the tangent spaces to the orbits of the action and the induced map $\bar{N}:TP /\operatorname{ker}\mathrm{d}\varphi\to TP /\operatorname{ker}\mathrm{d}\varphi$ is $\mathscr{G}$-equivariant.
\end{corollary}

In particular, a Nijenhuis tensor field $N\in \Omega^{1}(P,TP)$ satisfying the conditions of Corollary \ref{cor:projectable-iff-preserves-action} induces a Nijenhuis tensor field on the quotient.
\subsection{Multiplicative Nijenhuis structures}

\textit{Multiplicative tensor fields} on a Lie groupoid $\mathscr{G}=(G_{1}\rightrightarrows G_{0})$ are defined by means of the \textit{tangent groupoid} $T\mathscr{G}$. The tangent groupoid is obtained by applying the tangent functor to $\mathscr{G}$; specifically, the manifolds of arrows and objects are $TG_{1}$ and $TG_{0}$, respectively, and the structure maps are the differentials of the structure maps of $\mathscr{G}$. Moreover, by taking $k$-fold Whitney sums, we obtain a Lie groupoid $\oplus^{k}T\mathscr{G}=\left(\oplus^{k}TG_{1}\rightrightarrows \oplus^{k}TG_{0}\right)$, with structure maps defined componentwise.

A $(1,k)$-tensor field $K\in \Omega^{k}(G_{1},TG_{1})$ is said to be a $\textbf{multiplicative $(1,k)$-tensor field}$ if there exists a $(1,k)$-tensor field $r_{K}\in \Omega^k(G_{0},TG_{0})$ such that
\[
  \begin{tikzcd}
\oplus^{k}TG_1 \arrow[r, "K"] \arrow[d, shift left] \arrow[d, shift right] & TG_1 \arrow[d, shift right] \arrow[d, shift left] \\
\oplus^k TG_0 \arrow[r, "r_K"]                                             & TG_0                                             
\end{tikzcd}
\]
is a Lie groupoid morphism. We denote by $\Omega^{k}(\mathscr{G},T\mathscr{G})$ the space of multiplicative $(1,k)$-tensor fields on $\mathscr{G}$. For details on multiplicative tensor fields we refer to \cite{bursztyn-drummond2019multiplicative}.

\begin{definition}
  \label{def:nijenhuis-groupoid}
  A \textbf{Nijenhuis groupoid} is a pair $(\mathscr{G},N)$, where $\mathscr{G}=\left( G_{1}\rightrightarrows G_{0} \right) $ is a Lie groupoid and $N \in \Omega^{1}(\mathscr{G},T\mathscr{G})$ is a Nijenhuis tensor field; that is, $N$ is a multiplicative $(1,1)$-tensor field whose Nijenhuis torsion $\mathscr{N}_{N}$ vanishes.
\end{definition}
\begin{example}
  \label{eg:holomorphic-groupoid}
  A \textit{holomorphic groupoid} \cite{laurent2009integration,weinstein2007integration} is a groupoid object in the category of complex manifolds; that is, the spaces of arrows and objects are complex manifolds and all structural maps are holomorphic. Equivalently, such a groupoid is a (real) Lie groupoid with an integrable multiplicative almost-complex structure \cite[Proposition 3.16]{laurent2009integration}. Thus, a holomorphic groupoid is a Nijenhuis groupoid $(\mathscr{G},I_{\mathscr{G}})$, where $I_{\mathscr{G}}\in \Omega^{1}(\mathscr{G},T\mathscr{G})$ satisfies $I_{\mathscr{G}}^{2}=-\mathrm{id}_{T\mathscr{G}}$.
\end{example}
\begin{example}
  \label{eg:SN-groupoid}
  Nijenhuis groupoids also arise naturally in the integration of Poisson-Nijenhuis manifolds (see Example \ref{eg:poisson-nijenhuis}). Stiénon and Xu \cite{stienon-xu2007poisson} proved that if the underlying Poisson structure of a Poisson-Nijenhuis manifold $(M,\pi,r)$ integrates to a symplectic groupoid $(\mathscr{G}, \omega)$, then $r$ lifts to a unique multiplicative Nijenhuis tensor field $N\in \Omega^{1}(\mathscr{G},T\mathscr{G})$. The compatibility between $\pi$ and $r$ ensures that $N$ and $\omega$ are compatible, endowing $\mathscr{G}$ with the structure of a \textit{symplectic-Nijenhuis groupoid}.
\end{example}
The following lemma establishes conditions under which a $(1,1)$-tensor field on a manifold $P$ is projectable when a Nijenhuis groupoid acts on $P$.
\begin{lemma}
    \label{lem:projectable_tensor_principal_bundle}
    Let $\mathscr{G} = (G_1 \rightrightarrows G_0)$ be a Lie groupoid and $\pi: P \to B$ a principal $\mathscr{G}$-bundle equipped with an action $\mathscr{A}: G_1 \times_{G_0} P \to P$. Assume that $\mathscr{G}$ is equipped with a multiplicative $(1,1)$-tensor field $N \in \Omega^1(\mathscr{G}, T\mathscr{G})$ and that $P$ carries a $(1,1)$-tensor field $J \in \Omega^1(P, TP)$. If the action intertwines the tensor fields, i.e.,
    \begin{equation}
        \label{eq:intertwine_action_lemma}
        J \circ \mathrm{d}\mathscr{A} = \mathrm{d}\mathscr{A} \circ (N \times J),
    \end{equation}
    then $J$ projects to a unique $(1,1)$-tensor field $J_B \in \Omega^1(B, TB)$ on the base. Furthermore, if $N$ and $J$ are Nijenhuis tensor fields, then $J_B$ is also a Nijenhuis tensor field.
\end{lemma}
\begin{proof}
We apply Corollary \ref{cor:projectable-iff-preserves-action}. First, we verify that $J$ preserves the vertical bundle $\operatorname{ker}\mathrm{d}\pi$, which coincides with the tangent spaces to the orbits of the $\mathscr{G}$-action. Let $\mathscr{O}$ be the orbit through a point $p\in P$, and write $x=\mu(p)$. For any $v\in T_{p}\mathscr{O}$, there exists $\xi\in \Gamma(A_{\mathscr{G}})$, where $A_{\mathscr{G}}$ is the Lie algebroid of $\mathscr{G}$, such that $v = \mathrm{d}\mathscr{A}_{(\mathbf{1}_{x},p)}\left( \overrightarrow{\xi}_{x},0_{p} \right)$. Applying condition \eqref{eq:intertwine_action_lemma} and the linearity of $J$, we obtain:
    \[
        J(v) = J\left(\mathrm{d}\mathscr{A}\left( \overrightarrow{\xi}_{x},0_{p} \right) \right) = \mathrm{d}\mathscr{A} \left( N\left( \overrightarrow{\xi}_{x} \right),J\left( 0_{p} \right) \right) = \mathrm{d}\mathscr{A} \left( N\left( \overrightarrow{\xi}_{x} \right),0_{p} \right).
    \]
    Since $N$ is a morphism of the tangent groupoid, it preserves the kernel of the source map differential. Consequently, there exists $\zeta\in \Gamma\left( A_{\mathscr{G}} \right)$ such that $N\left( \overrightarrow{\xi} \right) = \overrightarrow{\zeta}$. Therefore, $J(v)$ takes the form $J(v) = \mathrm{d}\mathscr{A}\left( \overrightarrow{\zeta}_{x}, 0_{p} \right)$, which implies $J(v) \in T_{p}\mathscr{O}$.

To establish that the induced map $\bar{J} \colon TP/\operatorname{ker}\mathrm{d}\pi \to TP/\operatorname{ker}\mathrm{d}\pi$ is $\mathscr{G}$-equivariant, we must show that $g \cdot \bar{J}(\bar{v}) = \bar{J}(g \cdot \bar{v})$ for every $v \in T_{p}P$, and $(g,p)\in G_{1}\times_{G_{0}}P$. Under the identification $TP/\operatorname{ker}\mathrm{d}\pi \cong \varphi^{*}TB$, we may write $\bar{v} \cong (\varphi(p),\mathrm{d}\pi(v))$. Thus, $w \in T_{g\cdot p}P$ satisfies $g \cdot \bar{v} = \bar{w}$ if and only if $\mathrm{d}\pi(v) = \mathrm{d}\pi(w)$. We proceed as before: if $g \cdot \bar{v} = \bar{w}$, then there exists $\xi \in \Gamma(A_{\mathscr{G}})$ such that $w = \mathrm{d}\mathscr{A}\left( \overrightarrow{\xi}_{g},v \right)$. Applying condition \eqref{eq:intertwine_action_lemma}, we obtain
    \[
      J_{g\cdot p} (w) = J_{g\cdot p} \left( \mathrm{d}\mathscr{A}\left( \overrightarrow{\xi}_{g},v \right) \right) = \mathrm{d}\mathscr{A}\left( N\left( \overrightarrow{\xi}_{g} \right), J_{p}(v) \right) = \mathrm{d}\mathscr{A}\left( \overrightarrow{\zeta}_{g}, J_{p}(v) \right)
    \]
for some $\zeta \in \Gamma(A_{\mathscr{G}})$. Therefore, $\bar{J}(\bar{w}) = g \cdot \overline{J(v)}$, which concludes the proof.

Finally, assume that $N$ and $J$ are Nijenhuis tensor fields. Since $J$ projects to $J_B$, the Nijenhuis torsion $N_{J}$ is $\pi$-related to $N_{J_B}$. The vanishing of $N_J$ implies the vanishing of $N_{J_B}$, proving that $J_B$ is a Nijenhuis tensor field.
\end{proof}

\subsection{Morita equivalence of Nijenhuis groupoids}

 In order to introduce Morita equivalence of Nijenhuis groupoids, we recall two equivalent notions of Morita equivalence for Lie groupoids following \cite[Theorem 2.2]{behrend2011differentiable} (see also \cite[Section 2.6]{moerdijk-mrcun2005lie}).

A first formulation of Morita equivalence is given in terms of \textit{principal bibundles}. A \textbf{Morita equivalence} between Lie groupoids $\mathscr{G}=\left( G_{1}\rightrightarrows G_{0} \right)$ and $\mathscr{H}=\left( H_{1}\rightrightarrows H_{0} \right)$ consists of a manifold $P$ fitting into a diagram
\[
    \begin{tikzcd}
       G_1 \arrow[d, shift right] \arrow[d, shift left] & P \arrow[ld, "\mu_1"'] \arrow[rd, "\mu_2"] & H_1 \arrow[d, shift right] \arrow[d, shift left] \\ 
       G_0                                              &                                            & H_0                                             
    \end{tikzcd}
\]
together with two commuting actions $\mathscr{A}_{1}:G_{1}\times_{(\mathbf{s}_{\mathscr{G}},\mu_{1})}P\to P$ and $\mathscr{A}_{2}:P\times_{(\mu_{2},\mathbf{t}_{\mathscr{H}})}H_1\to P$. We require the moment maps $\mu_{1}$ and $\mu_{2}$ to be invariant with respect to the actions of $\mathscr{H}$ and $\mathscr{G}$, respectively. Furthermore, the actions must be principal; that is, $\mathscr{G}$ acts freely and transitively on the fibers of $\mu_{2}$, and $\mathscr{H}$ acts freely and transitively on the fibers of $\mu_{1}$. In this setting, we call $P$ a \textit{Morita $(\mathscr{G},\mathscr{H})$-bibundle}.

The other formulation of Morita equivalence requires the notion of \textit{Morita morphism}. A Morita morphism is a Lie groupoid morphism $ \varphi :\mathscr{H}\to \mathscr{G} $ such that $ \varphi:H_{0}\to G_{0} $ is a surjective submersion and 
\[ 
\begin{tikzcd}
  H_{1} \arrow[d, "{(\mathbf{s}_{\mathscr{H}},\mathbf{t}_{\mathscr{H}})}"'] \arrow[r, "\varphi "] & G_{1} \arrow[d, "{(\mathbf{s}_{\mathscr{G}},\mathbf{t}_{\mathscr{G}})}"]\\
  H_{0}\times H_{0} \arrow[r, "\varphi\times \varphi"] & G_{0} \times G_{0}
\end{tikzcd}
\]
is a cartesian (pullback) diagram, i.e., $ H_{1} \cong (\varphi\times \varphi)^{*}G_{1}$. Two Lie groupoids $\mathscr{G}$ and $\mathscr{H}$ are Morita equivalent if and only if they are connected by a span of Morita morphisms $\mathscr{G}\gets  \mathscr{P}\to  \mathscr{H}$. The diagram $\mathscr{G}\gets \mathscr{P} \to  \mathscr{H}$ is also referred as a Morita equivalence (in terms of Morita morphisms) between $\mathscr{G}$ and $\mathscr{H}$. 

The following proposition establishes equivalent characterizations for a relation between multiplicative (1,1)-tensor fields, which form the basis for the definition of Morita equivalence of Nijenhuis groupoids introduced in Definition \ref{def:morita-nijenhuis}.
\begin{proposition}
  \label{prop:morita-equiv-characterization} 
  Let $\mathscr{G}=\left( G_{1}\rightrightarrows G_{0} \right)$ and $\mathscr{H}=\left( H_{1}\rightrightarrows H_{0} \right)$ be two Lie groupoids, and let $N_{1}\in \Omega^{1}\left( \mathscr{G},T\mathscr{G} \right)$ and $N_{2}\in \Omega^{1}\left( \mathscr{H},T\mathscr{H} \right)$ be multiplicative $(1,1)$-tensor fields. The following properties are equivalent. 
\begin{itemize}
  \item[a.] There is a diagram
\begin{equation}
    \label{eq:morita-nijenhuis-diagram}
  \left(\mathscr{G},N_{1}\right) \xleftarrow{\quad \varphi_{1}\quad}  \left(\mathscr{P},K\right) \xrightarrow{\quad \varphi_{2}\quad } \left( \mathscr{H},N_{2}\right),
\end{equation}
where $\mathscr{G}\gets \mathscr{P}\to \mathscr{H}$ is a span of Morita morphisms and $K\in \Omega^{1}(\mathscr{P},T\mathscr{P})$ is  a multiplicative $(1,1)$-tensor field such that $K$ is $\varphi_{1}$-related to $N_{1}$ and $\varphi_{2}$-related to $N_{2}$. Explicitly, these conditions are given by
\begin{equation}
  \label{eq:morita-nijenhuis-equations}
  \mathrm{d}\varphi_{1}\circ K = N_{1}\circ\mathrm{d}\varphi_{1}\qquad \text{and}\qquad \mathrm{d}\varphi_{2}\circ K =N_{2} \circ \mathrm{d}\varphi_{2}.
\end{equation}
  \item[b.] There is a Morita bibundle endowed with a $(1,1)$-tensor field $J\in \Omega^{1}(P,TP)$,
\begin{equation}
  \label{eq:morita-bibundle-tensor fields}
\begin{tikzcd}
  \left(G_{1},N_{1}\right) \arrow[d, shift right] \arrow[d, shift left] & \left( P,J\right) \arrow[ld, "\mu_{1}" '] \arrow[rd, "\mu_{2}"] & \left(H_{1},N_{2}\right) \arrow[d, shift right] \arrow[d, shift left] \\
  \left(G_{0},r_{1}\right)        &           & \left(H_{0},r_{2}\right),
\end{tikzcd}
\end{equation}
 satisfying
\begin{equation}
  \label{eq:morita-tensor-restriction-2}
  N_{1}\times J\times N_{2} \left( TG_{1}\times_{TG_{0}}TP \times_{TH_{0}}TH_{1} \right)\subseteq TG_{1}\times_{TG_{0}}TP\times_{TH_{0}}TH_{1}
\end{equation}
and the invariance condition
\begin{equation}
  \label{eq:morita-tensor-action-2}
  J\circ \mathrm{d}\mathscr{A} = \mathrm{d}\mathscr{A}\circ\left( N_{1}\times J\times N_{2} \right),
\end{equation}
where $\mathscr{A}:G_{1}\times_{G_{0}}P\times_{H_{0}}H_{1}\to P$ is defined by $(g,p,h)\mapsto g\cdot p\cdot h$. 
\end{itemize}
\end{proposition}

\begin{proof}
 Let $(\mathscr{G},N_{1}) \xleftarrow{\varphi_1} (\mathscr{P},K) \xrightarrow{\varphi_2} (\mathscr{H},N_{2})$ be a span of Morita morphisms equipped with multiplicative $(1,1)$-tensor fields as shown in Diagram \eqref{eq:morita-nijenhuis-diagram}. Assume that these fields satisfy the relations \eqref{eq:morita-nijenhuis-equations}. Let $r_1 \in \Omega^1(G_0, TG_0)$, $r_K \in \Omega^1(P_0, TP_0)$ and $r_2 \in \Omega^1(H_0, TH_0)$ denote the $(1,1)$-tensor fields induced on the bases by $N_{1}$, $K$ and $N_{2}$, respectively. 

 Focusing on the left Morita morphism $(\mathscr{G},N_{1})\gets (\mathscr{P},K)$, let $Q_{1}$ be the fiber product $Q_{1}\coloneqq G_{1}\times_{(\mathbf{s}_{\mathscr{G}},\varphi_{1})}P_{0}$. Identify the tangent bundle of $Q_{1}$ with the fiber product of the tangent bundles: $TQ_{1}\cong TG_{1}\times_{TG_{0}}TP_{0}\subseteq TG_{1}\times  TP_{0}$. Since $N_{1}:T\mathscr{G}\to T\mathscr{G}$ is a groupoid morphism and $\mathrm{d}\varphi_{1} \circ r_{K} = r_{1}\circ\mathrm{d}\varphi_{1}$, the product $N_{1}\times r_{K}$ restricts to the subbundle $TQ_{1} \subseteq TG_{1}\times TP_{0}$. To verify this, consider $(u,v)\in TG_{1}\times TP_{0}$ satisfying $\mathrm{d}\mathbf{s}_{\mathscr{G}}(u)= \mathrm{d}\varphi_{1}(v)$. We compute
\[
  \mathrm{d}\mathbf{s}_{\mathscr{G}}\left( N_{1}( u)  \right) = r_{1}\left( \mathrm{d}\mathbf{s}_{\mathscr{G}}(u) \right) = r_{1}\left( \mathrm{d}\varphi_{1}(v) \right) =   \mathrm{d}\varphi_{1}\left( r_{K}(v) \right).
\]
Consequently, $ \left(N_{1}\times r_{K}\right) \left( TQ_{1} \right) \subseteq TQ_{1}$, and this restriction to $Q_{1}$ defines a $(1,1)$-tensor field $N_{1}\times r_{K}\in \Omega^{1}(Q_{1},TQ_{1})$.

This construction yields a principal bibundle equipped with $(1,1)$-tensor fields, as summarized in the following diagram:
\[
  \begin{tikzcd}
    {\left(G_1,N_{1}\right)} \arrow[d, shift right] \arrow[d, shift left] & {\left(Q_{1} , N_{1}\times r_{K}\right)} \arrow[ld, "\mathbf{t}_{\mathscr{G}} \circ \mathrm{pr}_1"] \arrow[rd, "\mathrm{pr}_2"'] & {\left(P_1,K\right)} \arrow[d, shift right] \arrow[d, shift left] \\
    {\left(G_0,r_{1}\right)}                                              &                                                                                                                        & \left(P_0,r_{K}\right).
  \end{tikzcd}
\]
The left and right actions on $Q_{1}$, $\mathscr{A}_{1}:G_{1}\times_{G_{0}}Q_{1}\to Q_{1}$ and $\widetilde{\mathscr{A}}_{1}:Q_{1}\times_{P_{0}}P_{1}\to Q_{1}$, are defined respectively by
\begin{equation}
    \label{eq:left-action}
  g \cdot  \left( g',p \right) = \left( g g', p \right)\qquad \text{and}\qquad \left( g,p' \right)\cdot p = \left( g \, \varphi_{1}(p), \mathbf{s}_{\mathscr{P}}(p) \right).
\end{equation}
Let ${\bullet}$ denote the multiplication map on the tangent groupoid $T\mathscr{G}$. Since $N_{1}$ is a Lie groupoid morphism from $T\mathscr{G}$ to itself, we have that:
\begin{align*}
    \mathrm{d}\mathscr{A}_{1} \left( N_{1} (u),  N_{1}\times r_{K}(u',v) \right) & = \left( N_{1}(u)\bullet N_{1}(u'), r_{K}(v)  \right)\\
                                                                              & = \left( N_{1}\left( u\bullet u' \right), r_{K}(v) \right)\\
                                                                              & = \left( N_{1}\times r_{K}\right)\left( \mathrm{d}\mathscr{A}_{1}\left( u,(u',v) \right) \right).
\end{align*}
Furthermore, since $K$ is a Lie groupoid morphism and is $\varphi_{1}$-related to $N_{1}$, we have
\begin{align*}
    \mathrm{d}\widetilde{\mathscr{A}}_{1} \left( (N_{1}\times r_{K}(u,v')),K(v) \right)  & = \left( N_{1}(u)\bullet \mathrm{d}\varphi_{1}\left( K(v) \right),\mathrm{d}\mathbf{s}_{\mathscr{P}}\left( K(v) \right) \right)\\
                                                    & = \left( N_{1}(u)\bullet N_{1} (\mathrm{d}\varphi_{1}(v)) , r_{K}\left( \mathrm{d}\mathbf{s}_{\mathscr{P}}(v) \right) \right)\\
                                                    & =  \left( N_{1} \times r_{K}\right) \left( u \bullet\mathrm{d}\varphi_{1} (v) , \mathrm{d}\mathbf{s}_{\mathscr{P}}(v)  \right)\\
                                                    & = \left( N_{1}\times r_{K}\right) \left( \mathrm{d}\widetilde{\mathscr{A}}_{1} \left( (u,v'),v \right) \right).
\end{align*}
In other words,
\begin{align}
  &\mathrm{d}\mathscr{A}_{1}\circ N_{1}\times (N_{1}\times r_{K})= \left( N_{1}\times r_{K} \right)\circ \mathrm{d}\mathscr{A}_{1} \qquad \text{and}\\
  &\mathrm{d}\widetilde{\mathscr{A}}_{1}\circ\left( N_{1}\times r_{K} \right)\times K = \left( N_{1}\times r_{K} \right)\circ\mathrm{d}\widetilde{\mathscr{A}}_{1}. \label{eq:action_red1}
\end{align}
Similarly, consider the Morita morphism $(\mathscr{P},K)\to  (\mathscr{H},N_{2})$ on the right side of the Morita span and define $Q_{2}\coloneqq P_{0}\times_{(\varphi_{2},\mathbf{t}_{\mathscr{H}})}H_{1}$. The analogous construction yields a principal bibundle equipped with $(1,1)$-tensor fields:
\[
    \left( P_{0},r_{K} \right)\xleftarrow{\quad \mathrm{pr}_{1}\quad } \left( Q_{2} , r_{K}\times N_{2}\right) \xrightarrow{\quad \mathbf{s}_{\mathscr{H}}\circ\mathrm{pr}_{2}\quad} \left( H_{0},r_{2}\right).
\]
Here, the left action $\widetilde{\mathscr{A}}_{2}:P_{1}\times_{P_{0}} Q_{2}\to Q_{2}$ and the right action $\mathscr{A}_{2}:Q_{2}\times_{H_{0}}H_{1}\to Q_{2}$ are defined, respectively, by
\begin{equation}
    \label{eq:right-action}
    p\cdot (p',h) = \left( \mathbf{t}_{\mathscr{P}}(p), \varphi_{2}(p) \, h  \right) \qquad \text{and}\qquad (p,h')\cdot h  = (p,h' h).
\end{equation}
We verify that these actions satisfy analogous properties to the first case; that is,
\begin{align}
    &\mathrm{d}\widetilde{\mathscr{A}}_{2}\circ\left( K\times \left( r_{K}\times N_{2} \right) \right)= \left( r_{K}\times N_{2} \right)\circ \mathrm{d}\widetilde{\mathscr{A}}_{2} \qquad \text{and} \\
    &\mathrm{d}\mathscr{A}_{2}\circ \left(\left( r_{K}\times N_{2} \right)\times N_{2}\right) = \left( r_{K}\times N_{2} \right)\circ\mathrm{d}\mathscr{A}_{2}. \label{eq:action_red2}
\end{align}
We equip the fiber product $Q \coloneqq Q_{1}\times _{P_{0}}Q_{2}$ with a $\mathscr{P}$-action $\widetilde{\mathscr{A}}:P_{1}\times_{P_{0}}Q\to Q$ by combining the opposite of $\widetilde{\mathscr{A}}_{1}$ with $\widetilde{\mathscr{A}}_{2}$:
\[
  p\cdot \left( (g,p'),(p'',h) \right) \coloneqq \left( (g,p')\cdot p^{-1}, p\cdot (p'',h) \right).
\]
We define the $(1,1)$-tensor field $\widetilde{J}\in \Omega^{1}(Q,TQ)$ by $\widetilde{J} \coloneqq N_{1} \times r_{K} \times r_{K} \times N_{2}$. By Equations \eqref{eq:action_red1} and \eqref{eq:action_red2}, we obtain the relation
\[
    \mathrm{d}\widetilde{\mathscr{A}} \circ ( K \times \widetilde{J} ) = \widetilde{J} \circ \mathrm{d}\widetilde{\mathscr{A}}.
\]
Therefore, by Lemma \ref{lem:projectable_tensor_principal_bundle}, $\widetilde{J}$ projects to a well-defined $(1,1)$-tensor field $J \in \Omega^{1}(P,TP)$ on the quotient manifold $P \coloneqq Q/\mathscr{P}$.

The pair $(P,J)$ defines the principal bibundle \eqref{eq:morita-bibundle-tensor fields}, whose moment maps are induced by the natural projections $\mathbf{t}_{\mathscr{G}} \circ \mathrm{pr}_1\circ \mathrm{pr}_{Q_{1}}$ and $\mathbf{s}_{\mathscr{H}} \circ \mathrm{pr}_2\circ \mathrm{pr}_{Q_{2}}$. The actions of $\mathscr{G}$ and $\mathscr{H}$ are induced on the quotient by the left and right actions \eqref{eq:left-action} and \eqref{eq:right-action}. Since these actions correspond to left and right multiplications, they satisfy \eqref{eq:morita-tensor-restriction-2} and \eqref{eq:morita-tensor-action-2}.

Conversely, assume there exists a Morita bibundle \eqref{eq:morita-bibundle-tensor fields} satisfying \eqref{eq:morita-tensor-restriction-2} and \eqref{eq:morita-tensor-action-2}. Let $\mathscr{P}=\left( P_{1}\rightrightarrows P \right)$ be the Lie groupoid where $P_{1}\coloneqq G_{1}\times_{(\mathbf{s}_{\mathscr{G}},\mu_{1})}P\times_{(\mu_{2},\mathbf{t}_{\mathscr{H}})}H_{1}$. The source and target maps are $(g,p,h)\mapsto p$ and $(g,p,h)\mapsto g\cdot p\cdot h$, respectively, while the multiplication is $((g,p,h),(g',p',h'))\mapsto (gg',p',hh')$. Then the diagram
\[
  \begin{tikzcd}
G_1 \arrow[d, shift left] \arrow[d, shift right] &  & {G_{1}\times_{(\mathbf{s}_{\mathscr{G}},\mu_{1})}P\times_{(\mu_{2},\mathbf{t}_{\mathscr{H}})}H_{1}} \arrow[d, shift left] \arrow[d, shift right] \arrow[rr, "\mathbf{i}_{\mathscr{H}}\circ \mathrm{pr}_3"] \arrow[ll, "\mathrm{pr}_1"'] &  & H_1 \arrow[d, shift left] \arrow[d, shift right] \\
G_0                                              &  & P \arrow[ll, "\mu_1"'] \arrow[rr, "\mu_2"]                                                                                                                                                                        &  & H_0                                             
\end{tikzcd}
\]
is a span of Lie groupoid morphisms, which are Morita morphisms by construction (due to the free and transitive actions on the fibers). Define the $(1,1)$-tensor field $K\in \Omega^{1}(\mathscr{P},T\mathscr{P})$ given by $K=N_{1}\times J\times N_{2}$, where $J:TP\to TP$ is the map on the base. The tensor field $K$ indeed restricts to the fiber product due to Equation \eqref{eq:morita-tensor-restriction-2} and is a Lie groupoid morphism due to Equation \eqref{eq:morita-tensor-action-2}. Observe that every tangent vector $v\in TP$ can be written as the image of $\mathrm{d}\mathscr{A}$ by considering curves $\gamma(\tau)=u_{1}(\tau)\cdot p(\tau)\cdot u_{2}(\tau)$, where $u_{1}(\tau)=\mathbf{u}_{\mathscr{G}}(\mu_{1}(p)) $ and $u_{2}(\tau)=\mathbf{u}_{\mathscr{G}}(\mu_{2}(p))$. Differentiation at $\tau=0$ yields
\begin{align*}
    \mathrm{d}\mu_{1}\left( J(v) \right) & = \mathrm{d}\mu_{1}\left( J\left(\mathrm{d}\mathscr{A}(u_{1}'(0),v,u_{2}'(0)) \right)\right)\\
                                         & = \mathrm{d}\mu_{1}\left( \mathrm{d}\mathscr{A}_{1}\left( N_{1}(u_{1}'(0)),J(v) \right) \right)\\
                                         & = \mathrm{d}\mathbf{t}_{\mathscr{G}}\left( N_{1}(u_{1}'(0)) \right)\\
                                         & = r_{1}\left( \mathrm{d}\mathbf{t}_{\mathscr{G}}(u_{1}'(0)) \right). 
\end{align*}
The analogous equalities hold for $\mathrm{d}\mu_{2}$. From the condition $\mu_{1}(g\cdot p)=\mathbf{t}_{\mathscr{G}}(g)$ on the action, we obtain:
\begin{equation}
  \label{eq:tensor fields-related-moment-map}
  \mathrm{d}\mu_{1}\circ J = r_{1}\circ \mathrm{d}\mu_{1} \qquad \text{and}\qquad \mathrm{d}\mu_{2}\circ J= r_{2}\circ \mathrm{d}\mu_{2}.
\end{equation}
The arrow components of the Morita morphisms are canonical projections. Thus, Equations \eqref{eq:tensor fields-related-moment-map} imply that $N_{1}$ and $N_{2}$ are related to $K$ through the Morita morphisms, as required by Equation \eqref{eq:morita-nijenhuis-equations}. Consequently, we have a span of Morita morphisms $(\mathscr{G},N_{1})\gets (\mathscr{P},K)\to (\mathscr{H},N_{2})$ as in Diagram \eqref{eq:morita-nijenhuis-diagram}.
\end{proof}

This proof consists in constructing a $(1,1)$-tensor $J\in \Omega^{1}(P,TP)$ on a Morita bibundle from a multiplicative $(1,1)$-tensor field $K\in \Omega^{1}(\mathscr{P},T\mathscr{P})$ on the span of Morita morphisms, and vice versa. From this construction, it follows that $J$ has vanishing Nijenhuis torsion if and only if $K$ does. In the following definition, we refer to either of these as the \textit{intertwining $(1,1)$-tensor field}.
\begin{definition}
  \label{def:morita-nijenhuis}
  Two multiplicative $(1,1)$-tensor fields $N_{1} \in \Omega^{1}(\mathscr{G},T\mathscr{G})$ and $N_{2}\in \Omega^{1}(\mathscr{H},T\mathscr{H})$ are \textbf{Morita equivalent} if they satisfy the conditions of Proposition \ref{prop:morita-equiv-characterization}. Furthermore, if $N_{1}$ and $N_{2}$ are Nijenhuis tensor fields, we say that $(\mathscr{G},N_{1})$ and $(\mathscr{H},N_{2})$ are \textbf{Morita equivalent Nijenhuis groupoids} if the intertwining $(1,1)$-tensor field has vanishing Nijenhuis torsion.
\end{definition}

\begin{remark}
    \label{rmk:morita-nijenhuis-separate}
    Using the individual actions $\mathscr{A}_{1}(g,p)=g\cdot p$ and $\mathscr{A}_{2}(p,h)=p \cdot h$ on the Morita bibundle, we can equivalently express the conditions \eqref{eq:morita-tensor-restriction-2} and \eqref{eq:morita-tensor-action-2}  as 
    \begin{equation}
        \label{eq:morita-tensor-restriction-3}
        N_{1}\times J \left( TG_{1}\times_{TG_{0}}TP  \right)\subseteq TG_{1}\times_{TG_{0}}TP, \qquad J\times N_{2} \left( TP \times_{TH_{0}}TH_{1} \right)\subseteq TP\times_{TH_{0}}TH_{1},
    \end{equation}
    and
    \begin{equation}
        \label{eq:morita-tensor-action-3}
        J\circ \mathrm{d}\mathscr{A}_{1} = \mathrm{d}\mathscr{A}_{1}\circ\left( N_{1}\times J \right), \qquad J\circ \mathrm{d}\mathscr{A}_{2} = \mathrm{d}\mathscr{A}_{2}\circ\left( J\times N_{2} \right).
    \end{equation}
\end{remark}

\begin{example}
  \label{eg:holomorphic-groupoids-morita}
  Two holomorphic groupoids $(\mathscr{G},I_{\mathscr{G}})$ and $(\mathscr{H},I_{\mathscr{H}})$ are Morita equivalent if there exists a span of holomorphic Morita morphisms $\mathscr{G} \xleftarrow{\varphi_{1}} \mathscr{P} \xrightarrow{\varphi_{2}} \mathscr{H}.$ By Proposition \ref{prop:morita-equiv-characterization}, this corresponds to a Morita bibundle $P$ equipped with a complex structure such that the actions of both $\mathscr{G}$ and $\mathscr{H}$ are holomorphic. Moreover, the moment maps $\mu_{1}: P \to G_{0}$ and $\mu_{2}: P \to H_{0}$ are holomorphic as a consequence of \eqref{eq:tensor fields-related-moment-map}. We call this formulation a \textit{holomorphic Morita bibundle}.
\end{example}

\subsection{Morita equivalence is an equivalence relation.}
We show that Morita equivalence of Nijenhuis groupoids is an equivalence relation, extending the established result for the underlying Lie groupoids. Recall that reflexivity is realized by $G_{0} \xleftarrow{\ \mathbf{s}\ } G_{1}\xrightarrow{\ \mathbf{t}\ }G_{0}$, while symmetry follows by considering the opposite action of the groupoids. For transitivity, the composition of two Morita bibundles $G_{0}\gets P\to H_{0}$ and $H_{0}\gets Q\to G_{0}'$ (between Lie groupoids $\mathscr{G}$, $\mathscr{H}$, and $\mathscr{G}'$) is defined by 
\begin{equation}
    \label{eq:quotient}
  P \diamond Q \coloneqq \frac{P \times_{H_{0}} Q}{\mathscr{H}},
\end{equation}
where $\mathscr{H}$ acts on the fiber product $P\times_{H_{0}}Q$ via the \textit{diagonal action}
\[
  h \cdot (p,q) \coloneqq \left( p\cdot h^{-1}, h\cdot q \right).
\]
\begin{theorem}
  \label{thm:morita-equiv-relation}
  Morita equivalence of multiplicative $(1,1)$-tensor fields is an equivalence relation. Moreover, Morita equivalence of Nijenhuis groupoids is also an equivalence relation.
\end{theorem}
 \begin{proof}
  For reflexivity, let $\mathscr{G}$ be a Lie groupoid equipped with a multiplicative $(1,1)$-tensor field $N \in \Omega^{1}(\mathscr{G},T\mathscr{G})$. Because the actions of the principal bibundle $G_{0}\xleftarrow{\ \mathbf{t}\ } G_{1} \xrightarrow{\ \mathbf{s}\ }G_{0}$ correspond to left and right multiplication, and because $N$ is multiplicative, the tensor fields $N$ and $N\times N \times N$ satisfy \eqref{eq:morita-tensor-action-2}.

  To verify that Morita equivalence is symmetric, suppose that $G_{0}\gets P \to H_{0}$ is a Morita equivalence between $(\mathscr{G},N)$ and $(\mathscr{H},K)$. It follows that $H_{0}\gets P \to G_{0}$, equipped with the opposite actions, constitutes a Morita equivalence between $(\mathscr{H},K)$ and $(\mathscr{G},N)$.

  We now establish transitivity. The proof proceeds by observing that, given two Morita bibundles $P$ and $Q$ sharing a Lie groupoid $\mathscr{H}$, the direct product of the tensor fields on the bibundles restricts to the fiber product $P\times_{H_{0}}Q$ and descends to the quotient of the fiber product by the diagonal action of $\mathscr{H}$. Let
  \[
    \begin{tikzcd}
      {(G_1,N)} \arrow[d, shift right] \arrow[d, shift left] & (P,J) \arrow[ld, "\mu_1"'] \arrow[rd, "\mu_2"] & {(H_1,K)} \arrow[d, shift right] \arrow[d, shift left] \\
      {(G_0,r_N)}                                            &                                                & {(H_0,r_K)}
    \end{tikzcd}\qquad \text{and}\qquad
    \begin{tikzcd}
      {(H_1,K)} \arrow[d, shift right] \arrow[d, shift left] & (Q,J') \arrow[ld, "\mu_2'"'] \arrow[rd, "\mu_3"] & {(G_1',N')} \arrow[d, shift right] \arrow[d, shift left] \\
      {(H_0,r_{K})}                                          &                                                  & {(G_0',r_{N'})}
    \end{tikzcd}
  \]
  be Morita bibundles connecting $(\mathscr{G},N)$, $(\mathscr{H},K)$, and $(\mathscr{G}',N')$. To verify that $J\times J'$ restricts to the fiber product $P\times_{H_0}Q$, recall from \eqref{eq:tensor fields-related-moment-map} that $\mathrm{d}\mu_{2}(J(u)) = r_{K}(\mathrm{d}\mu_{2}(u))$ and $\mathrm{d}\mu_{2}'(J'(v)) = r_{K}(\mathrm{d}\mu_{2}'(v))$ for every $u\in TP$ and $v\in TQ$. If $(u,v)\in TP\times_{TH_{0}}TQ$, then
  \[
    \mathrm{d}\mu_{2}(J(u))= r_{K}(\mathrm{d}\mu_{2}(u)) = r_{K}(\mathrm{d}\mu_{2}'(v)) = \mathrm{d}\mu_{2}'(J'(v)).
  \]
  Therefore, the tensor field $J\times J'$ restricts to $TP\times_{TH_{0}}TQ$, and we have $J\times J'\in \Omega^{1}(P\times_{H_{0}}Q,TP\times_{TH_{0}}TQ)$.

To project $J \times J'$ to the quotient $P \diamond Q$, let $\mathscr{A} \colon H_{1} \times_{H_{0}} (P \times_{H_{0}} Q) \to P \times_{H_{0}} Q$ denote the diagonal action of $\mathscr{H}$. By \eqref{eq:morita-tensor-action-3}, this action satisfies the intertwining relation
\[
    \mathrm{d}\mathscr{A} \circ \left( K \times (J \times J') \right) = (J \times J') \circ \mathrm{d}\mathscr{A}.
\]
Consequently, by Lemma \ref{lem:projectable_tensor_principal_bundle}, $J \times J'$ projects to a well-defined $(1,1)$-tensor field $J \diamond J'$ on the quotient manifold $P \diamond Q$.

  The actions of $\mathscr{G}$ and $\mathscr{G}'$ on $P\diamond Q$ are induced by their respective actions on $P\times_{H_{0}}Q$. These actions descend to the quotient because they commute with the diagonal action of $\mathscr{H}$. Furthermore, these actions satisfy Equation \eqref{eq:morita-tensor-action-2} with respect to the tensor field $J\diamond J'$, proving that $(\mathscr{G},N)$ and $(\mathscr{G}',N')$ are Morita equivalent.

  To conclude the proof for Nijenhuis groupoids, assume that $J$ and $J'$ are Nijenhuis tensor fields. Then the product $J\times J'$ is also Nijenhuis (i.e., its Nijenhuis torsion $\mathscr{N}_{J\times J'}$ vanishes). Since $J\times J'$ is related to $J\diamond J'$, the naturality of the Nijenhuis torsion implies that $\mathscr{N}_{J\diamond J'} =0$, which completes the proof.
\end{proof}
In particular, we recover the well-known fact that Morita equivalence of holomorphic groupoids is an equivalence relation. It suffices to observe that the condition $(J\times J')^{2} = -\mathrm{id}_{T(P\times_{H_{0}}Q)}$ descends to the quotient by the diagonal action. 

\section{Infinitesimal Nijenhuis structures}
\label{sec:nijenhuis-lie-algebroids}
In this section, we study the infinitesimal version of the notion of Morita equivalence for Nijenhuis groupoids that we developed in Section \ref{sec:nijenhuis-groupoids}. We begin by defining \textit{infinitesimal Nijenhuis structures} and providing an equivalent characterization in terms of \textit{$1$-derivations}. We define the notion of Morita equivalence of infinitesimal Nijenhuis structures in Definition \ref{def:morita-equivalence-linear-tensor}, and in Theorem \ref{thm:morita-nijenhuis-groupoid-to-algebroid} we show that it corresponds to the differentiation of the global Morita equivalence of Nijenhuis groupoids.

By directly applying the Lie functor to a Nijenhuis groupoid, we obtain the following definition.
\begin{definition}
    \label{def:nijenhuis-lie-algebroid}
    An \textbf{infinitesimal Nijenhuis structure} is a pair $(A,R)$, where $A\to M$ is a Lie algebroid and $R\in \Omega^{1}(A,TA)$ is a Nijenhuis tensor field on the total space of $A$, such that $R:TA\to TA$ is a morphism of Lie algebroids from the tangent Lie algebroid to itself.
\end{definition}
In particular, $R$ is a \textit{linear $(1,1)$-tensor field} on the vector bundle $A\to M$, i.e., $R:TA\to TA$ is a vector bundle morphism from the tangent prolongation to itself:
\[
        \begin{tikzcd}
            TA \arrow[d] \arrow[r, "R"] & TA \arrow[d] \\
            TM \arrow[r, "r"]           & TM          \ , 
        \end{tikzcd}
    \]
for some $r\in \Omega^{1}(M,TM)$. Linear $(1,1)$-tensor fields that are morphisms of Lie algebroids are the infinitesimal counterparts of multiplicative $(1,1)$-tensor fields, they are called \textit{IM (infinitesimally multiplicative) $(1,1)$-tensor fields} \cite{bursztyn-drummond2019multiplicative}. These objects can be characterized in algebraic terms as \textit{$1$-derivations} as we explain in the following subsection.

\subsection{Characterization of IM tensor fields via 1-derivations}
A \textit{$1$-derivation} on a vector bundle $\pi:E\to M$ is a triple $\mathscr{D}=\left( \nabla,\ell , r\right)$, where  $r\in \Omega^{1}(M,TM)$ is a vector-valued $1$-form, $\ell \in \mathrm{End}(E)$ is an endomorphism of $E$, and $\nabla:\Gamma(E)\to \Omega^{1}(M,E)$ is an $\mathbb{R}$-linear map satisfying the \textit{Leibniz rule}
\begin{equation}
    \label{eq:leibniz}
    \nabla_{v}(f\xi) =f \nabla_{v}(\xi) + \left( \mathscr{L}_{v}f \right)\ell (\xi) - \left( \mathscr{L}_{r(v)}f \right)\xi,
\end{equation}
  for every $v \in \mathfrak{X}(M)$, $\xi \in \Gamma(E)$ and $f\in C^{\infty}(M)$.
\subsubsection{Correspondence between linear tensor fields and $1$-derivations} To make this correspondence explicit, we first recall the definition of the \textit{vertical lift}. For a section $\xi\in \Gamma(E)$, its vertical lift is the vector field $\xi^{\uparrow}\in \mathfrak{X}(E)$ defined by
\begin{equation}
  \label{eq:vertical-lift}
  \xi^{\uparrow}(e)\coloneqq \left. \frac{\mathrm{d}}{\mathrm{d}\tau}\right|_{\tau=0} \left( e +\tau \xi_{\pi(e)} \right).
\end{equation}
Linear $(1,1)$-tensor fields $R:TE\to TE$ are in one-to-one correspondence with  $1$-derivations $\mathscr{D}= \left( \nabla,\ell ,r \right)$ through the relations  \cite[Corollary 4.11]{bursztyn-drummond2019multiplicative}
  \begin{equation}
    \label{eq:1-der-linear-tensor}
    \nabla_{\mathrm{d}\pi(v)}(\xi) = \left( \mathscr{L}_{\xi^{\uparrow}}R \right)(v) , \quad R\left( \xi^{\uparrow} \right) = \ell (\xi)^{\uparrow}\quad \text{and}\quad\left\langle  R , \pi^{*}\alpha\right\rangle = \pi^{*}\left\langle  r,\alpha \right\rangle,
  \end{equation}
  for every $\xi \in \Gamma(E)$, $v\in TE$ and $\alpha\in \Omega^{1}(M)$.
\subsubsection{The Nijenhuis condition in terms of 1-derivations}
We can express the vanishing of the Nijenhuis torsion for a linear tensor field $R\in \Omega^{1}(E,TE)$ in terms of $1$-derivations. The tensor field $R$ is a Nijenhuis tensor field if and only if its corresponding $1$-derivation $(\nabla,\ell ,r)$ satisfies the equations
\begin{equation}
    \label{eq:nijenhuis-1-derivation}
    \begin{cases}
        \mathscr{N}_{r}(u,v)= 0, \\
        \ell \left( \nabla_{v}(\xi) \right)- \nabla_{v}\left( \ell (\xi) \right)= 0 ,\\
        \nabla^2_{(u,v)}(\xi) \coloneqq  \ell \left( \nabla_{\left[ u,v \right]}(\xi) \right) - \left[ \nabla_{u},\nabla_{v} \right](\xi)- \nabla_{\left[ u,v \right]_{r}} (\xi) = 0,
    \end{cases} 
\end{equation}
where $[u,v]_{r}\coloneqq [r(u),v]+[u,r(v)] - r\big([u,v]\big)$. We call Equations \eqref{eq:nijenhuis-1-derivation} the \textbf{Nijenhuis equations}.
\subsubsection{Infinitesimally multiplicative 1-derivations} The IM condition for $(1,1)$-tensor fields can be expressed equivalently in the language of $1$-derivations as follows. A $1$-derivation $\mathscr{D}=\left( \nabla,\ell ,r \right)$ is said to be an \textbf{IM $1$-derivation} if it satisfies the so-called \textbf{IM-equations} \cite[Section 6]{bursztyn-drummond2019multiplicative}:
\begin{equation}
  \label{eq:im-equations}
  \begin{cases}
    \nabla_{v}\left( \left[ \xi,\zeta \right]_{A} \right) = \left[ \nabla_{v}(\xi),\zeta \right]_{A}+\left[ \xi,\nabla_{v}(\zeta) \right]_{A}-\nabla_{\left[ \rho(\zeta),v \right]_{A}}(\xi) - \nabla_{\left[ \rho(\xi),v \right]_{A}}(\zeta)\\
    \ell \left( \left[ \xi,\zeta \right]_{A} \right) = \left[ \xi,\ell (\zeta) \right]_{A} - \nabla_{\rho(\zeta)}(\xi)\\
    \rho\left( \nabla_{v}(\xi) \right) = \nabla^{r}_{v}\left( \rho(\xi) \right),\\
    r\circ \rho = \rho \circ \ell .
  \end{cases} 
\end{equation}

An infinitesimal Nijenhuis structure on  a Lie algebroid $A\to M$ can now be reformulated as an IM $1$-derivation $\mathscr{D}=(\nabla,\ell,r)$ satisfying the Nijenhuis equations \eqref{eq:nijenhuis-1-derivation}.

\begin{example}
  \label{eg:tangent-cotangent-lift}
  The \textit{tangent lift} $r^{\mathrm{tg}}\in \Omega^{1}(TM,T(TM))$ of a given $(1,1)$-form $r\in \Omega^{1}(M,TM)$ is a linear $(1,1)$-tensor field on $TM\to M$  (Equation \eqref{eq:tangent-lift}) is always an IM tensor field (\cite[Corollary 3.5]{drummond2022lie}). We consider also the \textit{cotangent lift} $r^{\mathrm{cotg}}\in \Omega^{1}(T^{*}M,T(T^{*}M))$, it is a linear tensor field constructed as follows. Let $\omega_{\mathrm{can}}\in \Omega^{2}(T^{*}M)$ be the canonical symplectic form on $T^{*}M$ and $\phi_{r}\coloneqq r^{*} :T^{*}M\to T^{*}M$. Then $r^{\mathrm{cotg}}$ is defined by
  \[
    \iota_{r^{\mathrm{cotg}}(v)}\omega_{\mathrm{can}} = \iota_{v}\left( \phi^{*}_{r}\omega_{\mathrm{can}} \right), \qquad  v\in T(T^{*}M).
  \]
  If $r$ is a Nijenhuis tensor field, then both $r^{\mathrm{tg}}$ and $r^{\mathrm{cotg}}$ are Nijenhuis tensor fields. To define the corresponding $1$-derivations of $r^{\mathrm{tg}}$ and $r^{\mathrm{cotg}}$ consider the non-linear operators $\nabla^{r}:\mathfrak{X}(M)\to \Omega^{1}(M,TM)$ and $\nabla^{r,*}:\Omega^{1}(M)\to \Omega^{1}(M,T^{*}M)$ given by
 \[
   \nabla^{r}_{v}(u) = \iota_{v}\left( \nabla^{r}(u) \right) = \left( \mathscr{L}_{u}r \right)(v)  \quad \text{and}\quad \nabla_{v}^{r,{*}}(\alpha) = \iota_{v}\left( \nabla^{r,*}(\alpha) \right)\coloneqq \mathscr{L}_{v}r^{*}(\alpha)-\mathscr{L}_{r(v)}\alpha.
 \]
 More explicitly, $\nabla^{r}_{v}(u)=\left[ u,r(v) \right] - r \left( \left[ u,v \right]\right)$ and $\nabla^{r,*}_{v}(\alpha) = \iota_{v}\mathrm{d}(r^{*}\alpha) - \iota_{r(v)}\mathrm{d}\alpha$. The corresponding $1$-derivations to the tangent and cotangent lifts are 
 \[
   \mathscr{D}^{r} = \left( \nabla^{r},r,r \right) \qquad \text{and} \qquad \mathscr{D}^{r,*} = \left( \nabla^{r,*},r^{*},r \right),
 \]
 respectively \cite[Section 3]{drummond2022lie}.
\end{example}
\begin{example}
  \label{eg:holomorphic-vector-bundle}
  A holomorphic vector bundle can be characterized as a (real) vector bundle $E\to M$, together with a complex structure $r \in \Omega^{1}(M,TM)$ on $M$, an endomorphism $\ell :E\to E$ such that $\ell^2=-\mathrm{id}_{E}$ and a flat $T^{(1,0)}M$-connection $\nabla':\Gamma\left( T^{(1,0)}M \right)\times \Gamma(E)\to \Gamma(E)$ on the complex vector bundle $(E,\ell)$. Equivalently, the holomorphic structure of $E$ is encoded in the $1$-derivation $\mathscr{D}^{\mathrm{hol}}=\left( \nabla^{\mathrm{hol}},\ell ,r \right)$, where 
  \[
      \nabla_{v}^{\mathrm{hol}} (u) \coloneqq \ell \left(\nabla'_{v+ir(v)} u \right).
  \]
  Moreover, a $1$-derivation $\mathscr{D}=\left( \nabla,\ell ,r \right)$ defines a holomorphic vector bundle if and only if
    \begin{equation}
     \label{eq:holomorphic-1-derivations}
      r^{2} = -\mathrm{id}_{TM},\quad \ell^2 = - \mathrm{id}_E \quad \text{and}\quad \ell(\nabla_v(u))+\nabla_{r(v)}(u)=0,
    \end{equation}
  and the Nijenhuis conditions \eqref{eq:nijenhuis-1-derivation} are satisfied (see \cite[Example 2.3]{bursztyn-drummond-netto2022dirac}).
\end{example}
 \begin{example}
    \label{eg:holomorphic-1-derivations}
    A \textit{holomorphic Lie algebroid} \cite{chemla1999duality,weinstein2007integration} is a holomorphic vector bundle $A\to M$ equipped with a holomorphic anchor and  a Lie bracket that satisfies the Leibniz identity. Equivalently, it is a (real) Lie algebroid equipped with an IM $(1,1)$-tensor field $I_{A}\in \Omega^{1}(A,TA)$ with vanishing Nijenhuis torsion \cite[Proposition 2.3]{laurent2009integration}. In the language of $1$-derivations, the holomorphic structure $I_{A}$ corresponds to an IM $1$-derivation $(\nabla,\ell ,r)$  satisfying \eqref{eq:holomorphic-1-derivations}.
\end{example}

\subsection{The Lie correspondence for (1,1)-tensor fields}

Let $\mathscr{G}$ be a Lie groupoid with Lie algebroid $A_{\mathscr{G}}$. The tangent Lie algebroid $TA_{\mathscr{G}}\to TM$ is integrable to the tangent Lie groupoid $T\mathscr{G}$ (specifically, via the canonical identification $TA_{\mathscr{G}}\cong A_{T\mathscr{G}}$; see \cite[Theorem 7.1]{mackenzie1994lie}). Assuming that $\mathscr{G}$ is source-simply connected, Lie's second theorem guarantees that any IM $(1,1)$-tensor field $R:TA_{\mathscr{G}}\to TA_{\mathscr{G}}$, viewed as a Lie algebroid morphism, integrates to a unique Lie groupoid morphism, yielding a multiplicative $(1,1)$-tensor field $R_{\mathscr{G}}\in \Omega^{1}(\mathscr{G},T\mathscr{G})$. The Lie functor is the inverse of this integration process, establishing a bijection between multiplicative $(1,1)$-tensor fields and IM tensor fields. Furthermore, this bijection preserves the Nijenhuis condition on the $(1,1)$-tensor fields (see \cite[Theorem 3.14]{laurent2009integration}).

\begin{lemma}
  \label{lemma:tensor fields-correspondence-nijenhuis}
  Let $\mathscr{G}$ be a source-simply connected Lie groupoid. The Lie functor establishes a bijection between multiplicative $(1,1)$-tensor fields $K\in \Omega^{1}(\mathscr{G},T\mathscr{G})$ on $\mathscr{G}$ and IM $(1,1)$-tensor fields $R_{K}\in \Omega^{1}(A_{\mathscr{G}},TA_{\mathscr{G}})$ on $A_{\mathscr{G}}$. Moreover, $K $ is Nijenhuis if and only if its corresponding $R_{K}$ is Nijenhuis.
\end{lemma}
In particular, there is a one-to-one correspondence between (source-simply connected) Nijenhuis groupoids and (integrable) infinitesimal Nijenhuis structures. The IM $1$-derivation $\mathscr{D}=\left( \nabla,\ell ,r \right)$ on $A_{\mathscr{G}}$ corresponding to a multiplicative  $(1,1)$-tensor field $K\in \Omega^{1}\left( \mathscr{G},T\mathscr{G} \right)$ is defined by the equations
\begin{equation}
  \label{eq:IM-derivation-groupoid}
  K\left( \overrightarrow{\xi} \right) = \overrightarrow{\ell (\xi)},\quad \mathrm{d}\mathbf{t}\left( K(v)  \right) = r\left( \mathrm{d}\mathbf{t}(v) \right) \quad \text{and}\quad \left( \mathscr{L}_{\overrightarrow{\xi}}K\right)(v) = \overrightarrow{\nabla_{\mathrm{d}\mathbf{t}(v)}(\xi)},
\end{equation}
for every $\xi \in \Gamma(A_{\mathscr{G}})$ and $\mathbf{t}$-projectable vector fields $v \in \mathfrak{X}\left(G_{1}\right)$.
 
\subsubsection{Integration of tangent lifts} The tangent bundle $TP\to P$ of a manifold $P$, equipped with its canonical Lie algebroid structure, integrates to the pair groupoid $\mathrm{Pair}(P)= (P\times P\rightrightarrows  P)$. The \textbf{tangent lift} $J^{\mathrm{tg}}\in \Omega^{1}(TP,T(TP))$ of a vector-valued $1$-form $J\in \Omega^{1}(P,TP)$ is a $(1,1)$-tensor field on $TP$ given by the conjugation of $\mathrm{d}J$ by the \textit{canonical involution} $\kappa_{P}:T(TP)\to T(TP)$; that is,
  \begin{equation}
    \label{eq:tangent-lift}
    J^{\mathrm{tg}}\coloneqq \kappa_{P}\circ\mathrm{d}J\circ \kappa_{P}:T(TP)\to T(TP).
  \end{equation}
  The tangent lift $J^{\mathrm{tg}}$ is an IM $(1,1)$-tensor field on $TP$. We now focus on its integration to a multiplicative $(1,1)$-tensor field on $\mathrm{Pair}(P)$, because it is used in the integration of Morita equivalences.
\begin{lemma}
  \label{lemma:integration-tangent-lift}
  For every $J \in \Omega^{1}(P,TP)$, the tensor field $J\times J \in \Omega^{1}(P\times P,T(P\times P))$ is a multiplicative tensor field in the pair groupoid $P\times P\rightrightarrows P$ which differentiates to the tangent lift $J^{\mathrm{tg}}$. In particular, $J^{\mathrm{tg}}$ is an IM tensor field.
\end{lemma}
\begin{proof}
Observe that $J\times J$ is a Lie groupoid morphism (with base map $J$) on the tangent Lie groupoid. Consequently, it is a multiplicative $(1,1)$-tensor field on $\mathscr{G}$ by definition. 

We next show that $J^{\mathrm{tg}}$ is the differentiation of $J\times J$ under the Lie functor. Theorem 7.1 of \cite{mackenzie1994lie} establishes a canonical isomorphism between the double vector bundles $T\left( A_{\mathscr{G}} \right)$ and $A_{T\mathscr{G}}$ for any Lie groupoid $\mathscr{G}$, which is an isomorphism of Lie algebroids. This isomorphism is the restriction of the canonical involution $\kappa_{G_{1}}:T(TG_{1})\to T(TG_{1})$ to a map between $TA_{\mathscr{G}}=T \left(\mathbf{u}_{\mathscr{G}}^{*}\operatorname{ker} \mathrm{d}\mathbf{s}_{\mathscr{G}}\right)$ and  $A_{T\mathscr{G}}=\mathbf{u}_{T\mathscr{G}}^{*}\operatorname{ker} \mathrm{d} \mathbf{s}_{T\mathscr{G}}$. In the present case, $G_{1}  = P\times P$ and $G_{0}= P$, whereas $A_{\mathscr{G}} = TP\to P$ is the tangent bundle of $P$ and $A_{T\mathscr{G}} = T(TP)\to TP$ is the tangent Lie algebroid. It follows that $\kappa_{G_{1}}=\kappa_{P}\times \kappa_{P}$, and the isomorphism corresponds to the upper arrow in the following diagram:
\[
  \begin{tikzcd}
      TA_{\mathscr{G}} \arrow[d, hook] \arrow[rr, "\kappa_{P}"]       &      & A_{T\mathscr{G}} \arrow[d, hook]                        \\
      T \operatorname{ker} \mathrm{d}\mathbf{s}_{\mathscr{G}} \arrow[rr, "\kappa_{P}\times\kappa_{P}"] &  & \operatorname{ker} \mathrm{d} \mathbf{s}_{T\mathscr{G}}
\end{tikzcd}
\]
Consider the map $\mathrm{d}J \times\mathrm{d}J :TG_{1}\times TG_{1}\to  TG_{1}\times TG_{1}$. Since $\mathrm{d}\mathbf{s}_{\mathscr{G}}:(u,v)\mapsto v$, the map $\mathrm{d}J\times \mathrm{d}J$ clearly preserves $T\operatorname{ker}\mathrm{d}\mathbf{s}_{\mathscr{G}}$. Because $\kappa_{P}$ is involutive, we obtain the following commutative diagram on the left, which restricts to the commutative diagram on the right:
\[
\begin{tikzcd}
T \operatorname{ker} \mathrm{d}\mathbf{s}_{\mathscr{G}}  \arrow[rr, "\mathrm{d}J \times\mathrm{d}J"]            &  &  T \operatorname{ker} \mathrm{d}\mathbf{s}_{\mathscr{G}}  \arrow[d, "\kappa_P\times\kappa_P"] \\
 \operatorname{ker} \mathrm{d} \mathbf{s}_{T\mathscr{G}} \arrow[rr, dashed] \arrow[u, "\kappa_P\times\kappa_P"] &  & \operatorname{ker} \mathrm{d} \mathbf{s}_{T\mathscr{G}}                                      
\end{tikzcd}
 \qquad\longleftarrow\joinrel\rhook\qquad 
\begin{tikzcd}
TA_{\mathscr{G}} \arrow[rr, "\mathrm{d}J"]                &  & TA_{\mathscr{G}} \arrow[d, "\kappa_P"] \\
A_{T\mathscr{G}} \arrow[u, "\kappa_P"] \arrow[rr, dashed] &  & A_{T\mathscr{G}}                      
\end{tikzcd}
\]
Therefore, $J\times J$ induces the morphism $J^{\mathrm{tg}}= \kappa_{P}\circ \mathrm{d}J\circ\kappa_{P}$ on the Lie algebroid $A_{T\mathscr{G}}$, which concludes the proof.
\end{proof}
\begin{remark}
  \label{rmk:tangent-lift-integration}
  Every IM $(1,1)$-tensor field $R\in \Omega^{1}(TP,T(TP))$ on a tangent bundle $TP\to P$ is of the form $R=J^{\mathrm{tg}}$ for some $J\in \Omega^{1}(P,TP)$ (see \cite[Corollary 3.5]{drummond2022lie}).
\end{remark}
 
\subsection{Morita equivalence of infinitesimal structures}

In this subsection, we introduce a notion of Morita equivalence for IM $(1,1)$-tensor fields on Lie algebroids. While distinct notions of Morita equivalence for Lie algebroids exist in the literature—defined, for instance, via isomorphic pullback Lie algebroids \cite{ginzburg2001grothendieck} or in terms of preserved geometric data \cite{villatoro2018poisson}—we adopt a simplified definition (Definition \ref{def:morita-equiv-algebroid}) resulting from differentiating the global formulation of Morita equivalence via principal bibundles. This approach is analogous to the construction of Morita equivalence for Poisson manifolds, which is obtained by differentiating Morita equivalences of symplectic groupoids \cite{xu1991morita}. Consequently, our framework is built upon \textit{infinitesimal actions}. Because this differentiation yields complete actions, this definition applies exclusively to integrable Lie algebroids (see Remark \ref{rmk:daniel-complete}), as is the case for Poisson structures.

An \textbf{infinitesimal left action} (or simply \textit{action}) of a Lie algebroid $A\to M$ on a smooth map $\mu:P\to M$ is a Lie algebra morphism $\mathscr{a}:\Gamma(A)\to \mathfrak{X}(P)$ satisfying
\begin{equation}
  \label{eq:lie-algebroid-action}
  \rho(\xi)= \mathrm{d}\mu\left(\mathscr{a}(\xi)\right)\qquad \text{and}\qquad \mathscr{a}\left( f\xi \right) = (\mu^{*} f)\mathscr{a}(\xi),
\end{equation}
for every $\xi \in \Gamma(A)$ and $f \in C^{\infty}(M)$. The second condition, referred to as \textit{$C^{\infty}(M)$-linearity}, ensures that the action induces a vector bundle map $\mathscr{a}:\mu^{*}A\to TP$ covering the identity. The action is said to be \textbf{complete} if the vector field $\mathscr{a}(\xi)$ on $P$ is complete for every compactly supported section $\xi\in  \Gamma(A)$. 

A right infinitesimal action of a Lie algebroid $A\to M$ is defined as a left action of the \textit{opposite Lie algebroid} $A^{\mathrm{op}}$. For a given Lie algebroid $(A,[-,-]_{A},\rho)$, its opposite Lie algebroid $A^{\mathrm{op}}$ consists of the same vector bundle $A\to M$ equipped with the Lie bracket $[\xi,\zeta]_{A^{\mathrm{op}}}=-[\xi,\zeta]_{A}$ and anchor map $\rho_{\mathrm{op}}\coloneqq -\rho$. 
\begin{definition}
  \label{def:morita-equiv-algebroid}
  Two Lie algebroids $A_{1}\to M_{1}$ and $A_{2}\to M_{2}$ are said to be \textbf{Morita equivalent} if there are two surjective submersions
  \[
    M_{1}\xleftarrow{\quad \mu_{1}\quad } P\xrightarrow{\quad \mu_{2} \quad} M_{2},
  \]
  both with connected and simply-connected fibers, together with two \textit{complete} infinitesimal actions $\mathscr{a}_{1}:\Gamma(A_{1})\to \mathfrak{X}(P)$ and $\mathscr{a}_{2}:\Gamma(A_{2}^{\mathrm{op}})\to \mathfrak{X}(P)$, satisfying the following properties:
\begin{enumerate}[label=(\alph*)]
   \item The actions commute, i.e., for all $\xi\in \Gamma(A_{1})$ and $\zeta\in \Gamma(A_{2}^{\mathrm{op}})$,
\begin{equation}
 \label{eq:algebroid-morita-commuting}
  \left[ \mathscr{a}_{1}(\xi),\mathscr{a}_{2}(\zeta) \right]=0. 
\end{equation}
  \item For every $p\in P$, the bundle maps $\mathscr{a}_{1\, p}:\left(\mu_{1}^{*}A_{1}\right)_{p}\to T_{p}P$ and $\mathscr{a}_{2\, p}:\left(\mu_{2}^{*}A_{2}\right)_{p}\to T_{p}P$ are  injective, and
\begin{equation}
  \label{eq:algebroid-morita-fibers}
  \operatorname{im}\left(\mathscr{a}_{1 \, p}\right) = \operatorname{ker}\mathrm{d}\mu_{2\, p} \qquad \text{and}\qquad \operatorname{im} \left(\mathscr{a}_{2\, p }\right)= \operatorname{ker}\mathrm{d}\mu_{1\, p} .
\end{equation}
\end{enumerate}
In this setting, we call $P$ an \textbf{infinitesimal Morita $(A_1, A_2)$-bibundle}.
\end{definition}
\begin{remark}
  \label{rmk:daniel-complete}
  The existence of a complete action of a Lie algebroid on a surjective submersion is equivalent to the integrability of the algebroid (\cite[Theorem 4.1]{alvarez2021complete} and \cite[Theorem 8]{crainic-fernandes2004integrability}). Consequently, Definition \ref{def:morita-equiv-algebroid} applies exclusively to integrable Lie algebroids. This ensures that infinitesimal Morita equivalences of Lie algebroids correspond to Morita equivalences of the respective source-simply connected Lie groupoids (Lemma \ref{lemma:correspondence-morita}).  In contrast, the notion of Morita equivalence in \cite{ginzburg2001grothendieck} is defined by requiring the pullback Lie algebroids $\mu_{1}^{!}A_{1}$ and $\mu_{2}^{!}A_{2}$ to be isomorphic (with simply-connected $\mu_{i}$-fibers), a condition that arises from differentiating Morita morphisms rather than bibundles.
\end{remark}
\begin{remark}
    \label{rmk:total_action}
    The conditions for Morita equivalence of Lie algebroids allow us to construct an action of the product algebroid $A_{1} \times A_{2}^{\mathrm{op}}$ on the submersion $\mu \coloneqq (\mu_{1}, \mu_{2}) \colon P \to M_{1} \times M_{2}$. This action is defined by
    \begin{equation}
        \label{eq:sum-of-actions}
        \Gamma(A_{1}) \oplus \Gamma(A_{2}^{\mathrm{op}}) \to \mathfrak{X}(P), \qquad (\xi_{1}, \xi_{2}) \mapsto \mathscr{a}_{1}(\xi_{1}) + \mathscr{a}_{2}(\xi_{2}),
    \end{equation}
    and extended by $C^{\infty}(M_{1} \times M_{2})$-linearity to a Lie algebroid morphism $\mathscr{a} \colon \Gamma(A_{1} \times A_{2}^{\mathrm{op}}) \to \mathfrak{X}(P)$. We refer to this morphism as the \textit{total (infinitesimal) action}, and we denote it by $\mathscr{a} = \mathscr{a}_{1} + \mathscr{a}_{2}$.
\end{remark}
It is known that a complete action of a Lie algebroid integrates to a global action of the corresponding source-simply connected Lie groupoid  (see \cite[Theorem 2.1]{mokri1996lie} and \cite[Theorem 5.3]{moerdijk2002integrability}). More precisely, if $\mathscr{a}:\Gamma(A)\to  \mathfrak{X}(P)$ is a complete action of $A\to M$ on $\mu:P\to M$, then there exists a unique action $\mathscr{A}:G\times_{(\mathbf{s},\mu)}P\to P$ of the source-simply connected Lie groupoid $G\rightrightarrows M$ integrating $A$. The infinitesimal and global actions are related by
\begin{equation}
    \label{eq:integration-action}
    \mathscr{a}(\xi)_{p} = \left. \frac{\mathrm{d}}{\mathrm{d}\tau}\right|_{\tau=0} \mathscr{A}(\varphi(\tau),p),
    \end{equation}
    for any $\xi\in \Gamma(A)$. Here, $\varphi(\tau)$ is a path in the source-fiber $\mathbf{s}^{-1}(\mu(p))$ such that $\dot{\varphi}(0) = \overrightarrow{\xi}_{\mu(p)} $. The action $\mathscr{A}$ is called the \textit{integration of the infinitesimal action} $\mathscr{a}$. The conditions of Definition \ref{def:morita-equiv-algebroid} are formulated so that this correspondence between global and infinitesimal actions establishes an equivalence between the respective notions of Morita equivalence.
\begin{lemma}
    \label{lemma:correspondence-morita}
    Let $\mathscr{G}=\left( G_{1}\rightrightarrows G_{0} \right)$ and $\mathscr{H}=\left( H_{1}\rightrightarrows H_{0} \right)$ be Lie groupoids, and let $A_{\mathscr{G}}$ and $A_{\mathscr{H}}$ be their associated Lie algebroids.
    \begin{enumerate}[label=(\alph*)]
        \item If $\mathscr{G}$ and $\mathscr{H}$ are Morita equivalent via a Morita $(\mathscr{G},\mathscr{H})$-bibundle $P$, then $A_{\mathscr{G}}$ and $A_{\mathscr{H}}$ are Morita equivalent. In this case, $P$ serves as the infinitesimal Morita $(A_{\mathscr{G}},A_{\mathscr{H}})$-bibundle, with infinitesimal actions induced by the global actions via \eqref{eq:integration-action}.
        \item Conversely, suppose $\mathscr{G}$ and $\mathscr{H}$ are source-simply connected. If $A_{\mathscr{G}}$ and $A_{\mathscr{H}}$ are Morita equivalent via an infinitesimal Morita $(A_{\mathscr{G}},A_{\mathscr{H}})$-bibundle $P$, then $\mathscr{G}$ and $\mathscr{H}$ are Morita equivalent. In this case, $P$ serves as the Morita $(\mathscr{G},\mathscr{H})$-bibundle, with global actions obtained by integrating the infinitesimal actions according to \eqref{eq:integration-action}.
    \end{enumerate}
\end{lemma}

\subsubsection{Infinitesimal Morita equivalence of linear $(1,1)$-tensor fields}
To define the Morita equivalence of IM $(1,1)$-tensor fields, we must first introduce the pullback of linear $(1,1)$-tensor fields. Let $E \to M$ be a vector bundle, and let $R \colon TE \to TE$ be a linear $(1,1)$-tensor field on $E$ covering a $(1,1)$-tensor field $r \colon TM \to TM$. Furthermore, let $\mu \colon P \to M$ be a smooth map, and let $J \colon TP \to TP$ be a $(1,1)$-tensor field on $P$ that is $\mu$-related to $r$.
\begin{definition}
The \textbf{pullback $(1,1)$-tensor field} $\mu^{*}R: T(\mu^{*}E) \to T(\mu^{*}E)$ of $R$ (with respect to $J$) is the linear $(1,1)$-tensor field on $\mu^{*}E \to P$ covering $J$ defined by:
\begin{equation}
    \label{eq:pullback_tensor}
    (\mu^{*}R)(v, X) = (J(v), R(X))
\end{equation}
for every $(v, X) \in TP \times_{TM} TE \cong T(\mu^* E)$.
\end{definition}
Since $J$ and $r$ are $\mu$-related and $R$ covers $r$, the pair $(J(v), R(X))$ remains in the fiber product, ensuring the map is well-defined.

Now we extend the notion of Morita equivalence for Lie algebroids to a relation between IM $(1,1)$-tensor fields. In the following, $R_{1}\in \Omega^{1}(A_{1},TA_{1})$ and $R_{2}\in \Omega^{1}(A_{2},TA_{2})$ are IM $(1,1)$-tensor fields with base maps $r_{1}:TM_{1}\to TM_{1}$ and $r_{2}:TM_{2}\to TM_{2}$, respectively.

\begin{definition}
  \label{def:morita-equivalence-linear-tensor}
  A \textbf{Morita equivalence of IM $(1,1)$-tensor fields} $R_{1}$ and $R_{2}$ consists of a Morita $(A_{1},A_{2})$-bibundle
  \[
      M_1 \xleftarrow{\quad \mu_{1}\quad } (P,J) \xrightarrow{\quad \mu_{2}\quad } M_{2},
   \]
   equipped with a linear $(1,1)$-tensor field $J\in \Omega^{1}(P,TP)$ that is $\mu_{i}$-related to $r_{i}$ for $i=1,2$ and satisfies
  \begin{equation}
    \label{eq:morita-equivalence-linear-tensor}
    \mathrm{d}\mathscr{a}\circ \mu^{*}(R_{1}\times R_{2}) = J^{\mathrm{tg}}\circ\mathrm{d}\mathscr{a},
  \end{equation}
  where $\mathscr{a} = \mathscr{a}_{1}+\mathscr{a}_{2}$ denotes the total action and $\mu\coloneqq (\mu_{1},\mu_{2})$.
  This relation is called a \textbf{Morita equivalence of Nijenhuis IM tensor fields} if $R_{1}$, $R_{2}$, and $J$ are Nijenhuis tensor fields.
\end{definition}

\begin{remark}
  \label{rmk:morita-equivalence-linear-tensor}
  Instead of equation $\eqref{eq:morita-equivalence-linear-tensor}$, we can write the equivalent condition
  \begin{equation}
    \label{eq:morita-equiv-linear-graph}
    R_{1}\times R_{2}\times J^{\mathrm{tg}} \left( T \mathrm{Graph}(\mathscr{a}) \right) \subseteq T\mathrm{Graph}(\mathscr{a}).
  \end{equation}
\end{remark}
\begin{example}
  \label{eg:morita-holomorphic-algebroid}
  Consider two holomorphic Lie algebroids, $(A_{1},I_{1})$ and $(A_{2},I_{2})$, and suppose they are Morita equivalent via a bibundle $M_{1}\xleftarrow{\ \mu_{1}\ } (P,J)\xrightarrow{\ \mu_{2}\ } M_{2}$, where $(P,J)$ is a complex manifold. One can verify that the tangent lift of $J$ is a complex structure on  $TP$. Consequently, Equation \eqref{eq:morita-equivalence-linear-tensor} guarantees that the infinitesimal action $\mathscr{a}:\Gamma(A_{1}\times A_{2}^{\mathrm{op}})\to \mathfrak{X}(P)$ is holomorphic. 
\end{example}
 
\subsubsection{Infinitesimal Morita equivalence of $1$-derivations.}  
We now translate the formulation of Morita equivalence of IM tensor fields into the algebraic language of $1$-derivations. An advantage of this formulation is that it avoids the need to explicitly perform the differential of the infinitesimal actions by separating the tensorial conditions from the differential ones. As seen below, Equations \eqref{eq:base_1-der_morita} and \eqref{eq:end_1-der_morita} are purely tensorial; the differential content is entirely isolated within condition \eqref{eq:connection_1-der_morita} for the connection-like components.

To fix notation for the following definition, let $\mathscr{D}_{1}=(\nabla^{1},\ell_{1},r_{1})$ and $\mathscr{D}_{2}=(\nabla^{2},\ell_{2},r_{2})$ be IM $1$-derivations on the Lie algebroids $A_1\to M_{1}$ and $A_2\to M_{2}$, respectively. Additionally, let $\mathscr{a}_{1}:\Gamma(A_{1})\to \mathfrak{X}(P)$ and $\mathscr{a}_{2}:\Gamma(A_{2}^{\mathrm{op}})\to \mathfrak{X}(P)$ denote the infinitesimal actions, and set $\mathscr{a}\coloneqq \mathscr{a}_{1}+\mathscr{a}_{2}$.

\begin{definition}
\label{def:morita-equiv-1-derivations}
A \textbf{Morita equivalence of $1$-derivations} $\mathscr{D}_{1}=(\nabla^{1},\ell_{1},r_{1})$ and $\mathscr{D}_{2}=(\nabla^{2},\ell_{2},r_{2})$ consists of a Morita $(A_{1},A_{2})$-bibundle $M_{1}\xleftarrow{\ \mu_{1}\ } P \xrightarrow{\ \mu_{2}\ } M_{2}$, together with a $(1,1)$-tensor field $J\in \Omega^{1}(P,TP)$, satisfying
\begin{numcases}{}
      \mathrm{d}\mu \circ J = (r_{1}\times r_{2})\circ \mathrm{d}\mu\label{eq:base_1-der_morita}\\
      J\circ \mathscr{a} = \mathscr{a}\circ (\ell_{1}\times \ell_{2}) \label{eq:end_1-der_morita}\\
      \nabla_{v}^{J}\left( \mathscr{a}_{1}(\xi)  + \mathscr{a}_{2}(\zeta) \right) = \mathscr{a}_{1}\left( \nabla^{1}_{\mathrm{d}\mu_{1}(v)} (\xi) \right) + \mathscr{a}_{2}\left( \nabla^{2}_{\mathrm{d}\mu_{2}(v)} (\zeta) \right) \label{eq:connection_1-der_morita}
\end{numcases}
for every $v\in TP$, $\xi\in \Gamma(A_{1})$ and $\zeta \in \Gamma(A_{2})$.
\end{definition}

\begin{proposition}
  \label{prop:morita-equiv-1-derivation}
  A Morita equivalence of Lie algebroids $M_{1}\xleftarrow{\ \mu_{1}\ }(P,J)\xrightarrow{\ \mu_{2}\ }M_{2}$ equipped with a $(1,1)$-tensor field $J\in \Omega^{1}(P,TP)$ is a Morita equivalence of IM $(1,1)$-tensor fields $R_{1}\in \Omega^{1}(A_{1},TA_{1})$ and $R_{2}\in \Omega^{1}(A_{2},TA_{2})$ if and only if it is a Morita equivalence of the respective $1$-derivations $\mathscr{D}_{1}$ and $\mathscr{D}_{2}$.
\end{proposition}
The proof of this proposition relies on the following lemma, which characterizes the pullback of $1$-derivations.
\begin{lemma}
  \label{lemma:pullback_1-derivations}
    Let $\pi:E \to M$ be a vector bundle equipped with a linear $(1,1)$-tensor field $R\in \Omega^{1}(E, TE)$ covering $r \in \Omega^{1}(M,TM)$, and let $\mathscr{D} = (\nabla, \ell,r)$ be its associated $1$-derivation.
  Let $\mu: P \to M$ be a smooth map and suppose $J \in \Omega^{1}(P,TP)$ is a $(1,1)$-tensor field $\mu$-related to $r$. Then the $1$-derivation associated with the pullback $\mu^{*}R$ is given by the triple $\mu^{*}\mathscr{D} = \left( \mu^{*}\nabla , \mu^{*}\ell , J \right)$, whose components are uniquely determined by their action on the pullback of sections:
  \begin{numcases}{}
    (\mu^{*}\nabla)_{v}(\mu^{*}\xi) = \mu^{*}\left( \nabla_{\mathrm{d}\mu(v)}\xi \right) , \label{eq:pullback_connection} \\
    (\mu^{*}\ell)(\mu^{*}\xi) = \mu^{*}(\ell(\xi)), \label{eq:pullback_end}
  \end{numcases} 
  for every $v \in TP$ and $\xi \in \Gamma(E)$.
\end{lemma}
\begin{proof}
    It suffices to determine $\mu^{*}\mathscr{D}$ on sections of the form $\mu^{*}\xi$ for $\xi \in \Gamma(E)$ because the $C^{\infty}(P)$-module of sections $\Gamma(\mu^{*}E)$ is locally generated by the pullbacks of a frame of $E$. Indeed, since the component $\mu^{*}\ell$ is an endomorphism of $\mu^{*}E$, it is fully characterized by its action on these generators. Similarly, the connection component $\mu^{*}\nabla$ is determined by its values on pullback sections because it must satisfy the Leibniz rule (Equation \eqref{eq:leibniz}):
\[
    \left(\mu^{*}\nabla\right)_{v}(f\zeta) = f \left(\mu^{*}\nabla\right)_{v}(\zeta) + (\mathscr{L}_{v}f)(\mu^{*}\ell)(\zeta) - (\mathscr{L}_{J(v)}f)\zeta,
\]
for any $v \in TP$, $f \in C^{\infty}(P)$, and $\zeta \in \Gamma(\mu^{*}E)$.

Now, applying the correspondence \eqref{eq:1-der-linear-tensor} between linear $(1,1)$-tensors and $1$-derivations, we directly identify the base component as $J \in \Omega^{1}(P, TP)$. To check the other components, note that the vertical lift of a pullback section $\mu^{*}\xi \in \Gamma(\mu^{*}E)$ at $(p,e) \in \mu^{*}E$ is given by:
\[
    \left(\mu^{*}\xi\right)^{\uparrow}_{(p,e)} = \left.\frac{\mathrm{d}}{\mathrm{d}\tau}\right|_{\tau =0} \left( p, e + \tau\xi_{\pi(e)} \right) = \left( 0_{p}, \xi^{\uparrow}_{e} \right).
\]
The pullback tensor $\mu^*R$ acts on a tangent vector $(v, w) \in T(\mu^*E)$ by $\mu^*R(v,w) = (J(v), R(w))$ (Equation \eqref{eq:pullback_tensor}). Therefore, applying this to the vertical lift and using the correspondence \eqref{eq:1-der-linear-tensor}, we obtain:
\[
   \mu^*R \left( (\mu^*\xi)^\uparrow \right) = \mu^*R(0, \xi^\uparrow) = \left( 0, R(\xi^{\uparrow}) \right) = \left( 0, \ell (\xi)^{\uparrow} \right) = \left( \mu^{*}(\ell (\xi)) \right)^{\uparrow},
\]
from which it follows that $(\mu^{*}\ell)(\mu^{*}\xi) = \mu^{*}(\ell(\xi))$. Similarly, for the connection component, we compute the Lie derivative at a vector $(v,w) \in T(\mu^{*}E) \cong TP\times_{TM}TE$ and apply the correspondence \eqref{eq:1-der-linear-tensor}:
\[
  \left( \mathscr{L}_{(\mu^{*}\xi)^{\uparrow}}(\mu^{*}R) \right)(v,w)
  = \left( 0, (\mathscr{L}_{\xi^{\uparrow}}R)(w)\right)
  = \left( 0, \left(\nabla_{\mathrm{d}\pi(w)} \xi\right)^{\uparrow}\right).
\]
Since $(v,w)$ lies in the fiber product, we have $\mathrm{d}\pi(w) = \mathrm{d}\mu(v)$. Thus, identifying the vertical lift in the pullback bundle, the expression becomes:
    \[
        \left( 0, \left(\nabla_{\mathrm{d}\mu(v)} \xi\right)^{\uparrow}\right) = \left(\mu^{*}\left( \nabla_{\mathrm{d}\mu(v)}\xi \right)\right)^{\uparrow}.
    \]
    On the other hand, since the projection of $(v,w)$ to $TP$ is $\mathrm{d}(\mu^*\pi)(v,w) = v$, the definition of the connection via the correspondence \eqref{eq:1-der-linear-tensor} yields $\left( (\mu^{*}\nabla)_{v}(\mu^{*}\xi) \right)^{\uparrow} = \mathscr{L}_{(\mu^{*}\xi)^{\uparrow}}(\mu^{*}R)(v,w)$. Thus, we conclude:
\[
    (\mu^{*}\nabla)_{v}(\mu^{*}\xi) = \mu^{*}\left( \nabla_{\mathrm{d}\mu(v)}\xi \right).
\]
\end{proof}
We now proceed to the proof of Proposition \ref{prop:morita-equiv-1-derivation}. To this end, we recall the characterization of morphisms of related linear tensor fields established in \cite[Theorem 2.1]{bursztyn-drummond-netto2022dirac}. Let $R_{i}\in \Omega^{1}(E_{i},TE_{i})$ be linear $(1,1)$-tensor fields on vector bundles $E_{i}\to M$ for $i=1,2$, and let $\Phi:E_{1}\to E_{2}$ be a vector bundle map covering the identity. Then
\begin{equation}
    \label{eq:related_tensors_derivation}
    \mathrm{d}\Phi \circ R_{1} = R_{2}\circ \mathrm{d}\Phi \qquad \iff \qquad  \Phi\circ \nabla_{v}^{1}= \nabla_{v}^{2}\circ \Phi , \quad \Phi \circ \ell_{1} = \ell_{2}\circ\Phi \quad \text{and}\quad r_{1}=r_{2},
\end{equation}
for every $v\in \mathfrak{X}(M)$.
\begin{proof}[Proof of Proposition \ref{prop:morita-equiv-1-derivation}]
  The definition of Morita equivalence for linear $(1,1)$-tensor fields $R_{1}$ and $R_{2}$ requires, first, that the base tensors $r_{1}$ and $r_{2}$ are $\mu_{1}$-related and $\mu_{2}$-related to $J$, respectively. Equation \eqref{eq:base_1-der_morita} is simply a restatement of this condition.

  The second requirement is that the bundle map $\mathscr{a} = \mathscr{a}_1 + \mathscr{a}_2 : \mu^{*}(A_{1}\times A_{2}) \to TP$ intertwines the tensor fields:
  \[
      \mathrm{d}\mathscr{a} \circ \mu^*(R_{1}\times R_{2}) = J^{\mathrm{tg}}\circ \mathrm{d}\mathscr{a}.
  \]
  We show that this condition is equivalent to Equations \eqref{eq:end_1-der_morita} and \eqref{eq:connection_1-der_morita} by using the identification \eqref{eq:related_tensors_derivation} between related $(1,1)$-tensors and related $1$-derivations. Let $\left(\mu^{*}\nabla,\mu^{*}\ell ,J\right)$ be the $1$-derivation associated with $\mu^{*}(R_{1}\times R_{2})$ and recall that the $1$-derivation corresponding to $J^{\mathrm{tg}}$ is $(\nabla^{J},J,J)$. Due to \eqref{eq:related_tensors_derivation}, the condition that $\mu^{*}(R_{1}\times R_{2})$ is $\mathscr{a}$-related to $J^{\mathrm{tg}}$ is equivalent to:
  \[
    \mathscr{a}\circ\mu^{*}\ell = J\circ \mathscr{a}  \qquad \text{and}\qquad  \mathscr{a}\circ \mu^{*}\nabla  = \nabla^{J}\circ \mathscr{a}.
  \]
  Viewing the action as a map of sections $\mathscr{a}:\Gamma(A_{1}\times A_{2}) \to \mathfrak{X}(P)$, the first identity $\mathscr{a}\circ\mu^{*}\ell = J\circ \mathscr{a}$ yields precisely Equation \eqref{eq:end_1-der_morita}.

  To verify that $\mathscr{a}\circ \mu^{*}\nabla  = \nabla^{J}\circ \mathscr{a}$ corresponds to Equation \eqref{eq:connection_1-der_morita}, it suffices to check it on pullback sections of the form $\mu^{*}(\xi_{1},\xi_{2})$, where $\xi_{1}\in \Gamma(A_{1})$ and $\xi_{2}\in \Gamma(A_{2})$. Using Lemma \ref{lemma:pullback_1-derivations}, we compute:
  \[
    \mathscr{a}\left( (\mu^{*}\nabla)_{v}\left( \mu^{*}(\xi_{1},\xi_{2}) \right) \right) = \mathscr{a}\left( \mu^{*}\left( \nabla_{\mathrm{d}\mu(v)}(\xi_{1},\xi_{2}) \right) \right) = \mathscr{a}\left( \mu^{*}\left( \nabla^{1}_{\mathrm{d}\mu_{1}(v)}(\xi_{1}), \nabla^{2}_{\mathrm{d}\mu_{2}(v)} (\xi_{2}) \right) \right),
  \]
 where in the last equality we used the $C^{\infty}(M_{1}\times M_{2})$-linearity of $\nabla$ in the argument $\mathrm{d}\mu(v)$. Therefore, viewing $\mathscr{a}$ as a map of sections $\Gamma(A_{1}\times A_{2})\to \mathfrak{X}(P)$, the condition $\nabla^{J}\circ\mathscr{a} = \mathscr{a}\circ \mu^{*}\nabla$ becomes:
  \[
    \nabla^{J}_{v}\left( \mathscr{a}(\xi_{1},\xi_{2}) \right) =\mathscr{a}\left( \nabla_{\mathrm{d}\mu(v)}(\xi_{1},\xi_{2}) \right)  = \mathscr{a}_{1}\left( \nabla^{1}_{\mathrm{d}\mu_{1}(v)}\xi_{1} \right)  + \mathscr{a}_{2}\left( \nabla^{2}_{\mathrm{d}\mu_{2}(v)}\xi_{2} \right),
  \]
  which concludes the proof. 
\end{proof}
 
\subsection{Lie correspondence for Morita equivalences} 

In this subsection we show how the Lie functor establishes a bijection between the Morita classes of multiplicative tensor fields on source-simply connected Lie groupoids and Morita classes of IM tensor fields on Lie algebroids. Moreover, it follows immediately that this correspondence respects the Nijenhuis condition of the Morita equivalences.
\begin{theorem}
  \label{thm:morita-nijenhuis-groupoid-to-algebroid}
  Let $\mathscr{G}$ and $\mathscr{H}$ be Lie groupoids with associated Lie algebroids $A_{\mathscr{G}}$ and $A_{\mathscr{H}}$, and let $P$ be a Morita $(\mathscr{G},\mathscr{H})$-bibundle. Let $N_{1}\in \Omega^{1}(\mathscr{G},T\mathscr{G})$ and $N_{2}\in \Omega^{1}\left( \mathscr{H},T\mathscr{H} \right)$ be multiplicative $(1,1)$-tensor fields, and suppose there exists $J\in \Omega^{1}(P,TP)$ such that $N_{1}$ and $N_{2}$ are Morita equivalent via $(P,J)$. Then $N_1$ and $N_2$ differentiate to IM $(1,1)$-tensor fields which are Morita equivalent via the associated infinitesimal Morita $(A_{\mathscr{G}}, A_{\mathscr{H}})$-bibundle $(P,J)$. Moreover, if the global Morita equivalence is Nijenhuis, then the infinitesimal Morita equivalence is also Nijenhuis.
\end{theorem}
\begin{proof}
    Let $\mathscr{A}_{1}:G_{1}\times_{G_{0}}P\to P$ and $\mathscr{A}_{2}:P\times_{H_{0}}H_{1}\to P$ be the right and left actions in the Morita bibundle, respectively. These actions combine into a single action $\mathscr{A}:(G_{1}\times H_{1})\times_{(\mathbf{s},\mu)}P\to P$ of the Lie groupoid $\mathscr{G}\times \mathscr{H}^{\mathrm{op}}$ on the submersion $\mu:P\to G_{0}\times H_{0}$ defined by $\mu\coloneqq (\mu_{1},\mu_{2})$.

Consider the graph of $\mathscr{A}$ as a Lie groupoid $\mathscr{K}=\left( \mathrm{Graph}(\mathscr{A})\rightrightarrows P \right)$, where $\mathrm{Graph}(\mathscr{A}) \subseteq P\times_{(\mu,\mathbf{t})}(G_{1}\times H_{1})\times_{(\mu,\mathbf{s})}P$. The source and target maps are the respective projections onto the second and first $P$-factors; thus, we may view $\mathscr{K}$ as a subgroupoid of $\mathscr{G}\times \mathscr{H}^{\mathrm{op}}\times \mathrm{Pair}(P)$. The condition that the tensor fields $J$ and $(N_{1}\times N_{2})\times J$ intertwine the action (Equation \eqref{eq:morita-tensor-action-2}) implies that $N\coloneqq (N_{1}\times N_{2})\times (J\times J)$ is a morphism from the tangent Lie groupoid $T\mathscr{K}$ to itself. Thus, $N\in \Omega^{1}(\mathscr{K},T\mathscr{K})$ is a multiplicative $(1,1)$-tensor field on $\mathscr{K}$ covering $J$.

The Lie algebroid of $\mathscr{K}=\left( \mathrm{Graph}(\mathscr{A})\rightrightarrows P \right)$ is the graph of the infinitesimal action $\mathscr{a}:\Gamma\left(A_{\mathscr{G}}\times A_{\mathscr{H}}^{\mathrm{op}} \right)\to \mathfrak{X}(P)$; that is,  $A_{\mathscr{K}} = \mathrm{Graph}(\mathscr{a})$. The Lie algebroid $A_{\mathscr{K}}$ is a Lie subalgebroid of $(A_{\mathscr{G}}\times A_{\mathscr{H}}^{\mathrm{op}})\times_{(\rho,\mathrm{d}\mu)} TP\to P$, which integrates to $\mathscr{G}\times \mathscr{H}^{\mathrm{op}}\times \mathrm{Pair}(P)$. Applying the Lie functor to the morphism $N=(N_{1}\times N_{2})\times (J\times J)$ and using Lemma \ref{lemma:integration-tangent-lift} to identify $J^{\mathrm{tg}}$ as the infinitesimal counterpart of $J\times J$, we obtain the Lie algebroid morphism $R_{N}=R_{1}\times R_{2}\times J^{\mathrm{tg}}: TA_{\mathscr{K}}\to TA_{\mathscr{K}}$, where $R_{1}$ and $R_{2}$ are the infinitesimal counterparts of $K_{1}$ and $K_{2}$, respectively. The condition that $R_{N}$ preserves $\mathrm{Graph}(\mathscr{a})$ is equivalent to
\[
  \mathrm{d}\mathscr{a} \circ \mu^{*}\left( R_{1}\times R_{2}\right) = J^{\mathrm{tg}}\circ\mathrm{d}\mathscr{a}.
\]
This equation is one of the requirements of Definition \ref{def:morita-equivalence-linear-tensor} (cf. Remark \ref{rmk:morita-equivalence-linear-tensor}). It remains to observe that the action $\mathscr{a}$ decomposes as $\mathscr{a}=\mathscr{a}_{1}+\mathscr{a}_{2}$, where $\mathscr{a}_{1}$ and $\mathscr{a}_{2}$ are the differentiations of $\mathscr{A}_{1}$ and $\mathscr{A}_{2}$, respectively. By Lemma \ref{lemma:correspondence-morita}, $\mathscr{a}_{1}$ and $\mathscr{a}_{2}$ constitute an infinitesimal Morita $\left( A_{\mathscr{G}},A_{\mathscr{H}} \right)$-bibundle $P$. This completes the proof for multiplicative $(1,1)$-tensor fields.

Regarding the Nijenhuis property, the vanishing of the torsion for the multiplicative tensor fields implies that the IM tensor fields are also Nijenhuis (see \cite[Section 6.2]{bursztyn-drummond2019multiplicative}).
\end{proof}
We now address the converse of Theorem \ref{thm:morita-nijenhuis-groupoid-to-algebroid}.
\begin{theorem}
  \label{thm:morita-nijenhuis-algebroid-to-groupoid}
  Let $\mathscr{G}$ and $\mathscr{H}$ be source-simply connected Lie groupoids with Lie algebroids $A_{\mathscr{G}}$ and $A_{\mathscr{H}}$ respectively, and let $P$ be an infinitesimal Morita $(A_{\mathscr{G}},A_{\mathscr{H}})$-bibundle. Let $R_{1}\in \Omega^{1}(A_{\mathscr{G}},TA_{\mathscr{G}})$ and $R_{2}\in \Omega^{1}(A_{\mathscr{H}},TA_{\mathscr{H}})$ be IM tensor fields that are Morita equivalent via $(P,J)$ for some $J\in \Omega^{1}(P,TP)$. Then the multiplicative $(1,1)$-tensor fields $N_{1}$ and $N_{2}$ that integrate $R_{1}$ and $R_{2}$ are Morita equivalent via the associated Morita $(\mathscr{G},\mathscr{H})$-bibundle $(P,J)$. Furthermore, if the infinitesimal Morita equivalence is Nijenhuis, then the global Morita equivalence is also Nijenhuis.
\end{theorem}
\begin{proof}
    Let $\mathscr{a}_{1} \colon \mu_{1}^{*}A_{\mathscr{G}} \to TP$ and $\mathscr{a}_{2} \colon \mu_{2}^{*}A_{\mathscr{H}} \to TP$ denote the left and right infinitesimal actions of the respective Lie algebroids. The total infinitesimal action is given by $\mathscr{a}(\xi, \zeta) \coloneqq \mathscr{a}_{1}(\xi) + \mathscr{a}_{2}(\zeta)$. Letting $\mu\coloneqq (\mu_{1},\mu_{2})$, Remark \ref{rmk:morita-equivalence-linear-tensor} implies that the tensor fields are Morita equivalent if and only if the image of $T\mathrm{Graph}(\mathscr{a})$ under the map $R \coloneqq \mu^{*}(R_{1} \times R_{2}) \times J^{\mathrm{tg}}$ is contained in $T\mathrm{Graph}(\mathscr{a})$. Because each factor of $R$ is a Lie algebroid morphism, $R$ naturally defines a Lie algebroid endomorphism of $T\mathrm{Graph}(\mathscr{a})$.

    Let $\mathscr{A} \colon G_{1} \times_{G_{0}} P \times_{H_{0}} H_{1} \to P$ denote the combined global action, defined by $\mathscr{A}(g, p, h) = g \cdot p \cdot h$, which integrates $\mathscr{a}$. The source-simply connected Lie groupoid integrating $\mathrm{Graph}(\mathscr{a})$ is the graph of $\mathscr{A}$. We denote this Lie groupoid by $\mathscr{K} \coloneqq (\mathrm{Graph}(\mathscr{A}) \rightrightarrows P)$, which forms a Lie subgroupoid of $\mathscr{G} \times \mathscr{H}^{\mathrm{op}} \times \mathrm{Pair}(P)$. Applying Lie's second theorem, we integrate the Lie algebroid morphism $R$ to a Lie groupoid morphism $N \colon T\mathscr{K} \to T\mathscr{K}$. Thus, $N$ constitutes a multiplicative $(1,1)$-tensor field on $\mathscr{K}$.

  Let $N_{1} \in \Omega^{1}(\mathscr{G}, T\mathscr{G})$ and $N_{2} \in \Omega^{1}(\mathscr{H}, T\mathscr{H})$ be the respective integrations of the IM $(1,1)$-tensor fields $R_{1}$ and $R_{2}$. Since Lemma \ref{lemma:integration-tangent-lift} guarantees that $J^{\mathrm{tg}}$ integrates to $J \times J$, the uniqueness of integration implies that $N$ coincides with the restriction of $(N_{1} \times N_{2}) \times (J\times J)$ to $\mathrm{Graph}(\mathscr{A})$.

    The condition that $(N_{1} \times N_{2}) \times (J\times J)$ preserves $\mathrm{Graph}(\mathscr{A})$ is equivalent to the Morita $(\mathscr{G},\mathscr{H})$-bibundle $(P, J)$ establishing a Morita equivalence between the multiplicative tensor fields $N_{1}$ and $N_{2}$. Because the vanishing of the Nijenhuis torsion is preserved under integration (see \cite[Section 6.2]{bursztyn-drummond2019multiplicative}), the result follows.
\end{proof}

Considering the one-to-one correspondence between integrable Lie algebroids and source-simply connected Lie groupoids, from Theorem \ref{thm:morita-nijenhuis-groupoid-to-algebroid} and Theorem \ref{thm:morita-nijenhuis-algebroid-to-groupoid} we obtain the following corollary.
\begin{corollary}
    \label{cor:corresp-morita-nijenhuis}
    The Lie functor establishes a one-to-one correspondence between Morita equivalences of multiplicative $(1,1)$-tensor fields on source-simply connected Lie groupoids and Morita equivalences of IM $(1,1)$-tensor fields. Moreover, Morita equivalences of source-simply connected Nijenhuis groupoids correspond bijectively to Morita equivalences of infinitesimal Nijenhuis structures.
\end{corollary}
In particular, the following is a consequence of Theorem \ref{thm:morita-equiv-relation}. 
\begin{corollary}
  Morita equivalence of infinitesimal Nijenhuis structures is an equivalence relation. 
\end{corollary}
 
\section{Compatibility with quasi-symplectic groupoids and Dirac structures}
\label{sec:symplectic-nijenhuis}
In this section, we apply the framework of Morita equivalence for multiplicative $(1,1)$-tensor fields, developed in Section \ref{sec:nijenhuis-groupoids}, to Lie groupoids equipped with geometric structures compatible with the Nijenhuis condition.  Building upon the established notion of Morita equivalence for quasi-symplectic groupoids introduced by Xu \cite{xu2004momentum}, our primary objective is to extend this definition to \textit{quasi-symplectic-Nijenhuis groupoids} and prove that it constitutes an equivalence relation (Theorem \ref{thm:morita-equiv-rel-symp-nij}).
 
Subsequently, we investigate the infinitesimal counterpart of this theory as an application of Section \ref{sec:nijenhuis-lie-algebroids}. Recall that a Morita equivalence of quasi-symplectic groupoids differentiates to a dual pair of Dirac structures satisfying additional conditions \cite{frejlich-marcut2018dualpairs}. We enhance this correspondence by imposing compatibility with infinitesimal Nijenhuis structures, leading to a notion of Morita equivalence for Dirac-Nijenhuis structures (Definition \ref{def:morita-dirac-nijenhuis}). Finally, we establish the consistency of this framework via the Lie functor: Theorems \ref{thm:symp-nij-morita-dirac-nij} and \ref{thm:morita-diracnij-to-sympnij} prove the global-to-infinitesimal correspondence.
\subsection{Quasi-symplectic-Nijenhuis groupoids}
We begin by recalling that a \textbf{quasi-symplectic groupoid} \cite{xu2004momentum,BCWZ2004integration} (called \textit{twisted presymplectic groupoid in \cite{BCWZ2004integration}}) is a triple $(\mathscr{G},\omega,\eta)$, where $\mathscr{G}=\left( G_{1}\rightrightarrows G_{0} \right)$ is a Lie groupoid such that $\operatorname{dim}G_{1}=2\operatorname{dim}G_{0}$, $\omega\in \Omega^{2}(G_{1}) $ is a multiplicative $2$-form and  $\eta\in  \Omega^{3}(G_{0})$ is a closed $3$-form satisfying 
\begin{equation}
  \label{eq:quasi-symplectic}
  \mathrm{d}\omega = \mathbf{s}^{*}\eta - \mathbf{t}^{*}\eta\qquad \text{and}\qquad \operatorname{ker}\omega \cap \operatorname{ker} \mathrm{d}\mathbf{s} \cap \operatorname{ker} \mathrm{d}\mathbf{t} = 0.
\end{equation}
Although Morita equivalence for quasi-symplectic groupoids can be formulated in terms of Morita spans (see \cite{mayrand2023shifted,cueca2023shifted}), we give preference to the original definition \cite{xu2004momentum} here. This is because the formulation via spans does not naturally equip the intertwining groupoid with a multiplicative $2$-form, a feature which is essential for stating the compatibility with the Nijenhuis tensor field in our framework.

Recall from \cite{xu2004momentum} (see also \cite[Appendix A]{alekseev-meinrenken2024coadjoint}) that a \textbf{Morita equivalence of quasi-symplectic groupoids} $\left(\mathscr{G},\omega_{1},\eta_{1}\right)$ and $\left( \mathscr{H},\omega_{2}, \eta_{2} \right)$ is given by a Morita bibundle $G_{0}\leftarrow P\to H_{0}$ equipped with a $2$-form $\varpi\in \Omega^{2}(P)$, 
\[
\begin{tikzcd}
    {(G_{1},\omega_1)} \arrow[d, shift right] \arrow[d, shift left] & {(P,\varpi)} \arrow[ld, "\mu_{1}" '] \arrow[rd, "\mu_{2}"] & {(H_1,\omega_2)} \arrow[d, shift right] \arrow[d, shift left] \\
{(G_0,\eta_1)}                                                  &                                    & (H_0,\eta_2),                                               
\end{tikzcd}
\]
satisfying the following conditions:
\begin{enumerate}[label=(\alph*)]
    \item $\mathrm{d}\varpi = \mu_{2}^{*}\eta_{2}- \mu^{*}_{1}\eta_{1}$;            \item $\ker \varpi \cap \operatorname{ker} \mathrm{d}\mu_{1}\cap \operatorname{ker}\mathrm{d}\mu_{2}=0$;
    \item $\mathscr{A}^{*}\varpi = \operatorname{pr}_{G_1}^* \omega_1  + \operatorname{pr}_P^*\varpi + \operatorname{pr}_{H_1}^* \omega_2$,
\end{enumerate}
where $\mathscr{A}\colon G_1\times_{G_{0}}P\times_{H_{0}}H_{1}\to P$ is the action $(g,p,h)\mapsto g\cdot p\cdot h$. This definition extends the notion of Morita equivalence for symplectic groupoids introduced by Xu \cite{xu1991symplectic}. 
\begin{remark}
   \label{rmk:morita-equiv-relation-symplectic}
   Morita equivalence of quasi-symplectic groupoids is indeed an equivalence relation \cite[Theorem 4.5]{xu2004momentum}. For transitivity, given two principal bibundles $G_{0}\gets (P_{1},\varpi_{1})\to  H_{0}$ and $H_{0}\gets (P_{2},\varpi_{2})\to  G_{0}'$, a $2$-form $\varpi\in \Omega^{2}(P)$ is obtained on the quotient $P= P_{1}\diamond P_{2} = \left(P_{1}\times_{H_{0}}P_{2}\right) / \mathscr{H}$, realizing a Morita equivalence $G_{0}\gets (P,\varpi)\to G_{0}'$. The $2$-form $\varpi$ is defined by the property
\begin{equation}
  \label{eq:reduction-symplectic}
  \pi^{*}\varpi = \mathrm{pr}_{1}^{*}\varpi_{1} + \mathrm{pr}_{2}^{*}\varpi_{2},
\end{equation}
where $\pi:P_{1}\times_{H_{0}} P_{2}\to P$ is the quotient map and $\mathrm{pr}_{i}:P_{1}\times_{H_{0}} P_{2} \to P_{i}$ are the projections. This construction will be used when we extend this relation to \textit{quasi-symplectic-Nijenhuis groupoids}.
\end{remark}
To define quasi-symplectic-Nijenhuis groupoids we require a compatibility between $2$-forms and $(1,1)$-tensor fields.  Consider the following notation. For any $p$-form $\alpha \in \Omega^{p}\left( M \right)$ and vector-valued $1$-form $K\in \Omega^{1}(M,TM)$, write $\alpha_{K}$ to denote
\[
    \alpha_{K}(u_{1},u_{2},\ldots,u_{p}) \coloneqq \alpha\left( K ( u_{1} ),u_{2},\ldots,u_{p} \right).
\]
    A $2$-form $\omega\in \Omega^{2}(M)$ and a $(1,1)$-tensor field $K\in \Omega^{1}(M,TM)$ on a manifold $M$ are said to be \textbf{compatible} \cite{magri1984geometrical} if
    \begin{align}
        \label{eq:compatible-2-form-1}
  &\omega^{\flat}  \circ K  = K^{*} \circ \omega^{\flat} \qquad \text{and}\\
  \label{eq:compatible-2-form-2}
  &\mathrm{d}(\omega_{K}) = (\mathrm{d}\omega)_{K}.
    \end{align}
    In this case, $(\omega,K)$  is called a \textbf{compatible pair}.
\begin{definition}
  \label{def:quasi-symplectic-nijenhuis-groupoids}
  A \textbf{quasi-symplectic-Nijenhuis groupoid} is a quadruple $(\mathscr{G},\omega,\eta,N)$, where $(\mathscr{G},\omega,\eta)$ is a quasi-symplectic groupoid and $N\in \Omega^{1}(\mathscr{G},T\mathscr{G})$ is a multiplicative Nijenhuis tensor field that is compatible with the $2$-form $\omega$.
\end{definition}
\begin{example}
  \label{eg:holomorphic-symplectic-groupoid}
  Following the terminology of \cite{bursztyn-drummond-netto2022dirac}, a \textit{holomorphic presymplectic groupoid} is a holomorphic groupoid $(\mathscr{G},I_{\mathscr{G}})$ equipped with a multiplicative holomorphic closed $2$-form $\Omega \in \Omega^{(2,0)}(G_{1})$ that satisfies $\operatorname{ker}\Omega\cap \operatorname{ker}\mathrm{d}\mathbf{s}\cap \operatorname{ker}\mathrm{d}\mathbf{t}=0$. Equivalently, a holomorphic presymplectic groupoid is described as a quasi-symplectic-Nijenhuis groupoid $(\mathscr{G},\omega, 0 ,I_{\mathscr{G}})$, where $\mathrm{d}\omega=0$ and $\Omega = \omega - i \omega_{I_{\mathscr{G}}}$ \cite[Section 6.3]{bursztyn-drummond-netto2022dirac}. 
\end{example}

\begin{definition}
  \label{def:morita-equivalence-symplectic-nijenhuis}
  A \textbf{Morita equivalence of quasi-symplectic-Nijenhuis groupoids} $(\mathscr{G},\omega_{1},\eta_{1},N_{1})$ and $(\mathscr{H},\omega_{2},\eta_{2},N_{2})$ is given by a Morita $(\mathscr{G},\mathscr{H})$-bibundle $P$ equipped with a $2$-form $\varpi\in \Omega^{2}(P)$ and a $(1,1)$-tensor field $J\in \Omega^{1}(P,TP)$,
\[
\begin{tikzcd}
    (G_{1},\omega_1,N_{1}) \arrow[d, shift right] \arrow[d, shift left] & (P,\varpi,J) \arrow[ld, "\mu_{1}" '] \arrow[rd, "\mu_{2}"] & (H_1,\omega_2,N_{2}) \arrow[d, shift right] \arrow[d, shift left] \\
(G_0,\eta_1,r_{1})                                                  &                                    & (H_0,\eta_2,r_{2}),                                               
\end{tikzcd}
\]
satisfying the following conditions:
\begin{enumerate}[label=(\alph*)]
      \item $(P,\varpi)$ is a Morita equivalence of quasi-symplectic groupoids;
      \item $( P,J )$ is a Morita equivalence of Nijenhuis groupoids, and
      \item $(\varpi,J)$ is a compatible pair.
  \end{enumerate}
\end{definition}
\begin{example}
  \label{eg:holomorphic-morita-symplectic}
  Two holomorphic presymplectic groupoids $(\mathscr{G},\Omega_{1},I_{\mathscr{G}})$ and $(\mathscr{H},\Omega_{2},I_{\mathscr{H}})$ are Morita equivalent if and only if, there is a holomorphic Morita bibundle $G_{0}\xleftarrow{\ \mu_{1} \ } (P,I_{P}) \xrightarrow{\ \mu_{2} \ } H_{0}$ (Example \ref{eg:holomorphic-groupoids-morita}) and a holomorphic closed $2$-form $\Omega_{P} = \varpi - i \varpi_{I_{P}} \in \Omega^{(2,0)}(P)$ satisfying $\mathscr{A}^{*}\varpi = \mathrm{pr}_{G_{1}}^{*}\omega_{1} + \mathrm{pr}_{P}^{*}\varpi + \mathrm{pr}_{H_{1}}^{*}\omega_{2}$ and $\operatorname{ker}\varpi \cap \operatorname{ker}\mathrm{d}\mu_{1}\cap \operatorname{ker}\mathrm{d}\mu_{2}=0$, where $\Omega_{1}=\omega_{1}-i\omega_{1\, I_{\mathscr{G}}}$ and $\Omega_{2}=\omega_{2}-i \omega_{2\, I_{\mathscr{H}}}$.
\end{example}
 
\begin{theorem}
  \label{thm:morita-equiv-rel-symp-nij}
  Morita equivalence of quasi-symplectic-Nijenhuis groupoids is an equivalence relation.
\end{theorem}
\begin{proof} 
    Let $(P_{1},\varpi_{1},J_{1})$ and $(P_{2},\varpi_{2},J_{2})$ be Morita $(\mathscr{G},\mathscr{H})$- and $(\mathscr{H},\mathscr{G}')$-bibundles, respectively, realizing Morita equivalences between the quasi-symplectic--Nijenhuis groupoids $(\mathscr{G},\omega_{1},\eta_{1},N_{1})$, $(\mathscr{H},\omega_{2},\eta_{2},N_{2})$, and $(\mathscr{G}',\omega_{3},\eta_{3},N_{3})$. Let $P = P_{1}\diamond P_{2}$ denote the quotient of the fiber product $P_{1}\times_{H_{0}}P_{2}$ by the diagonal action of $\mathscr{H}$, and let $\pi:P_{1}\times _{H_{0}}P_{2}\to P$ be the natural projection. The proof relies on the following results:
\begin{enumerate}
    \item Theorem \ref{thm:morita-equiv-relation} establishes that $(P,J_{1}\diamond J_{2})$ constitutes a Morita bibundle between the Nijenhuis groupoids $(\mathscr{G},N_{1})$ and $(\mathscr{G}',N_{3})$, where the Nijenhuis tensor field $J_{1}\diamond J_{2} \in \Omega^{1}(P,TP)$ is the projection of $J_{1}\times J_{2}$. 
    \item By \cite[Theorem 4.5]{xu2004momentum}, the pair $(P,\varpi)$ defines a Morita bibundle between the quasi-symplectic groupoids $(\mathscr{G},\omega_{1},\eta_{1})$ and $(\mathscr{G}',\omega_{3},\eta_{3})$, where the $2$-form $\varpi\in \Omega^{2}(P)$ is defined by Equation \eqref{eq:reduction-symplectic}.
\end{enumerate}
It remains only to prove the compatibility between the $(1,1)$-tensor field $J_{1}\diamond J_{2}$ and the $2$-form $\varpi$ on the quotient manifold $P$. 

\textit{Compatibility on the fiber product.} First, we show that the compatibility holds for the $(1,1)$-tensor field $J_{1}\times J_{2} \in \Omega^{1}\left( P_{1}\times_{H_{0}} P_{2},T(P_{1}\times P_{2}) \right)$ and the $2$-form $\pi^{*}\varpi = \mathrm{pr}_{1}^{*}\varpi_{1}+\mathrm{pr}_{2}^{*}\varpi_{2} \in \Omega^{2}\left( P_{1}\times_{H_{0}} P_{2} \right)$. To verify the first condition, Equation \eqref{eq:compatible-2-form-1}, let $(v_{1},v_{2})$ be a vector field on $P_{1}\times_{H_{0}} P_{2}$. Using the fact that $(\varpi_{1},J_{1})$ and $(\varpi_{2},J_{2})$ are compatible pairs, we have
 \begin{align*}
\left( \mathrm{pr}_{1}^{*}\varpi_{1}+\mathrm{pr}_{2}^{*}\varpi_{2} \right)^{\flat} \left( J_{1}\times J_{2}(v_{1},v_{2}) \right)  & = \varpi_{1}^{\flat}\left( J_{1}(v_{1}) \right)\circ \mathrm{pr}_{1} + \varpi_{2}^{\flat}\left( J_{2}(v_{2}) \right)\circ \mathrm{pr}_{2}\\
                                                                                                                                       & = J_{1}^{*}\left( \varpi_{1}^{\flat}(v_{1}) \right)\circ \mathrm{pr}_{1} + J_{2}^{*} \left( \varpi_{2}^{\flat}(v_{2}) \right)\circ\mathrm{pr}_{2}\\
                                                                                                                                       & =J_{1}^{*} \times J_{2}^{*} \left( \mathrm{pr}_{1}^{*}\varpi_{1}^{\flat} + \mathrm{pr}_{2}^{*} \varpi_{2}^{\flat}\right) \left( v_{1},v_{2} \right).
 \end{align*}
This shows that $(\pi^{*}\varpi)^{\flat}\circ (J_{1}\times J_{2})= (J_{1}\times J_{2})^{*} \circ (\pi^{*}\varpi)^{\flat} $. To verify the second condition, Equation \eqref{eq:compatible-2-form-2}, we observe that the Lie bracket of vector fields on $P_{1}\times P_{2}$ decomposes component-wise:  $\left[ (v_{1},v_{2}) , (u_{1},u_{2}) \right]=\left( \left[ v_{1},v_{2} \right] , \left[ u_{1},u_{2} \right]\right)$. Consequently, computing the exterior derivative of $\left( \pi^{*}\varpi \right)_{J_{1}\times J_{2}}$ and comparing it with that of $ \left( \pi^{*}\varpi\right)$ yields
 \begin{equation}
     \label{eq:equation_in_proof}
  \mathrm{d}\left(\left( \pi^{*}\varpi \right)_{J_{1}\times J_{2}}\right) = \left( \mathrm{d}\left(\pi^{*}\varpi \right) \right)_{J_{1}\times J_{2}}.
\end{equation}
Thus, the pair $(\pi^{*}\varpi,J_{1}\times J_{2})$ is compatible.

\textit{Descent to the quotient.} Finally, these compatibility relations descend to the quotient $P$. Consider a vector field $\bar{v} \in \mathfrak{X}(P)$ as the image under $\mathrm{d}\pi$ of some $\pi$-projectable vector field $v=(v_{1},v_{2}) \in \mathfrak{X}\left( P_{1}\times_{H_{0}}P_{2} \right)$. Recall that $J_{1}\times J_{2}$ preserves projectable vector fields and that it is $\pi$-related to $J_{1}\diamond J_{2}$. Using \eqref{eq:reduction-symplectic} and evaluating at $\bar{v} =\mathrm{d}\pi(v)$, we find: 
\begin{align*}
    \varpi^{\flat} \circ (J_{1}\diamond J_{2}) \left( \mathrm{d} \pi (v) \right) & = \varpi^{\flat} \circ \mathrm{d}\pi \left( J_{1}\times J_{2} \left( v \right)\right)\\
                                                                 & = \pi^{*}\varpi^{\flat} (J_{1}\times J_{2} (v))\\
                                                                 & =  J_{1}^{*}\varpi_{1}^{\flat}(v_{1}) + J_{2}^{*}\varpi_{2}^{\flat}(v_{2})\\
                                                                 & = \left(  J_{1}\times J_{2} \right)^{*}(\pi^{*}\varpi)^{\flat}(v)  \\
                                                                 & = (J_{1}\diamond J_{2})^{*}\circ \varpi^{\flat} (\mathrm{d}\pi( v) ).
\end{align*}
Therefore $\varpi^{\flat}\circ (J_{1}\diamond J_{2}) = (J_{1}\diamond J_{2})^{*} \circ\varpi^{\flat}$. Similarly, using \eqref{eq:equation_in_proof} and the naturality of the exterior derivative, we obtain:
\begin{align*}
    \left( \mathrm{d}\varpi \right)_{J_{1}\diamond J_{2}}\left( \mathrm{d}\pi (u),\mathrm{d}\pi (v),\mathrm{d}\pi(w)\right)  & = \mathrm{d}\varpi \left( J_{1}\diamond J_{2}\left( \mathrm{d}\pi (u) \right), \mathrm{d}\pi(v), \mathrm{d}\pi(w) \right)\\
                                                                                                                             & =\pi^{*} \mathrm{d}\varpi \left( J_{1}\times J_{2}(u) , u ,w \right)\\
                                                                                                                             & = \mathrm{d}\left( \mathrm{pr}_{1}^{*}\varpi_{1}+\mathrm{pr}_{2}^{*}\varpi \right)\left( J_{1}\times J_{2}(u), v, w\right)\\
                                                                                                                             & = \mathrm{d}\left( \left( \pi^{*}\varpi \right)_{J_{1}\times J_{2}} \right) \left( u,v,w\right) .
\end{align*}
Computing the exterior derivative of $\left( \pi^{*}\varpi \right)_{J_{1}\times J_{2}}$ and using the naturality of the Lie bracket of vector fields, we get:
\begin{align*}
    \mathrm{d}\left( \left( \pi^{*}\varpi \right)_{J_{1}\times J_{2}} \right) \left( u,v,w \right)  & = \mathrm{d}\left( \varpi_{\mathrm{d}\pi\left( J_{1}\times J_{2} \right)} \right)  \left( \mathrm{d}\pi(u), \mathrm{d}\pi(v), \mathrm{d}\pi(w) \right) \\
                                                                                                                & =\mathrm{d}\left( \varpi_{J_{1}\diamond J_{2}} \right)   \left( \mathrm{d}\pi(u), \mathrm{d}\pi(v), \mathrm{d}\pi(w) \right).
\end{align*}
Therefore $\mathrm{d}(\varpi_{J_{1}\diamond J_{2}}) = (\mathrm{d}\varpi)_{J_{1}\diamond J_{2}}$. This completes the proof.
\end{proof}
 
\subsection{Twisted Dirac-Nijenhuis structures}
The infinitesimal counterparts of quasi-symplectic groupoids are \textbf{$\eta$-twisted Dirac structures}. Recall that a $\eta$-twisted Dirac structure is a Lagrangian subbundle $L\subseteq \mathbb{T}M \coloneqq TM\oplus T^{*}M$ whose space of sections is closed under the \textbf{$\eta$-Courant bracket}:
\begin{equation}
    \label{eq:courant-bracket}
    \llbracket (u,\alpha),(v,\beta)\rrbracket_{\eta}\coloneqq \left( \left[ u,v \right],\mathscr{L}_{u}\beta - \iota_{v}\mathrm{d}\alpha + \iota_{u \wedge v}\eta\right), \quad \eta \in \Omega^{3}(M).
\end{equation}
Every twisted Dirac structure $L$ is intrinsically a Lie algebroid, where the anchor is the projection to $TM$ and the bracket is the restriction of the $\eta$-Courant bracket. The conditions for when an arbitrary Lie algebroid is isomorphic to a Dirac structure are detailed in \cite{BCWZ2004integration}.

\subsubsection{Integrating Dirac structures.}
The precise correspondence between twisted Dirac structures and quasi-symplectic groupoids is established in \cite[Theorem 2.2]{BCWZ2004integration} and \cite[Theorem 2.4]{BCWZ2004integration}: an integrable Lie algebroid $A_{\mathscr{G}}$ is isomorphic to an $\eta$-twisted Dirac structure $L$ on its base if and only if its source-simply connected integration $\mathscr{G}$ admits a quasi-symplectic structure. This correspondence requires the target map to be a \textit{forward Dirac map}. A map $\mu:(M,L_M) \to (N,L_N)$ is said to be a \textbf{forward Dirac map } if for every $p\in M$
\begin{equation}
    \label{eq:forward-dirac}
    (L_N)_{\mu(p)} \coloneqq \left\{ (\mathrm{d}\mu(v), \alpha) \in \mathbb{T}_{\mu(p)}N \mid v \in T_pM, (v,(\mathrm{d}\mu)^{*}(\alpha)) \in L_M \right\}.
\end{equation}

\subsubsection{Twisted Dirac-Nijenhuis structures} We now turn to the compatibility with infinitesimal Nijenhuis structures that corresponds to quasi-symplectic-Nijenhuis groupoids. These infinitesimal objects are called \textit{twisted Dirac-Nijenhuis structures}. This compatibility between $(1,1)$-tensor fields and (twisted) Dirac structures corresponds to the compatibility of the IM $(1,1)$-tensor field with the Lie algebroid structure. We use the notation $\mathbb{D}^{r}_{v} \coloneqq  ( \nabla^{r}_{v},\nabla^{r,*}_{v} )$, where $\nabla^{r}$ and $\nabla^{r,*}$ are as defined in Example \ref{eg:tangent-cotangent-lift}.

We recall from \cite{bursztyn-drummond-netto2022dirac} a generalization of the compatibility for pairs given in Equations \eqref{eq:compatible-2-form-1} and \eqref{eq:compatible-2-form-2}.  Let $L\subseteq TM\oplus T^{*}M$ be a subbundle. A $(1,1)$-tensor field $r\in \Omega^{1}(M,TM)$ is said to be \textbf{compatible} with $L$ if
\begin{equation}
\label{eq:lagrangian-compatible} \left(r,r^{*}\right)(L) \subseteq L \qquad \text{and}\qquad \mathbb{D}^{r}_{v}\left( \Gamma(L) \right)\subseteq \Gamma(L), \end{equation}
for every $v \in \mathfrak{X}(M)$.
\begin{definition}[\cite{bursztyn-drummond-netto2022dirac}]
    \label{def:dirac-nijenhuis}
    A \textbf{Dirac-Nijenhuis structure} on $M$ is a pair $(L,r)$, where $L\subseteq TM\oplus T^{*}M$ is a (twisted) Dirac structure and $r\in \Omega^{1}(M,TM)$ a Nijenhuis tensor field compatible with $L$.
\end{definition}
\begin{example}[Poisson-Nijenhuis manifold \cite{magri1984geometrical,kosmann1990poisson}]
  \label{eg:poisson-nijenhuis}
  Let $(M,\pi)$ be a Poisson manifold and $r\in \Omega^{1}(M,TM)$ a Nijenhuis tensor field. Consider the Dirac structure $L=\mathrm{Graph}(\pi^{\sharp})$. Then $(L,r)$ is Dirac-Nijenhuis if and only if
\begin{align}
    & \pi^{\sharp}\circ r^{*} = r\circ \pi^{\sharp}  \label{eq:poisson-nijenhuis1} \\
    &R_{\pi}^{r}\left(v,\alpha  \right) \coloneqq \pi^{\sharp}\left( \mathscr{L}_{v}r^{*}(\alpha) - \mathscr{L}_{r(v)}\alpha \right)-\left( \mathscr{L}_{\pi^{\sharp}(\alpha)}r \right)(v)=0, \label{eq:poisson-nijenhuis2}
\end{align} 
for every $v \in \mathfrak{X}(M)$ and $\alpha\in \Omega^{1}(M)$. Under condition \eqref{eq:poisson-nijenhuis1}, $R^{r}_{\pi}:\mathfrak{X}(M)\times \Omega^{1}(M)\to \mathfrak{X}(M)$ is $C^{\infty }(M)$-bilinear; it is called the \textit{Magri-Morosi concomitant}. Equations \eqref{eq:poisson-nijenhuis1} and \eqref{eq:poisson-nijenhuis2} are precisely the conditions for the triple $(M,\pi,r)$ to be a \textit{Poisson-Nijenhuis manifold}.
\end{example}

The relation between $(1,1)$-tensor fields $r\in \Omega^{1}(M,TM)$ compatible with a Dirac structure $L\subseteq \mathbb{T}M$ and the Lie algebroid defined by $L$ is provided by \cite[Lemma 6.1]{bursztyn-drummond-netto2022dirac}: if $r \in \Omega^{1}(M,TM)$ is compatible with $L$, then the $1$-derivation $\left( \mathbb{D}^{r}|_{\Gamma(L)}, (r,r^{*}), r \right)$ is a IM $1$-derivation on the Lie algebroid $L$. Moreover, if $r$ is Nijenhuis then the corresponding $1$-derivation is Nijenhuis.
\begin{corollary}
  \label{cor:dirac-compatibility-algebroid}
  Let $L\subseteq \mathbb{T}M$ be a Dirac structure. If $r\in \Omega^{1}(M,TM)$ is a $(1,1)$-tensor field compatible with $L$, then $(r^{\mathrm{tg}}, r^{\mathrm{cotg}})$ is a IM $(1,1)$-tensor field on the Lie algebroid $L\to M$. Moreover, if $r$  is Nijenhuis then $(L\to M,(r^{\mathrm{tg}}, r^{\mathrm{cotg}}))$ is an infinitesimal Nijenhuis structure.
\end{corollary}
\begin{proof}
  See \cite[Section 3.3]{bursztyn-drummond-netto2022dirac}.
\end{proof}
\subsubsection{Integrating Dirac-Nijenhuis structures.} The following theorem is an adaptation of \cite[Theorem 6.3]{bursztyn-drummond-netto2022dirac} to the twisted case where $\eta\neq 0$. The proof is similar.
\begin{theorem}
  \label{thm:symplectic-nijenhuis-integration}
  Let $(\mathscr{G},\omega,\eta)$ be a quasi-symplectic groupoid integrating the $\eta$-twisted Dirac structure $L\subseteq \mathbb{T}M$. Let $N \in \Omega^{1}(\mathscr{G},T\mathscr{G})$ be a Nijenhuis $(1,1)$-tensor field integrating the $1$-derivation $\left( D,\ell ,r \right)$. If $\omega$ and $N$ are compatible (Equations \eqref{eq:compatible-2-form-1} and \eqref{eq:compatible-2-form-2}), then $L$ and $r$ are compatible in the sense of Equations \eqref{eq:lagrangian-compatible}, and the converse holds when $\mathscr{G}$ is source connected.
\end{theorem}
When $\mathscr{G}$ is source-simply-connected, the previous theorem sets  a bijective correspondence between quasi-symplectic-Nijenhuis structures on $\mathscr{G}$ and Dirac-Nijenhuis structures on $M$.
\subsection{Global and infinitesimal Morita equivalences of compatible structures}
In this subsection, we introduce a notion of Morita equivalence for Dirac-Nijenhuis structures which corresponds to the infinitesimal counterpart of Morita equivalence for quasi-symplectic Nijenhuis groupoids. To do so, we first need to review the Morita equivalence of Dirac manifolds themselves. Similarly to how a Morita equivalence of Poisson manifolds can be expressed in terms of symplectic realizations, Morita equivalences of Dirac structures are expressed in terms of the so-called \textit{Dirac realizations}, or more precisely a \textit{dual pair}, a concept detailed in \cite{frejlich-marcut2018dualpairs}.

In the following, we fix two twisted Dirac manifolds $(M_{1},L_{1},\eta_{1})$ and $(M_{2},L_{2},\eta_{2})$ and consider a manifold $P$ equipped with a $2$-form $\varpi\in \Omega^{2}(P)$. Consider surjective submersions $\mu_{1}:P\to M_{1}$ and $\mu_{2}:P\to M_{2}$, define $\eta\coloneqq \mu_{2}^{*}\eta_{2}-\mu_{1}^{*}\eta_{1}$, and equip $P$ with the $\eta$-twisted Dirac structure $L_{\varpi}$. We denote the opposite Dirac structure on $M_2$ by $\bar{L}_{2}\coloneqq \{ (v,-\alpha)\mid (v,\alpha)\in L_{2} \}$. Furthermore, let $\mathscr{a}_{i}:\Gamma(L_{i})\to \mathfrak{X}(P)$ be the infinitesimal actions uniquely determined by the condition $\mathscr{a}_{i}\left( \mathrm{d}\mu_{i}(v), \alpha \right) = v$, where $v\in \mathfrak{X}(P)$ and $(\mathrm{d}\mu_{i}(v), \alpha) \in \Gamma(L_i)$.

\begin{definition}
  \label{def:morita-equiv-dirac}
  The Dirac manifolds $(M_{1},L_{1},\eta_{1})$ and $(M_{2},L_{2},\eta_{2})$ are said to be \textbf{Morita equivalent} if there exists a manifold $P$ equipped with a $2$-form $\varpi \in \Omega^{2}(P)$, and two surjective submersions
  \[
    (M_{1},L_{1})\xleftarrow{\quad \mu_{1}\quad} (P,\varpi) \xrightarrow{\quad \mu_{2}\quad } (M_{2},\bar{L}_{2}),
  \]
  satisfying the following condition:
  \begin{enumerate}[label=(\alph*)]
    \item The maps $\mu_{1}$ and $\mu_{2}$ are forward Dirac maps and their fibers are simply connected;
    \item The actions $\mathscr{a}_{1}$ and $\mathscr{a}_{2}$ are complete;
    \item $\operatorname{ker}\varpi \cap \operatorname{ker}\mathrm{d}\mu_{1}\cap \operatorname{ker}\mathrm{d}\mu_{2} = 0$; \label{item:non-deg}
    \item $\varpi\left( \operatorname{ker}\mathrm{d}\mu_{1},\operatorname{ker}\mathrm{d}\mu_{2} \right)=0$; \label{item:orthogonal}
    \item $\mathrm{d}\varpi = \eta$.
  \end{enumerate}
  In this setting, we call $(P,\varpi)$ an \textbf{infinitesimal Morita $(L_{1},L_{2})$-bibundle}. 
\end{definition}
\begin{example}[\cite{xu1991morita}]
  \label{eg:poisson-morita}
  Two Poisson manifolds $(M_{1},\pi_{1})$ and $(M_{2},\pi_{2})$ are Morita equivalent if there exists a symplectic manifold $(P,\varpi)$ and two Poisson morphisms $\mu_{1}:(P,\varpi)\to (M_{1},\pi_{1})$ and $\mu_{2}:(P,\varpi)\to (M_{2},-\pi_{2})$ that are also surjective submersions. Furthermore, the fibers of these maps must be connected and simply connected, and mutually $\varpi$-orthogonal. In this case, the graphs $L_{1}\coloneqq \mathrm{Graph}(\pi_{1})$ and $L_{2}\coloneqq \mathrm{Graph}(\pi_{2})$ are Morita equivalent Dirac manifolds. 
\end{example}
 
The following proposition justifies Definition \ref{def:morita-equiv-dirac} by showing that the Lie algebroids associated with  two Morita equivalent twisted Dirac structures are Morita equivalent as Lie algebroids. Consequently, we distinguish between these contexts by explicitly referring to an \textit{infinitesimal Morita bibundle of Dirac structures} for the setting of Definition \ref{def:morita-equiv-dirac}, and an \textit{infinitesimal Morita bibundle of Lie algebroids} when viewing $L_{1}$ and $L_{2}$ as Lie algebroids as in Definition \ref{def:morita-equiv-algebroid}.

\begin{proposition}
  \label{prop:morita-dirac-algebroid-same}
If two twisted Dirac manifolds are Morita equivalent via an infinitesimal Morita $(L_{1},L_{2})$-bibundle $(P,\varpi)$, then their respective Lie algebroids are Morita equivalent in the sense of Definition \ref{def:morita-equiv-algebroid} with $P$ serving as the infinitesimal Morita bibundle.
\end{proposition}
\begin{proof}
    Let $(M_{1},L_{1},\eta_{1})$  and $(M_{2},L_{2},\eta_{2})$ be Morita equivalent Dirac structures. We also denote their respective Lie algebroids by $L_{1}$ and $L_{2}$. Condition \ref{item:non-deg} in Definition \ref{def:morita-equiv-dirac} guarantees a well-defined map
\begin{equation}
   \label{eq:action-dirac-morita}
    \mathscr{a}:\Gamma(L_{1})\oplus\Gamma(L_{2})\to \mathfrak{X}(P), \quad (\widetilde{v}_{1},\alpha_{1})\oplus(\widetilde{v}_{2},\alpha_{2}) \mapsto v, 
\end{equation}
where $v\in \mathfrak{X}(P)$ is the unique vector field such that $\mathrm{d}\mu_{1}(v)=\widetilde{v}_{1}$, $\mathrm{d}\mu_{2}(v)=\widetilde{v}_{2}$ and $\varpi^{\flat}(v) = \mu_{1}^{*}\alpha_{1} -\mu_{2}^{*}\alpha_{2}$. We define the actions
  \[
      \mathscr{a}_{1}:\Gamma(L_{1})\to \mathfrak{X}(P)\qquad \text{and}\qquad \mathscr{a}_{2}:\Gamma(L_{2}) \to \mathfrak{X}(P)
  \]
  by restricting the map $\mathscr{a}$ to $\Gamma(L_{1})\oplus \{ 0 \} $ and $\{ 0 \}\oplus\Gamma(L_{2}) $, respectively. Because $\mu=(\mu_{1},\mu_{2})$ is a forward Dirac map, any section $\xi =(\mathrm{d}\mu_{1}(v_{1}),\alpha)  \in \Gamma(L_{1}) $ satisfies $\mathscr{a}_{1}(\xi) \in \operatorname{ker}\mathrm{d}\mu_{2}$. Similarly,  we have $\mathscr{a}_{2}(\zeta) \in \operatorname{ker}\mathrm{d}\mu_{1}$ for any section $\zeta = (\mathrm{d}\mu_{2}(v_{2}),\alpha_{2})$ of $L_{2}$.

  From condition \ref{item:orthogonal} in Definition \ref{def:morita-equiv-dirac}, it follows that $\varpi(v_{1},v_{2}) = 0 $. Consequently, evaluating the $\eta$-Courant bracket $\llbracket -,- \rrbracket_{\eta}$ yields 
\begin{align*}
    \left\llbracket \left( v_{1},\varpi^{\flat}(v_{1}) \right),\left( v_{2},\varpi^{\flat}(v_{2}) \right) \right\rrbracket_{\eta}& = \left( \left[ v_{1},v_{2} \right], \mathscr{L}_{v_{1}}\varpi^{\flat}(v_{2}) - \iota_{v_{2}}\mathrm{d}\left( \varpi^{\flat}(v_{1}) \right)+ \iota_{v_{1}\wedge v_{2}}\eta \right)\\
                                                                                                                      & = \left( \left[ v_{1},v_{2} \right], \iota_{v_{1}}\mathrm{d}\iota_{v_{2}}\varpi- \iota_{v_{2}}\mathrm{d}\iota_{v_{1}}\varpi \right).
\end{align*}
Because $\mu_{1}$ and $\mu_{2}$ are forward Dirac maps, we have that $\iota_{v_{1}}\varpi = \mu_{1}^{*}\alpha_{1}$ and $\iota_{v_{2}}\varpi = - \mu_{2}^{*}\alpha_{2}$. Using the fact that $v_{1}\in \operatorname{ker}\mathrm{d}\mu_{2}$ and $v_{2}\in \operatorname{ker}\mathrm{d}\mu_{1}$, we observe that $ \iota_{v_{1}}\mathrm{d}\iota_{v_{2}}\varpi- \iota_{v_{2}}\mathrm{d}\iota_{v_{1}}\varpi = - \iota_{v_{1}}\mu^{*}_{2}\mathrm{d}\alpha_{2} - \iota_{v_{2}}\mu_{1}^{*}\mathrm{d}\alpha_{1}=0$. Since $L_{\varpi}$ is closed under the $\eta$-Courant bracket, it follows that  $\varpi^{\flat}\left( \left[ v_{1},v_{2} \right] \right)=0$, meaning $\left[ v_{1},v_{2} \right] \in  \operatorname{ker}\varpi$. Furthermore, the naturality of the Lie bracket of vector fields ensures that $\left[ v_{1},v_{2} \right]$ also lies in the kernels of both $\mathrm{d}\mu_{1}$ and $\mathrm{d}\mu_{2}$. Applying condition \ref{item:non-deg} in Definition \ref{def:morita-equiv-dirac}, we obtain
\[
  \left[ \mathscr{a}_{1}(\xi),\mathscr{a}_{2}(\zeta) \right]=0;
\]
that is, the actions $\mathscr{a}_{1}$ and $\mathscr{a}_{2}$ commute, as required by Definition \ref{def:morita-equiv-algebroid}.

Furthermore, the actions are injective as a consequence of the forward Dirac map condition: if $\mathscr{a}_{1}(\xi)=0$ then $\rho_{1}(\xi)=\mathrm{d}\mu_{1}(0)=0$ and $(\mathrm{d}\mu_{1})^{*}\alpha = \varpi^{\flat}(0)=0$. Since $\mathrm{d}\mu_{1}$ is surjective, $\alpha = 0$, and thus $\xi=0$. An analogous argument shows that $\mathscr{a}_{2}(\zeta)=0$ implies $\zeta=0$. Finally, because $\mu:P\to M_{1}\times M_{2}$ is a forward Dirac map, the identity $\mathscr{a}_{1}\left( \mu_{1}^{*}L_{1} \right)_{p} = \operatorname{ker}\mathrm{d}\mu_{2 \, p}$ holds for every $p \in P$. To see this, note that for any $v \in \operatorname{ker}\mathrm{d}\mu_{2\, p}$, there exists $\alpha \in T_{\mu_{1}(p)}^{*}M_{1}$ such that $\mu_{1}^{*}\xi\coloneqq (\mathrm{d}\mu_{1}(v),\alpha)\in L_{1, \mu_{1}(p)}$. Therefore, $\mathscr{a}_{1}\left( \mu^{*}\xi \right) = v$ as desired. An identical argument establishes that $\mathscr{a}_{2}\left( \mu_{2}^{*}L_{2} \right)_{p} = \operatorname{ker}\mathrm{d}\mu_{1\, p}$. Because the fibers are connected and simply-connected, and the actions complete, the proposition follows.
\end{proof}
 
\begin{definition}
  \label{def:morita-dirac-nijenhuis}
  Two Dirac-Nijenhuis structures $(M_{1},L_{1},r_{1})$ and $(M_{2},L_{2},r_{2})$ are \textbf{Morita equivalent Dirac-Nijenhuis structures} if there exists an infinitesimal Morita bibundle of Lie algebroids
\[
    (M_{1},L_{1},r_{1}) \xleftarrow{\quad \mu_{1}\quad} (P,\varpi,J) \xrightarrow{\quad \mu_{2}\quad } (M_{2},\bar{L}_{2},r_{2})
\]
equipped with a compatible pair $(\varpi,J)$ such that:
\begin{enumerate}[label=(\alph*)]
    \item $(P,\varpi)$ defines a Morita equivalence of the underlying Dirac structures, and
    \item $(P,J)$ defines a Morita equivalence of the underlying infinitesimal Nijenhuis structures.
\end{enumerate}
\end{definition}
\begin{example}[Morita equivalence of Poisson-Nijenhuis manifolds]
  \label{eg:poisson-nijenhuis-morita}
  Consider two Poisson-Nijenhuis manifolds $(M_{1},\pi_{1},r_{1})$ and $(M_{2},\pi_{2},r_{2})$ and let $L_{1}=\mathrm{Graph}(\pi_{1})$ and $L_{2}=\mathrm{Graph}(\pi_{2})$ be the respective Dirac structures they define. The actions that the Lie algebroids associated to $\pi_{1}$ and $\pi_{2}$ define on $P$ are given by
\[
    \mathscr{a}_{i}:\Gamma(L_i)\to \mathfrak{X}(P), \quad (\pi_i(\alpha),\alpha)\mapsto \pi_{\varpi}^{\sharp}(\mu_i^{*}(\alpha)),
\]
where $\pi_{\varpi}^{\sharp}$ is the negative inverse of $\varpi^{\flat}:TM\to T^{*}M$. Then $(M_{1},L_{1},r_{1})$ and $(M_{2},L_{2},r_{2})$ are Morita equivalent Dirac-Nijenhuis structures if and only if there is a Morita equivalence of Poisson manifolds $(M_1,\pi_{1})\xleftarrow{\ \mu_{1}\ } (P,\varpi) \xrightarrow{\ \mu_{2}\ }(M_{2},-\pi_{2})$ together with a Nijenhuis tensor field $J\in \Omega^{1}(P,TP)$ compatible with $\varpi$, satisfying the Morita equivalence of the corresponding 1-derivations.

We express the conditions of the infinitesimal Morita equivalence of the Nijenhuis structures explicitly. Recall that the 1-derivation on the Dirac structure $L_i$ is $\mathscr{D}_i = (\mathbb{D}^{r_i}|_{L_i}, (r_i, r_i^*), r_i)$ (see Corollary \ref{cor:dirac-compatibility-algebroid}), where $\mathbb{D}^{r_i} = (\nabla^{r_i}, \nabla^{r_i,*})$ (see Example \ref{eg:tangent-cotangent-lift}). Considering the definition of the actions $\mathscr{a}_i$, the conditions of Definition \ref{def:morita-equiv-1-derivations} yield:
\begin{equation} 
    \label{eq:poisson-nijenhuis-morita-conditions}
    \begin{cases}
        \mathrm{d}\mu_i \circ J = r_i \circ \mathrm{d}\mu_i, \\
        J \circ \pi_{\varpi}^{\sharp} \circ \mu_i^{*} = \pi_{\varpi}^{\sharp} \circ \mu_i^{*} \circ r_i^*, \\
        \nabla^J_{v}\left( \pi_{\varpi}^{\sharp}(\mu_1^{*}\alpha + \mu_2^{*}\beta) \right) = \pi_{\varpi}^{\sharp}\left( \mu_1^{*}\left( \nabla^{r_{1},*}_{\mathrm{d}\mu_1(v)}\alpha \right) + \mu_2^{*}\left( \nabla^{r_{2},*}_{\mathrm{d}\mu_2(v)}\beta \right) \right),
    \end{cases}
\end{equation}
for every $v\in TP$, $\alpha \in \Omega^1(M_1)$ and $\beta \in \Omega^1(M_2)$, with $i=1,2$.
\end{example}
\begin{example}[Morita equivalence of holomorphic Dirac structures]
  \label{eg:holomorphic-dirac-morita}
  Recall that a \textit{holomorphic Dirac structure} is a holomorphic subbundle $L\subseteq \mathbb{T}^{(1,0)}M \coloneqq T^{(1,0)}M\oplus (T^{(1,0)}M)^{*}$ that is Lagrangian with respect to a natural non-degenerate symmetric pairing, and such that its sheaf of holomorphic sections is closed under the bracket
\[
  \llbracket (u,\alpha),(v,\beta) \rrbracket = \left( \left[ u,v \right], \mathscr{L}_{u}\beta - \iota_{v}\partial\alpha\right).
\]
Holomorphic Dirac structures can be equivalently described in terms of the Nijenhuis structure $I_{M}$ of the complex manifold \cite[Proposition 3.16]{bursztyn-drummond-netto2022dirac}, using the map
\[
    \Phi: \mathbb{T}M \to \mathbb{T}^{(1,0)}M, \quad \Phi(v,\alpha) = \left( \frac{1}{2} (v - iI_{M}(v)),\alpha - i I_{M}^*(\alpha)\right).
\]
Via this identification, we can see that two holomorphic Dirac structures, $L_{1}\subseteq \mathbb{T}^{(1,0)}M_{1}$ and $L_{2}\subseteq \mathbb{T}^{(1,0)}M_{2}$, are Morita equivalent if and only if there is a holomorphic closed $2$-form $\varpi \in \Omega^{(2,0)}(P)$, and two holomorphic forward Dirac maps $\left( M_{1},L_{1} \right)  \xleftarrow{\ \mu_{1}\ } (P,\varpi) \xrightarrow{\ \mu_{2}\ }\left( M_{2},\bar{L}_{2} \right)$, satisfying (a) and (b) in Definition \ref{def:morita-equiv-dirac}, and the actions $\mathscr{a}_{i}:\Gamma(L_{i}) \to \mathfrak{X}(P)$ defined by $(\mathrm{d}\mu_{i}(v),\alpha)\mapsto v$ preserve the holomorphic structure. In particular, Morita equivalences of \textit{holomorphic Poisson structures} \cite{laurent2008holomorphic} are described in this way. 
\end{example}
The next theorem shows that a Morita equivalence of quasi-symplectic-Nijenhuis groupoids (Definition \ref{def:morita-equivalence-symplectic-nijenhuis}) differentiates under the Lie functor to a Morita equivalence of Dirac-Nijenhuis structures as defined above.
\begin{theorem}
  \label{thm:symp-nij-morita-dirac-nij}
  Let $(\mathscr{G},\omega_{1},\eta_{1},N_{1})$ and $(\mathscr{H},\omega_{2},\eta_{2},N_{2})$ be quasi-symplectic-Nijenhuis groupoids. Assume they are Morita equivalent via a Morita $(\mathscr{G},\mathscr{H})$-bibundle
  \[
    \begin{tikzcd}
      {(G_{1},\omega_1,N_{1})} \arrow[d, shift right] \arrow[d, shift left] & {(P,\varpi,J)} \arrow[ld, "\mu_{1}" '] \arrow[rd, "\mu_{2}"] & {(H_1,\omega_2,N_{2})} \arrow[d, shift right] \arrow[d, shift left] \\
      {(G_0,\eta_1,r_{1})} & & {(H_0,\eta_2,r_{2})}.
    \end{tikzcd}
  \]
  Then their associated twisted Dirac-Nijenhuis structures, $(L_{1},r_{1})$ and $(L_{2},r_{2})$, are Morita equivalent, with $(P,\varpi,J)$ serving as the infinitesimal Morita $(L_{1},L_{2})$-bibundle equipped with infinitesimal actions induced by the global actions.
\end{theorem}

\begin{proof}
    By Theorem \ref{thm:morita-nijenhuis-groupoid-to-algebroid}, the associated infinitesimal Nijenhuis structures are Morita equivalent. Furthermore, the compatibility of the pair $(\varpi,J)$ holds by hypothesis. Therefore, it remains only to establish the Morita equivalence of the underlying Dirac structures.

    The quasi-symplectic Morita equivalence yields an action $\mathscr{A}:K_{1}\times_{(\mathbf{s},\mu)}P\to P$ of the product quasi-symplectic groupoid $\mathscr{K}\coloneqq \mathscr{G}\times \mathscr{H}^{\mathrm{op}}$ on the submersion $\mu\coloneqq (\mu_{1}, \mu_{2})$. Here, the $2$-form on $\mathscr{K}$ is given by $\omega\coloneqq \mathrm{pr}_{G_{1}}^{*}\omega_{1}-\mathrm{pr}_{H_{1}}^{*}\omega_{2}$, and the action satisfies the conditions
    \[
        \operatorname{ker}\mathrm{d}\mu \cap \operatorname{ker}\varpi = 0 \qquad \text{and} \qquad \mathscr{A}^{*}\varpi= \mathrm{pr}_{P}^{*}\varpi + \mathrm{pr}_{K_{1}}^{*}\omega.
    \]
    Consequently, $P$ inherits the structure of a \textit{Hamiltonian $\mathscr{K}$-space}. Applying \cite[Theorem 4.7]{bursztyn-crainic2005dirac}, we conclude that $\mu:P\to G_{0}\times H_{0}$ is a forward Dirac map from $L_{\varpi}$ to the product Dirac structure $L_{1}\times \bar{L}_{2}$ induced by $(\mathscr{K},\omega,\mathrm{pr}_{1}^{*}\eta_{1} - \mathrm{pr}_{2}^{*}\eta_{2})$.

    Let $A_{\mathscr{G}}$ and $A_{\mathscr{H}}$ denote the Lie algebroids of $\mathscr{G}$ and $\mathscr{H}$, respectively. These Lie algebroids are isomorphic to their associated Dirac structures via the bundle maps $(\rho_{1},\sigma_{\omega_{1}}):A_{\mathscr{G}}\xrightarrow{\sim} L_{1}$ and $(\rho_{2},\sigma_{\omega_{2}}):A_{\mathscr{H}}\xrightarrow{\sim} L_{2}$. Because the underlying Lie groupoids are Morita equivalent, Lemma \ref{lemma:correspondence-morita} ensures that their corresponding Lie algebroids are also Morita equivalent. By Definition \ref{def:morita-equiv-algebroid}, the associated infinitesimal actions must satisfy $\mathscr{a}_{1}(\mu^{*}_{1}A_{\mathscr{G}})_{p} = \operatorname{ker}\mathrm{d}\mu_{2\, p}$ and $\mathscr{a}_{2}(\mu_{2}^{*}A_{\mathscr{H}})_{p}=\operatorname{ker}\mathrm{d}\mu_{1\, p}$  for every $p \in P$. Therefore, for any tangent vectors $u \in \operatorname{ker}\mathrm{d}\mu_{1\, p} $ and $v \in \operatorname{ker}\mathrm{d} \mu_{2\, p}$ there exist elements $\mu_{1}^{*} \xi \in \mu_{1}^{*}A_{\mathscr{G}}$ and $\mu_{2}^{*}\zeta \in \mu_{2}^{*} A_{\mathscr{H}}$ such that $\mathscr{a}_{2}(\mu_{2}^{*}\zeta)=u$ and $\mathscr{a}_{1}(\mu_{1}^{*}\xi) =v$. Because $\mu_{2}$ is a forward Dirac map, it follows that $(\mathrm{d}\mu_{2})^{*}(\sigma_{\omega_{2}}(\zeta)) = - \varpi^{\flat}(u)$. Evaluating $\varpi$ on the pair $(u,v)$ thus yields
    \[
      \varpi(u,v) = \varpi^{\flat}(u)(v) = - (\mathrm{d}\mu_{2})^{*} \left( \sigma_{\omega_{2}}(\zeta) \right) (v) = -\sigma_{\omega_{2}}(\zeta)(\mathrm{d}\mu_{2}(v)) = 0,
    \]
    where the final equality holds because $v\in \operatorname{ker}\mathrm{d}\mu_{2\, p}$. This confirms that the orthogonality condition \ref{item:orthogonal} of Definition \ref{def:morita-equiv-dirac} is satisfied. Because the remaining conditions hold by hypothesis, the Morita equivalence of the Dirac structures is established.
    
\end{proof}
By definition, a Morita equivalence of Dirac-Nijenhuis structures requires the actions of the associated Lie algebroids to be complete. Consequently, the Dirac manifolds are integrable, which allows us to establish the converse of Theorem \ref{thm:symp-nij-morita-dirac-nij}.
\begin{theorem}
  \label{thm:morita-diracnij-to-sympnij}
  Let $\mathscr{G} = (G_1 \rightrightarrows G_0)$ and $\mathscr{H} = (H_1 \rightrightarrows H_0)$ be source-simply connected Lie groupoids. Suppose that the base spaces $G_0$ and $H_0$ admit Dirac--Nijenhuis structures $(G_0, L_1, r_1)$ and $(H_0, L_2, r_2)$, respectively, where $L_1$ is $\eta_1$-twisted and $L_2$ is $\eta_2$-twisted. Assume further that these structures are Morita equivalent via an infinitesimal Morita $(L_1, L_2)$-bibundle $(P, \varpi, J)$:
  \[
    (G_0, L_1, r_1) \xleftarrow{\quad \mu_1 \quad} (P, \varpi, J) \xrightarrow{\quad \mu_2 \quad} (H_0, L_2, r_2).
  \]
  Then the corresponding quasi-symplectic--Nijenhuis groupoids $(\mathscr{G}, \omega_1, \eta_1, N_1)$ and $(\mathscr{H}, \omega_2, \eta_2, N_2)$ are Morita equivalent. In particular, $(P, \varpi, J)$ serves as the Morita $(\mathscr{G}, \mathscr{H})$-bibundle, equipped with global actions that integrate the infinitesimal actions.
\end{theorem}
\begin{proof}
    Theorem \ref{thm:morita-nijenhuis-algebroid-to-groupoid} guarantees that the underlying Morita equivalence of the infinitesimal Nijenhuis structures integrates to yield Morita equivalent Nijenhuis groupoids.
    
    Furthermore, the compatibility of the pair $(\varpi,J)$ holds by hypothesis. Because all other requisites for quasi-symplectic Morita equivalence are already established, it remains only to prove that the total infinitesimal action integrates to a global action $\mathscr{A}$ satisfying
    \[
      \mathscr{A}^{*}\varpi = \mathrm{pr}_{P}^{*}\varpi + \mathrm{pr}_{G_{1}}^{*}\omega_{1} + \mathrm{pr}_{H_{1}}^{*}\omega_{2}.
    \]
        Observe that the forward Dirac map $\mu:P\to G_{0}\times H_{0}$ from $L_{\varpi}$ to $L_{1}\times \bar{L}_{2}$ is a \textit{complete Dirac realization}; that is, the infinitesimal action of $L_{1}\times \bar{L}_{2}$ on $P$ is complete. To verify this, note that because the individual actions $\mathscr{a}_{1}$ and $\mathscr{a}_{2}$ are complete, the induced vector fields $\mathscr{a}_{1}(\xi)$ and $\mathscr{a}_{2}(\zeta)$ are complete for any compactly supported sections $\xi\in \Gamma(L_{1})$ and $\zeta\in \Gamma(L_{2})$. Since these vector fields commute, their sum $\mathscr{a}_{1}(\xi)+\mathscr{a}_{2}(\zeta)$ is also complete, which establishes the completeness of the total infinitesimal action.

        Applying \cite[Theorem 4.7]{bursztyn-crainic2005dirac}, this complete Dirac realization integrates to a global left action $\widetilde{\mathscr{A}}$ of the product groupoid $\mathscr{K}=\mathscr{G}\times \mathscr{H}^{\mathrm{op}}$ on the submersion $\mu=(\mu_{1},\mu_{2})$ satisfying $\widetilde{\mathscr{A}}^{*}\varpi = \mathrm{pr}_{P}^{*}\varpi + \mathrm{pr}_{K_{1}}^{*}\omega$, where $\omega \in \Omega^{2}(\mathscr{K})$ is the multiplicative $2$-form integrating $L_{1}\times \bar{L}_{2}$, given explicitly by $\omega = \mathrm{pr}_{1}^{*}\omega_{1} - \mathrm{pr}_{2}^{*}\omega_{2}$. Applying the projection identities $\mathrm{pr}_{G_{1}}\circ \mathrm{pr}_{K_{1}}=\mathrm{pr}_{G_{1}}$ and $\mathrm{pr}_{H_{1}}\circ \mathrm{pr}_{K_{1}}=\mathrm{pr}_{H_{1}}$, the pullback condition expands to $\widetilde{\mathscr{A}}^{*}\varpi = \mathrm{pr}_{P}^{*}\varpi + \mathrm{pr}_{G_{1}}^{*}\omega_{1} - \mathrm{pr}_{H_{1}}^{*}\omega_{2}$. 

            The existence of this left $\mathscr{K}$-action $\widetilde{\mathscr{A}}$ corresponds precisely to equipping $P$ with commuting left Hamiltonian $\mathscr{G}$- and right Hamiltonian $\mathscr{H}$-actions. Let $\mathscr{A}$ denote the combination of this left $\mathscr{G}$-action and right $\mathscr{H}$-action. By reversing the sign of $\omega_{2}$---as consequence of converting the left $\mathscr{H}^{\mathrm{op}}$-action into a right $\mathscr{H}$-action---the combined action satisfies the desired compatibility condition $\mathscr{A}^{*}\varpi = \mathrm{pr}_{P}^{*}\varpi + \mathrm{pr}_{G_{1}}^{*}\omega_{1} + \mathrm{pr}_{H_{1}}^{*}\omega_{2}$. This concludes the proof.
\end{proof}
\begin{corollary}
    \label{cor:PN-SN}
    There is a one-to-one correspondence between Morita equivalences of Poisson-Nijenhuis structures and those of source-simply connected symplectic-Nijenhuis groupoids. In particular, this restricts to a correspondence between Morita equivalences of holomorphic Poisson structures and source-simply connected holomorphic symplectic groupoids.
\end{corollary}

\subsection{Remarks on hierarchies of Morita equivalent Dirac-Nijenhuis structures}
 An important feature of Poisson-Nijenhuis manifolds is the fact that they produce a hierarchy of compatible geometric structures. In this subsection, we explore the compatibility between Morita equivalence and these hierarchies generated by Nijenhuis tensor fields. This is primarily theoretical: while we plan to find and study concrete examples of Morita equivalent hierarchies in subsequent work, demonstrating that under certain conditions the Morita equivalence persists through the hierarchy serves, at this point, as evidence for the soundness of the construction.

Recall that for a given Poisson-Nijenhuis manifold $(M,\pi,r)$, the hierarchy of compatible structures is a sequence of Poisson bivectors $\{ \pi^{(n)} \}_{n\in \mathbb{N}} $, defined recursively by
\begin{equation}
  \label{eq:hierarchy-poisson}
  \big(\pi^{(n)}\big)^{\sharp}\coloneqq r\circ \big(\pi^{(n-1)}\big)^{\sharp}= r^{n}\circ\pi^{\sharp}.
\end{equation}
This construction admits a generalization to the setting of Dirac-Nijenhuis structures \cite{bursztyn-drummond-netto2022dirac}. Let $(L,r)$ be a Dirac-Nijenhuis structure on a manifold $M$. For every $n\in \mathbb{N}$, define:
\begin{align}
  & L^{(n,0)}\coloneqq \left( r^{n},\mathrm{id}_{T^{*}M} \right)(L)\qquad \text{and} \label{hierarchy-dirac-1}\\
  & L^{(0,n)}\coloneqq \left( \mathrm{id}_{TM}, (r^{*})^{n} \right)(L). \label{hierarchy-dirac-2}
\end{align}
According to \cite[Proposition 4.8]{bursztyn-drummond-netto2022dirac}, if $\operatorname{ker}(r, \mathrm{id}_{T^{*}M})|_{L}=0$, then $(L^{(n,0)},r)$ is a Dirac-Nijenhuis structure for every $n\in \mathbb{N}$. Similarly, if $\operatorname{ker}(\mathrm{id}_{TM}, r^{*})|_{L}=0$ holds, then $(L^{(0,n)},r)$ is a Dirac-Nijenhuis structure.

We now turn to the Morita equivalence of these hierarchies. Note that given a Morita equivalence $(M_{1},L_{1},r_{1})\gets (P,\varpi,J)\to (M_{2},L_{2},r_{2})$, the Dirac structures $L_{\varpi}^{(0,n)}$ are graphs of $2$-forms; hence, they are natural candidates for bibundles between $L_{1}^{(0,n)}$ and $L_{2}^{(0,n)}$. In contrast, the structures $L_{\varpi}^{(n,0)}$ are not generally graphs of $2$-forms, so the hierarchies $L_{1}^{(n,0)}$ and $L_{2}^{(n,0)}$ given by \eqref{hierarchy-dirac-1} require separate analysis. Poisson-Nijenhuis structures yield hierarchies of the form \eqref{hierarchy-dirac-1}, so the investigation of Morita equivalence in this context requires extra refinement. Thus, we proceed first to study the Morita equivalence of Dirac structures defined by the recursion \eqref{hierarchy-dirac-2}. It turns out that, in this setting, the corresponding hierarchies are indeed Morita equivalent. Subsequently, we examine hierarchies of the form \eqref{hierarchy-dirac-1} specifically for the case of Poisson-Nijenhuis structures.
\begin{proposition}
  \label{prop:hierarchy-morita}
  Let $(M_{1},L_{1},r_{1})$ and $(M_{2},L_{2},r_{2})$ be Morita equivalent Dirac-Nijenhuis structures. Suppose that $\operatorname{ker}\left(\mathrm{id}_{TM_{i}}, r_{i}^{*}\right)|_{L_{i}}=0$ for $i=1,2$. Then the Dirac-Nijenhuis structures $\big(M_{1}, L_{1}^{(0,n)},r_{1}\big)$ and $\big(M_{2}, L_{2}^{(0,n)},r_{2}\big)$ are Morita equivalent for every $n\in \mathbb{N}$.
\end{proposition}
\begin{proof}
We claim that if $(M_{1},L_{1},r_{1})\xleftarrow{\ \mu_{1}\ } (P,\varpi,J)\xrightarrow{\ \mu_{2}\ } (M_{2},L_{2},r_{2})$ is a Morita equivalence of Dirac-Nijenhuis structures, then 
  \[
    \left(M_{1},L_{1}^{(0,n)},r_{1}\right)\xleftarrow{\quad \mu_{1}\quad }\big(P,\varpi^{(n)},J\big)\xrightarrow{\quad \mu_{2}\quad }\left(M_{2},L_{2}^{(0,n)},r_{2}\right)
  \]
  is a Morita equivalence of the Dirac-Nijenhuis structures $\big(M_{i},L_{i}^{(0,n)},r_{i}\big)$, where $\big(\varpi^{(n)}\big)^{\flat}\coloneqq \varpi^{\flat}\circ J^{n}$.

    We proceed by induction on $n$. For the base case $n=1$, we first show that the maps $\mu_{i}\colon\left( P,L_{\varpi^{(1)}} \right)\to \big(M_{i},L_{i}^{(0,1)}\big)$ are forward Dirac maps. Because $\mu_i \colon (P, L_\varpi) \to (M_i, L_i)$ is already a forward Dirac map, elements in the fiber $(L_i)_{\mu_i(p)}$ take the form $\zeta = (\mathrm{d}\mu_i(v), \alpha)$, where $v \in T_p P$ satisfies $\varpi^{\flat}(v) = (\mathrm{d}\mu_i)^*\alpha$. By definition, elements in $L_i^{(0,1)}$ are of the form $(\mathrm{d}\mu_i(v), r_i^*\alpha)$. Because $(\varpi,J)$ is a compatible pair, we have $\big(\varpi^{(1)}\big)^{\flat} = \varpi^{\flat} \circ J = J^* \varpi^{\flat}$. Furthermore, the assumption that $J$ is $\mu_i$-related to $r_i$ implies $J^{*}\circ \mathrm{d}\mu_i^{*} = \mathrm{d}\mu_i^{*}\circ r_{i}^{*}$. Consequently, we obtain
  \[
    \big(\varpi^{(1)}\big)^{\flat}(v) = J^*\varpi^{\flat}(v) = J^*(\mathrm{d}\mu_i)^*\alpha = (\mathrm{d}\mu_i)^*(r_i^*\alpha).
  \]
  This establishes that $\mu_i \colon (P, L_{\varpi^{(1)}}) \to \big(M_i, L_i^{(0,1)}\big)$ is a forward Dirac map.

    Second, we verify that the actions induced by these forward Dirac maps define an infinitesimal Morita bibundle. Denote by $\mathscr{a}_{i}\colon \Gamma(L_{i})\to \mathfrak{X}(P)$ the actions on $P$ induced by the Dirac structures $L_{i}$. For $i \in \{1,2\}$, the corresponding actions induced by the structures $L_{i}^{(0,1)}$, denoted by $\mathscr{a}_{i}^{(0,1)}\colon \Gamma\big(L_{i}^{(0,1)}\big)\to \mathfrak{X}(P)$, map sections of the form $\xi=(\mathrm{id}_{TM_{i}}, r_{i}^{*})(\zeta) \in \Gamma\big(L_{i}^{(0,1)}\big)$---where $\zeta=(\mathrm{d}\mu_{i}(v),\alpha) \in \Gamma(L_{i})$---to the vector field $v$. In other words, we have
  \begin{equation}
    \label{eq:actions-hierarchy-1}
    \mathscr{a}_{i}^{(0,1)}(\xi) = \mathscr{a}_{i}(\zeta).
  \end{equation}
  
  Since the images of these actions coincide exactly with the vector fields of the original infinitesimal bibundle, it follows immediately that they satisfy all conditions for the Morita equivalence of both Dirac structures (Definition \ref{def:morita-equiv-dirac}) and infinitesimal Nijenhuis structures (Definition \ref{def:morita-equivalence-linear-tensor}). Furthermore, because the $2$-form $\varpi^{(1)}$ is compatible with $J$, we have successfully constructed a Dirac--Nijenhuis Morita equivalence for the base case $n=1$.

    Now, assume inductively that the result holds for a given integer $k \ge 1$; namely, that the structures $\big(M_{1}, L_{1}^{(0,k)},r_{1}\big)$ and $\big(M_{2}, L_{2}^{(0,k)},r_{2}\big)$ are Morita equivalent via the infinitesimal bibundle $\big(P,\varpi^{(k)},J\big)$. Applying a computation identical to that of the base case, we obtain
  \[
    \big(\varpi^{(k+1)}\big)^{\flat}(v) = (\mathrm{d}\mu_{i})^{*}\big((r_{i}^{*})^{k+1}\alpha\big)
  \]
  for $i \in \{1,2\}$. This identity implies that the maps $\mu_{i}\colon \left( P,L_{\varpi^{(k+1)}} \right)\to \big( M_{i},L_{i}^{(0,k+1)} \big)$ are forward Dirac maps, and that the actions of the Dirac structures $L_{i}^{(0,k+1)}$ satisfy
  \begin{equation}
    \label{eq:actions-hierarchy-k1}
    \mathscr{a}_{i}^{(0,k+1)}\big( (\mathrm{id}_{TM_{i}}, (r_{i}^{*})^{k+1})(\zeta) \big) = \mathscr{a}_{i}(\zeta).
  \end{equation}
  Because the images of these actions coincide with those of the original structures, and since $\varpi^{(k+1)}$ remains compatible with $J$ by construction, it follows that $M_{1}\xleftarrow{\ \mu_{1}\ }P\xrightarrow{\ \mu_{2}\ }M_{2}$ defines a Morita equivalence for the case $k+1$. This completes the induction and finishes the proof.
\end{proof}

\begin{proposition}
  \label{prop:hierarchy-poisson-nij}
  Let $(M_{1},\pi_{1},r_{1})$ and $(M_{2},\pi_{2},r_{2})$ be Morita equivalent Poisson--Nijenhuis structures via a bibundle
  \[
    (M_{1},\pi_{1},r_{1})\xleftarrow{\quad \mu_{1}\quad } (P,\varpi, J)\xrightarrow{\quad \mu_{2}\quad } (M_{2},\pi_{2},r_{2}).
  \]
  If $J \in \Omega^{1}(P,TP)$ is an invertible Nijenhuis $(1,1)$-tensor field, then the Poisson--Nijenhuis structures $(M_{1},\pi_{1}^{(n)},r_1)$ and $(M_{2},\pi_{2}^{(n)},r_2)$ are (weakly) Morita equivalent for all $n \in \mathbb{N}$. 
\end{proposition}
\begin{remark}
  \label{rmk:completeness_assumption}
  To obtain a (strong) Morita equivalence, this proposition requires the assumption that the infinitesimal actions associated with the hierarchy are complete. These actions are defined by deforming the original actions $\mathscr{a}_{1}$ and $\mathscr{a}_{2}$ with powers of $J$ (as detailed in the construction below) and are not automatically complete. Without this completeness assumption, the structures are only weakly Morita equivalent in the sense of Ginzburg \cite{ginzburg2001grothendieck}.
\end{remark}
 
We outline the construction of the hierarchy of bibundles; the proof proceeds analogously to Proposition \ref{prop:hierarchy-morita}. 

Since $(\varpi,J)$ is a compatible pair and $\varpi$ is symplectic, $(P,\pi_{\varpi},J)$ is a Poisson-Nijenhuis manifold, where $\pi_{\varpi} = - (\varpi^{\flat})^{-1}$. Consequently, there exists a hierarchy of Poisson-Nijenhuis structures $(\pi_{\varpi}^{(n)},J)$ on $P$. Note that each $\pi_{\varpi}^{(n)}$ is a non-degenerate Poisson structure. Indeed, $J^n$ is invertible since $J$ is invertible; furthermore,  $\pi_{\varpi}^{\sharp}$ is an isomorphism by the non-degeneracy of $\varpi$. It follows that the composition $\big(\pi_{\varpi}^{(n)}\big)^{\sharp} = J^n \circ \pi_{\varpi}^{\sharp}$ is invertible. We define the symplectic forms $\varpi^{(n)} \in \Omega^{2}(P)$ by $\big(\varpi^{(n)}\big)^{\flat}\coloneqq -\big(\pi_{\varpi}^{(n)\sharp}\big)^{-1}$, i.e. $\big(\varpi^{(n)}\big)^{\flat} = \varpi^{\flat} \circ J^{-n}$. Finally, the bibundle for the $n$-th level of the hierarchy is given by the bibundle
\[
  \left(M_{1},\pi_{1}^{(n)},r_{1}\right) \xleftarrow{\quad \mu_{1}\quad } \left(P,\varpi^{(n)},J\right) \xrightarrow{\quad \mu_{2}\quad } \left(M_{2},\pi_{2}^{(n)},r_{2}\right).
\]
The infinitesimal actions are defined by deforming the original actions via the tensor field $J^{n}$, which yields the natural actions on the corresponding symplectic realizations of the Poisson manifolds. Let $\mathscr{a}^{(n)}_{i}\colon \Gamma (T^{*}M_{i})\to \mathfrak{X}(P)$ denote the actions of the cotangent algebroids associated to the Poisson manifolds $\big(M_{i},\pi_{i}^{(n)}\big)$. For any $\alpha\in \Omega^{1}(M_{1})$ and $\beta\in \Omega^{1}(M_{2})$, these actions are explicitly given by
\begin{equation}
    \label{eq:actions_poisson_hierarchy}
\mathscr{a}^{(n)}_{1} (\alpha)= \big(\pi_{\varpi}^{(n)}\big)^{\sharp}(\mu_{1}^{*}\alpha) \qquad \text{and}\qquad \mathscr{a}^{(n)}_{2}(\beta) = \big(\pi^{(n)}_{\varpi}\big)^{\sharp}(\mu_{2}^{*}\beta).
\end{equation}
Provided the completeness of these infinitesimal actions, establishing Morita equivalence at the $n$-th level reduces to verifying the remaining conditions for an infinitesimal Morita bibundle. We now proceed to verify these conditions.

The commutativity of the actions $\mathscr{a}_{1}^{(n)}$ and $\mathscr{a}_{2}^{(n)}$ follows directly from the compatibility of the pair $(\pi_{\varpi}, J)$ and the assumption that $J$ is $\mu_{i}$-related to $r_{i}$ for each $i \in \{1, 2\}$. Specifically, we compute
\begin{align*}
    \left[\mathscr{a}^{(n)}_{1}(\alpha), \mathscr{a}^{(n)}_{2}(\beta) \right] & = \left[ J^{n}\big(\pi_{\varpi}^{\sharp}(\mu_{1}^{*}\alpha)\big), J^{n}\big(\pi_{\varpi}^{\sharp}(\mu_{2}^{*}\beta)\big)\right]\\
                                                                              & = \left[ \pi_{\varpi}^{\sharp}\left( \mu_{1}^{*}\left( (r_{1}^{*})^{n}\alpha \right) \right),  \pi_{\varpi}^{\sharp}\left( \mu_{2}^{*}\left( (r_{2}^{*})^{n}\beta \right) \right)\right]\\
                                                                              & = 0.
\end{align*}
Finally, the non-degeneracy of $J$ guarantees that for every $p\in P$, the linear maps $\mathscr{a}_{1\, p}^{(n)}\colon (\mu_{1}^{*}(T^{*}M_{1}))_{p}\to T_{p}P$ and $\mathscr{a}_{2\, p}^{(n)}\colon (\mu_{2}^{*}(T^{*}M_{2}))_{p}\to T_{p}P$ map injectively onto $\operatorname{ker}\mathrm{d}\mu_{2\, p}$ and $\operatorname{ker}\mathrm{d}\mu_{1\, p}$, respectively. 
\begin{remark}[Non-degenerate infinitesimal Morita bibundle]
\label{rmk:morita-poisson-nijenhuis-convenient}
The framework of Proposition \ref{prop:hierarchy-poisson-nij} is most naturally framed in terms of an infinitesimal bibundle endowed with a non-degenerate Poisson bivector and an invertible, compatible Nijenhuis tensor field. Specifically, the infinitesimal Morita equivalence of Poisson--Nijenhuis manifolds is given by
\[
    (M_{1},\pi_{1},r_{1})\xleftarrow{\quad \mu_{1}\quad} (P,\pi_{\varpi},J)\xrightarrow{\quad \mu_{2}\quad } (M_{2},\pi_{2},r_{2}).
\]
In this setting, the triple $(r_{1}, J, r_{2})$ generates the hierarchy of Morita equivalences of Poisson manifolds, yielding the following diagram:
\[
\begin{tikzcd}[column sep=huge]
\left(M_{1},\pi_{1}\right) \arrow[r, "r_{1}"] & \left(M_{1},\pi_{1}^{(1)}\right) \arrow[r, "r_{1}"] & \cdots \arrow[r, "r_{1}"] & \left(M_{1},\pi_{1}^{(n)}\right) \arrow[r, "r_{1}"] & \cdots \\
\left(P,\pi_{\varpi}\right) \arrow[u, "\mu_{1}"] \arrow[d, "\mu_{2}"'] \arrow[r, "J"] & \left(P,\pi_{\varpi}^{(1)}\right) \arrow[u, "\mu_{1}"] \arrow[d, "\mu_{2}"'] \arrow[r, "J"] & \cdots \arrow[r, "J"] & \left(P,\pi_{\varpi}^{(n)}\right) \arrow[u, "\mu_{1}"] \arrow[d, "\mu_{2}"'] \arrow[r, "J"] & \cdots \\
\left(M_{2},\pi_{2}\right) \arrow[r, "r_{2}"] & \left(M_{2},\pi_{2}^{(1)}\right) \arrow[r, "r_{2}"] & \cdots \arrow[r, "r_{2}"] & \left(M_{2},\pi_{2}^{(n)}\right) \arrow[r, "r_{2}"] & \cdots
\end{tikzcd}
\]
We call an infinitesimal bibundle with this structure a \textit{non-degenerate Morita equivalence} of Poisson--Nijenhuis manifolds.
\end{remark}

\subsubsection{Modular vector field of a Poisson-Nijenhuis manifold.}
We now comment on the cohomological consequences arising from the hierarchy of Morita equivalent Poisson--Nijenhuis structures. Henceforth, we consider a non-degenerate infinitesimal Morita bibundle between Poisson--Nijenhuis structures $(M_{1},\pi_{1},r_{1})$ and $(M_{2},\pi_{2},r_{2})$. By Proposition \ref{prop:hierarchy-poisson-nij}, there exists a corresponding hierarchy of infinitesimal bibundles
\[
\left(M_{1},\pi_{1}^{(n)},r_{1}\right) \xleftarrow{\quad \mu_{1}\quad } \left(P,\pi_{\varpi}^{(n)},J\right) \xrightarrow{\quad \mu_{2}\quad } \left(M_{2},\pi_{2}^{(n)},r_{2}\right),
\]
which constitute weak Morita equivalences (as discussed in Remark \ref{rmk:completeness_assumption}). The results of the present subsection remain valid in this weaker setting.

Every Poisson manifold $(M, \pi)$ induces a differential $\mathrm{d}_{\pi}\colon \mathfrak{X}^{k}(M)\to \mathfrak{X}^{k+1}(M)$. The cohomology $H^{\bullet}_{\pi}(M)$ of the complex $(\mathfrak{X}^{\bullet}(M),\mathrm{d}_{\pi})$ is called the \textit{Poisson cohomology} of $(M,\pi)$. It was shown in \cite{ginzburg-lu1992poisson} that Morita equivalent Poisson manifolds $(M_{1},\pi_{1})$ and $(M_{2},\pi_{2})$ have isomorphic first Poisson cohomology groups; that is, $H^{1}_{\pi_{1}}(M_{1}) \cong H^{1}_{\pi_{2}}(M_{2})$. In our setting, we thus obtain a hierarchy of isomorphic first Poisson cohomology groups. To simplify notation, we write $H^{1}_{(n)}(M_{i})\coloneqq H^{1}_{\pi_{i}^{(n)}}(M_{i})$. For every $n\in \mathbb{N}$, it then follows that
\[
  H^{1}_{(n)}(M_{1})\cong H^{1}_{(n)}(M_{2}).
\]

We recall the \textit{modular vector field} of a Poisson manifold $(M, \pi)$. Suppose initially that $M$ is orientable. The divergence of a vector field $v\in \mathfrak{X}(M)$ with respect to a volume form $\nu$ is the unique function $\mathrm{div}_{\nu}(v)\in C^{\infty}(M)$ such that $\mathscr{L}_{v}(\nu) = \mathrm{div}_{\nu}(v)\nu$. For any function $f\in C^{\infty}(M)$, let $X_{f}=\pi^{\sharp}(\mathrm{d}f)$ denote its Hamiltonian vector field. The \textbf{modular vector field} of $(M,\pi)$ with respect to $\nu$ is the Poisson vector field $\mathscr{X}^{\nu} \in \mathfrak{X}(M)$ defined by
\[
  \mathscr{X}_{\nu}(f) \coloneqq \mathrm{div}_{\nu}\left( X_{f} \right).
\]
While this vector field depends on the choice of the volume form, its cohomology class $\left[ \mathscr{X}_{\nu} \right]\in  H^{1}_{\pi}(M)$ does not. We therefore write $[\mathscr{X}]\coloneqq [\mathscr{X}_{\nu}]$, and this class is called the \textbf{modular class} of the Poisson manifold $(M,\pi)$. By replacing volume forms with smooth densities in the definition of the modular vector field, we may drop the orientability requirement on the manifold $M$.  

Under Morita equivalence, these classes correspond via the isomorphism of the first Poisson cohomology groups \cite[Corollary 7]{crainic2003differentiable}. Therefore, every non-degenerate infinitesimal Morita equivalence between Poisson--Nijenhuis manifolds $(M_{1},\pi_{1},r_{1})$ and $(M_{2},\pi_{2},r_{2})$ induces an isomorphism $H^{1}_{(n)}(M_{1}) \cong H^{1}_{(n)}(M_{2})$ for every $n\in \mathbb{N}$. This isomorphism maps the modular class $\big[ \mathscr{X}_{1}^{(n)} \big]$ of $\big( M_{1},\pi_{1}^{(n)} \big)$ to the modular class $\big[ \mathscr{X}_{2}^{(n)} \big]$ of $\big( M_{2},\pi_{2}^{(n)} \big)$.

Damianou and Fernandes \cite{damianou-fernandes2008integrable} introduce a different notion of \textit{modular vector field}, which is intrinsically associated with a Poisson-Nijenhuis manifold $(M, \pi, r)$. This modular vector field, denoted by $\mathscr{Y}_{r}$, is a Poisson vector field relative to the second Poisson structure of the hierarchy. When its associated Poisson cohomology class vanishes, $\mathscr{Y}_{r}$ becomes a bi-Hamiltonian vector field. The \textbf{modular vector field of the Poisson-Nijenhuis structure} is defined using the modular vector fields of $\pi^{(1)}$ and $\pi^{(0)}=\pi$ with respect to an arbitrary volume form $\nu$; specifically, $\mathscr{Y}_r\in \mathfrak{X}(M)$ is given by
\begin{equation}
    \label{eq:modular_PN}
    \mathscr{Y}_{r} = \mathscr{X}^{(1)}_{\nu}- r \left( \mathscr{X}^{(0)}_{\nu} \right).
\end{equation}
This vector field is independent of the choice of volume form $\nu$. In analogy with \cite[Corollary 7]{crainic2003differentiable}, we obtain a correspondence for the modular vector fields of Morita equivalent Poisson-Nijenhuis manifolds.
\begin{theorem}
  \label{thm:modular-vectors}
  Let $(M_{1},\pi_{1},r_{1})$ and $(M_{2},\pi_{2},r_{2})$ be Poisson-Nijenhuis manifolds that are Morita equivalent via a non-degenerate infinitesimal Morita bibundle $(P,\pi_{\varpi},J)$. Then their modular classes $[\mathscr{Y}_{r_{1}}]$ and $[\mathscr{Y}_{r_{2}}]$ correspond under the isomorphism of the first Poisson cohomology groups $H^{1}_{(1)}(M_{1}) \cong H^{1}_{(1)}(M_{2})$.
\end{theorem}
To prove this theorem, we briefly recall the construction of the isomorphism $H_{\pi_1}^1(M_1) \to H_{\pi_2}^1(M_2)$ following Ginzburg and Lu \cite{ginzburg-lu1992poisson}. Any Poisson vector field $v_1 \in \mathfrak{X}(M_1)$ lifts to a Hamiltonian vector field $X_H$ on the symplectic manifold $P$ that is $\mu_{1}$-related to $v_{1}$. This lift projects to a well-defined Poisson vector field $v_2 = \mathrm{d}\mu_{2}( X_H)$ on $M_2$, whose resulting cohomology class $[v_2] \in H_{\pi_2}^1(M_2)$ is independent of the choice of $H$. Furthermore, if $v_1 = X_h$ is a Hamiltonian vector field on $M_1$, we may choose the explicit lift $H = \mu_1^*h$; since the associated vector field $X_{\mu_1^*h}$ is tangent to the $\mu_2$-fibers, it projects to zero under $\mathrm{d}\mu_{2}$. Consequently, the assignment $v_1 \mapsto [v_2]$ maps exact elements to zero, thereby descending to a well-defined homomorphism in cohomology $H_{\pi_1}^1(M_1) \to H_{\pi_2}^1(M_2)$. By symmetry, the analogous construction from $M_2$ to $M_1$ yields the inverse map, establishing the isomorphism $H_{\pi_1}^1(M_1) \cong H_{\pi_2}^1(M_2)$.

Following Crainic \cite{crainic2003differentiable}, applying this construction to the modular vector field $v_1 = \mathscr{X}_{1}$ of $(M_1,\pi_{1})$ yields the projected vector field $v_2 = \mathscr{X}_{2} + X_g$ on $M_2$, where $\mathscr{X}_{2}$ is the modular vector field of $(M_2,\pi_{2})$ and $X_g$ is a Hamiltonian vector field. Consequently, the isomorphism induced by the Hamiltonian lift $X_H$ identifies the corresponding modular cohomology classes, mapping $[\mathscr{X}_{1}]$ to $[\mathscr{X}_{2}]$.

\begin{proof}[Proof of Theorem \ref{thm:modular-vectors}]
    Suppose that the functions $H_0, H_1 \in C^{\infty}(P)$ generate Hamiltonian vector fields $X_{H_{0}}$ and $X_{H_{1}}$ that lift the modular vector fields $\mathscr{X}_{1}^{(0)}$ and $\mathscr{X}_{1}^{(1)}$ on the Poisson manifolds $\big(M_{1},\pi_{1}^{(0)}\big)$ and $\big(M_{1},\pi_{1}^{(1)}\big)$, respectively. We assume that $X_{H_{0}}$ and $X_{H_{1}}$ project via the differential $\mathrm{d}\mu_{2}$ to $\mathscr{X}_{2}^{(0)}+X_{g_{0}}$ and $\mathscr{X}_{2}^{(1)}+X_{g_{1}}$, respectively, for some smooth functions $g_{0},g_{1}\in C^{\infty}(M_{2})$.

    Since $\mathscr{Y}_{r_{1}}$ and $\mathscr{Y}_{r_{2}}$ are Poisson vector fields on $\big(M_{1},\pi^{(1)}_{1}\big)$ and $\big(M_{2},\pi_{2}^{(1)}\big)$, respectively, relating their cohomology classes requires an intermediate Hamiltonian vector field on $P$ with respect to the symplectic structure $\big(\pi_{\varpi}^{(1)}\big)^{\sharp} \coloneqq J\circ \pi_{\varpi}^{\sharp}$. Observe that the vector field $X_{H_{1}}- J(X_{H_{0}}) \in \mathfrak{X}(P)$ is Hamiltonian on the non-degenerate Poisson manifold $(P,\pi_{\varpi}^{(1)})$ with Hamiltonian function $H \coloneqq H_{1}-H_{0}$. Indeed, we have
\[
    X_{H} = \big(\pi_{\varpi}^{(1)}\big)^{\sharp}(\mathrm{d}(H_{1}-H_{0})) = \big(\pi_{\varpi}^{(1)}\big)^{\sharp}(\mathrm{d}H_{1}) - J\circ(\pi_{\varpi})^{\sharp}(\mathrm{d}H_{0}) = X_{H_{1}}-J(X_{H_{0}}).
\]
We claim that $X_{H}$ projects under $\mathrm{d}\mu_{2}$ to a vector field in the same cohomology class as $\mathscr{Y}_{r_{2}}$. Recalling that $J$ is $\mu_{2}$-related to $r_{2}$, we compute
\begin{align*}
    \mathrm{d}\mu_{2}(X_{H}) &= \mathrm{d}\mu_{2}(X_{H_{1}}) - \mathrm{d}\mu_{2}(J(X_{H_{0}})) \\
                             &= \mathscr{X}_{2}^{(1)} + X_{g_{1}} - r_{2}(\mathrm{d}\mu_{2}(X_{H_{0}})) \\
                             &= \mathscr{X}_{2}^{(1)} + X_{g_{1}} - r_{2}\big( \mathscr{X}_{2}^{(0)} + X_{g_{0}}\big) \\
                             &= \mathscr{X}_{2}^{(1)} - r_{2}\big(\mathscr{X}_{2}^{(0)}\big) + X_{g_{1}} - r_{2}(X_{g_{0}}).
\end{align*}
    We observe that $X_{g_{1}}-r_{2}(X_{g_{0}})$ is a Hamiltonian vector field on $\big(M_{2},\pi_{2}^{(1)}\big)$, since $X_{g_{1}}-r_{2}(X_{g_{0}}) = (\pi_{2}^{(1)})^{\sharp}(\mathrm{d}(g_{1}-g_{0}))$. Consequently, the cohomology class of $\mathrm{d}\mu_{2}(X_{H})$ in $H_{(1)}^{1}(M_{2})$ equals that of $\mathscr{X}_{2}^{(1)} - r_{2}\big(\mathscr{Y}_{2}^{(0)}\big)$; that is, $[\mathrm{d}\mu_{2}(X_{H})] = \big[\mathscr{Y}_{r_{2}}\big]$.   
\end{proof}
We remark that the above correspondence extends to the full hierarchy. That is, for every $n\in \mathbb{N}$, the hierarchy of cohomology classes for the modular vector fields $\mathscr{Y}^{(n)}_{r_{1}}$ and $\mathscr{Y}^{(n)}_{r_{2}}$ is in correspondence: the isomorphism $H^{1}_{(n+1)}(M_{1}) \cong H^{1}_{(n+1)}(M_{2})$ maps the modular class $\big[ \mathscr{Y}^{(n)}_{r_{1}} \big]$  to the modular class $\big[ \mathscr{Y}_{r_{2}}^{(n)} \big]$ .

\printbibliography
\end{document}